\documentclass[12pt]{article}
\usepackage{MyStyleSheet}
\usepackage{hyperref}
\usepackage{amsfonts}
\usepackage{amssymb}
\usepackage{amsmath}
\usepackage{blindtext}
\usepackage{geometry}
\usepackage{bm}
\usepackage{quiver}

\DeclareMathOperator{\per}{\text{per}}
\DeclareMathOperator{\gen}{\text{child}}

\DeclareMathOperator{\send}{\text{end}}

\DeclareMathOperator{\ts}{\text{ts}}
\DeclareMathOperator{\width}{\text{width}}

\DeclareMathOperator{\length}{\text{length}}
\usepackage{caption}
\usepackage{theoremref}

\title{Gliders on the Stranded Cellular Automata Model}
\author{Alexa Renner\footnote{renneram@rose-hulman.edu; Department of Mathematics, Rose-Hulman Institute of Technology, 5500 Wabash Ave., Terre Haute, IN 47803, USA.}}
\date{January 2026}

\newtheorem{counter}{counter}[section]

\theoremstyle{definition}
\newtheorem{definition}[counter]{Definition}
\newtheorem*{definition*}{Definition}

\theoremstyle{plain}
\newtheorem{theorem}[counter]{Theorem}
\newtheorem{lemma}[counter]{Lemma}
\newtheorem{corollary}[counter]{Corollary}
\newtheorem{proposition}[counter]{Proposition}

\theoremstyle{remark}

\begin{document}
\maketitle
\begin{abstract}
    The Stranded Cellular Automata (SCA) model consists of a grid of cells which can each contain between zero and two strands apiece and two turning rules that control when strands turn and when they cross. While patterns on this model have been studied previously, such research has not needed an algebraic description of the model. We provide a formal algebraic definition of patterns on the model, define gliders on the model in a way which is semi-compatible with definitions of gliders in other cellular automata models, and classify all 1- and 2-stranded gliders on this model. In addition, we prove an equivalence of two classes of gliders and design an algorithm to generate all such elements of that class.
\end{abstract}

\hfill\break

\hfill\break
\tableofcontents

\section{Introduction}
 Cellular automata have been used to model everything from computation to natural phenomena to fiber arts \cite{mathBasis}, \cite{hh}. Cellular automata are discrete models of a system evolving over time. They contain cells, which can each be in any one of a finite number of states, and these cells are arranged into generations \cite{hh}. These generations are arranged into some structure, in our case, a grid. 
 
 The Stranded Cellular Automata (SCA) model, introduced in \cite{hh}, is a model consisting of a grid of cells which can contain between zero and two strands apiece, a set of initial conditions on that grid, and two cellular automata that determine how those strands will evolve in time. The first controls whether a strand in a given cell will turn or not in the cell immediately adjacent to and above it (the turning rule), and the second controls which strand will cross over the other strand when two strands cross (the crossing rule). This model was originally designed to model fiber arts \cite{hh}.

 The first step of our project will be to describe patterns on a grid of cells, like the one below, algebraically to remove any ambiguity over what is meant by a ``strand" or a ``generation" and to provide a framework in which to prove results rigorously. 
 \begin{center}
    \includegraphics[scale = 1]{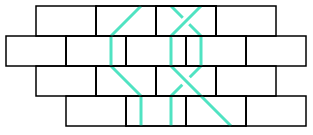}
\end{center}
Intuitively, we refer to the cyan lines as ``strands" and each horizontal row as a ``generation" so we can see the woven pattern evolving over several generations. We will make this precise in \ref{sec:Construction}.

   The SCA model has been used to model friendship bracelet weaving patterns, braids, and to analyze the complexity of woven patterns. In \cite{chan}, Chan examined individual braids made from cables woven together, and represented specific braids by using different turning and crossing rules for different positions on the grid and produced a software for generating SCA patterns given a turning rule, crossing rule, and initial conditions. In \cite{yang}, Yang examined the complexity of woven patterns on a cylindrical grid; and in \cite{loyd}, Loyd modeled friendship bracelet patterns, knots, and investigated the reversibility of rules on the SCA model.
 
 Intuitively, a glider is a pattern of cells in a cellular automaton that repeats indefinitely under some rule and moves in some direction across the landscape. While there has been much interest in gliders, especially in the context of Conway's Game of Life, the research done on them has been largely recreational \cite{informalGliders}, although there has been some more serious research \cite{goucher}, \cite{martinez}. We give an algebraic construction of the SCA model in Section \ref{sec:Construction} and provide background on SCA patterns in Section \ref{sec:SCAFacts}. We rigorously define gliders on the SCA model along with related concepts in Section \ref{sec:Gliders}. In Sections \ref{sec:OneStrand} and \ref{sec:TwoStrands}, we give a  complete classification of all 1- and 2-stranded gliders on the SCA model, respectively. We prove a general result about certain classes of gliders in Section \ref{sec:Pure}, and we generalize the method of classification used in Section \ref{sec:TwoStrands} to classify a certain class of gliders in Section \ref{sec:decidabilityAndPure}.
\subsection{Notation}
In this subsection, we collect some notation and conventions which we will use frequently:
\begin{itemize}
    \item We use the convention $0\in\mathbb{N}$.
    \item $\mathbb{N}_{\geq k}:$ Set of all natural numbers greater than or equal to $k$.
    \item $H_n:$ Set of all grid patterns on $n$ strands.
    \item $G_n:$ Set of all continuous grid patterns on $n$ strands.
    \item $\delta_i:$ Generation $i$ of some grid pattern.
    \item $\gen(A, B):$ Refers to the child of adjacent cells $A, B$.
    \item $\per(g):$ Refers to the period or the shortest repeating part of $g$.
    \item $\length(g):$ The number of generations in a grid pattern $g$ if $g$ is finite, $\infty$ otherwise.
    \item $\width(\delta_i):$ The number of cells between the first cell that contains a strand in $\delta_i$ and the last strand that contains a strand in $\delta_i$, inclusive.
    \item $a\circ b:$ $a$ composed with $b$.
    \item $a^\infty:$ The grid pattern $a$ composed with itself infinitely many times.
    \item $\operatorname{Speed}(g):$ For $g$ a pattern with $\per(g)$ defined, the number of indices $g$ moves to the left from the first generation to the last generation of $\per(g)$ over $\length(\per(g))$. This is treated as a symbol, not a number.
    \item $\operatorname{type}(x_j^{(k)}) = x$ where $x = n, s, l, $ or $r$.
    \item A cell is of the form $[x_j^{(k)}, y_{j+1}^{(r)}]$ or $\overline{[r_j^{(k)}, l_{j+1}^{(r)}]}$
    \begin{itemize}
        \item A cell of the form $[r_j^{(k)}, l_{j+1}^{(r)}]$ denotes a crossing where the rightmost strand in the cell goes over the leftmost strand.
        \item A cell of the form $\overline{[r_j^{(k)}, l_{j+1}^{(r)}]}$ denotes a crossing where the leftmost strand in the cell goes over the rightmost strand.
    \end{itemize}
    \item A generation is a concatenation of cells, it must be of the form $[w_j^{(k)}, x_{j+1}^{(r)}][y_{j+2}^{(h)}, z_{j+3}^{(m)}]\dots$, possibly also including cells of the form $\overline{[r_j^{(k)}, l_{j+1}^{(r)}]}$.
    \item $\frac{X}{YZ}$: A bit in a turning rule where $YZ$ describes the initial condition and $X$ describes the result.
    \item $XX0XX1XX0:$ A generic turning (crossing) rule with bit 2 0, bit 5 1, bit 8 0. This represents all turning (crossing) rules with bit 2 0, bit 5 1, bit 8 0.
    \item $g_k^l:$ The $k-$stranded left subpattern of $g$.
    \item $g_k^r:$ The $k-$stranded right subpattern of $g$.
\end{itemize}
\section{Grid Patterns}\label{sec:Construction}
\subsection{Construction}
We adopt the convention $0\in\mathbb{N}$. Let $\mathcal{S}=\{s, r, l, n\}$ be a set of symbols. Fix $n\in\mathbb{N}^+$. We define $\mathcal{C}$ to be the set consisting of the following eight elements, which we call \textbf{cell types}.
\begin{itemize}
    \item $[n, n]$, which can be visualized as 
    \begin{center}
    \includegraphics[scale = 1]{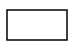}
\end{center}
    \item $[r, l]$, which represents
    \begin{center}
    \includegraphics[scale = 1]{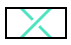}
\end{center}
    \item $\overline{[r,l]}$, which represents \begin{center}
    \includegraphics[scale = 1]{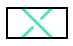}
\end{center}
    \item $[s, n]$, which we visualize as
    \begin{center}
    \includegraphics[scale = 1]{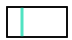}
\end{center}
    \item $[n, s]$, which represents
     \begin{center}
    \includegraphics[scale = 1]{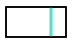}
\end{center}
    \item $[s,s]$, which we use to represent
     \begin{center}
    \includegraphics[scale = 1]{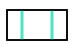}
\end{center}
    \item $[r, n]$ which represents
    \begin{center}
    \includegraphics[scale = 1]{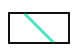}
\end{center}
    \item $[n, l]$ which we visualize as
    \begin{center}
    \includegraphics[scale = 1]{scaStuff/sca820_12.png}
\end{center}
\end{itemize}
We call symbols $s, l, r$ \textbf{strands}. We say a cell type ``contains a strand" if the cell type $[x, y]$ does not have $x = y = n$, and ``contains two strands" if the cell type $[x, y]$ has $x = y = s$.  Let $f:\mathbb{Z}\to\mathcal{C}$ be a function such that $0$ is the least $j\in\mathbb{Z}$ such that its image is a cell type containing a strand, and such that the sum of the number of strands in $f[\mathbb{Z}]$ is $n$. We call $f$ a \textbf{row}, which can be visualized as follows:
 \begin{center}
    \includegraphics[scale = 1]{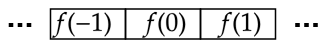}
\end{center}
An example of a row is the function:
\begin{align*}
    f(x) = [n, n]\text{ if }x < 0\text{ or }x > 1\\
    f(0) = [s, s]\\
    f(1) = [n, l]\\
\end{align*}
We can visualize $f$ as:
\begin{center}
    \includegraphics[scale = 1]{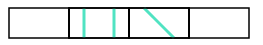}
\end{center}

Denote the space of all such functions by $\mathcal{F}$.

Now, take some list (not necessarily finite) of elements of $\mathcal{F}, [f_1,f_2,\dots]$. Define a symbol set $\mathcal{T}$ as containing the following symbols:
\begin{align*}
    s_j, r_j, l_j, n_j\text{ such that }j\in\mathbb{Z}
\end{align*}
We define a set $\mathcal{D}$ as follows:
\begin{align*}
    \mathcal{D} = \left\{[x_j, y_{j+1}]|x_j, y_{j+1}\in\mathcal{T}, j\in\mathbb{Z}\right\}\cup\{\overline{[r_j, l_{j+1}]}|r_j, l_{j+1}\in\mathcal{T}, j\in\mathbb{Z}\}
\end{align*}
An element of $\mathcal{D}$ will represent a place in an ordering of elements of $\mathcal{S}$ with respect to a row $f$. 
We will make this more precise later in the report.

Let $\chi:\mathcal{S}\to\{0,1\}$ be the map such that
\begin{align*}
    \chi(x) = 1\text{ if }x = s, r,\text{ or }l\\
    \chi(x) = 0\text{ otherwise}
\end{align*}
Intuitively, $\chi$ tells us whether a element of $\mathcal{S}$ is a strand or not. We define a few functions as follows in the following, mutually exclusive cases:

\hfill\break
\textbf{Case 1:} $f(0) = [x, y]$ is such that $x = s, r,$ or $l$
\hfill\break
Define $g_f:\mathbb{Z}\to\mathcal{D}$ as follows: for each $z\in\mathbb{Z}$, if $f(z) = [a, b]$, let $g_f(z) = [a_{2z + 1}, b_{2z + 2}]$.

We now define a function $r_f:\mathbb{Z}\to \mathcal{S}$: by 
\begin{align*}
    r_f(z) = a\text{ if }g_f\left(\left\lfloor \frac{z - 1}{2}\right\rfloor\right) = [a_{z}, b_{z + 1}]\text{ and }z\text{ is odd }\\
    r_f(z) = b\text{ if }g_f\left(\left\lfloor \frac{z - 1}{2}\right\rfloor\right) = [a_{z-1}, b_{z}]\text{ and }z\text{ is even }
\end{align*}

\hfill\break
\textbf{Case 2:} $f(0) = [x, y]$ is such that $x\neq s$
\hfill\break
Define $g_f:\mathbb{Z}\to\mathcal{D}$ as follows: for each $z\in\mathbb{Z}$, if $f(z) = [a, b]$, let $g_f(z) = [a_{2z}, b_{2z + 1}]$.

We define a function $r_f:\mathbb{Z}\to \mathcal{S}$ as follows:
\begin{align*}
    r_f(z) = a\text{ if }g_f\left(\left\lfloor \frac{z}{2}\right\rfloor\right) = [a_{z}, b_{z + 1}]\text{ and }z\text{ is even }\\
    r_f(z) = b\text{ if }g_f\left(\left\lfloor\frac{z}{2}\right\rfloor\right) = [a_{z-1}, b_z]\text{ and }z\text{ is odd}
\end{align*}
\hfill\break
In both cases, $g_f$ assigns a notation of ``position" to elements of $\mathcal{S}$ in the output of $f$ relative to each other. $r_f$ uses this idea of ``position" to recover specific strands or $n$ given an integer. We call $z$ the \textbf{position} of $a_z$ in $f(k) = [a_z, b_{z + 1}]$ or $f(k) = [b_{z-1}, a_z]$. For example, $g_f$ where $f$ was as in our previous example would be 
\begin{align*}
    g_f(x) = [n^{\left(\varnothing\right)}_{2x + 1}, n^{\left(\varnothing\right)}_{2x + 2}]\text{ for }x < 0, x > 1\\
    g_f(0) = [s_1, s_2]\\
    g_f(1) = [n_3, l_4]
\end{align*}
We can visualize this as follows:
\begin{center}
    \includegraphics[scale = 0.6]{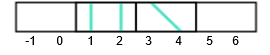}
\end{center}

Let $\mathcal{U}$ be the symbol set containing symbols $s_j^{(k)}, r_j^{(k)}, l_j^{(k)}, n_j^{\left(\varnothing\right)}$ for all $j, k\in\mathbb{Z}$. We now define a set $\mathcal{E}$ over $\mathcal{U}$. Let $\mathcal{E}$ be the set containing, for all $j\in\mathbb{Z}, k\in\mathbb{N}^+$:
\begin{itemize}
    \item $[n_j^{\left(\varnothing\right)}, n^{\left(\varnothing\right)}_{j+1}]$
    \item $[s_j^{(k)}, n^{\left(\varnothing\right)}_{j+1}]$
    \item $[n_j^{\left(\varnothing\right)}, s_{j+1}^{(k)}]$
    \item $[s_j^{(k)}, s_{j+1}^{(k+1)}]$
    \item $[r_j^{(k)}, n^{\left(\varnothing\right)}_{j+1}]$
    \item $[n_j^{\left(\varnothing\right)}, l_{j+1}^{(k)}]$
    \item $[r_j^{(k)}, l_{j+1}^{(k+1)}]$
    \item $\overline{[r_j^{(k)}, l_{j+1}^{(k+1)}]}$
\end{itemize}
One can think of $\mathcal{E}$ as cells that have both a position in the ordering induced by a row for their elements and a numbering of the strands. This notion will be made more precise later.

Let $\#_f:\mathbb{Z}\to\mathbb{Z}$ be defined as $\#_f(z) = 0$ for $z\leq 0$ and $\#_f(z) = \sum_{j=1}^z\chi\circ r_f(j)$. One can think of $\#_f(z)$ as counting the strands in the image of $f$ preceding some position. We call $s_j, r_j,$ and $l_j$ \textbf{strands} for any $j\in\mathbb{Z}$. We now define the function $h_f:\mathbb{Z}\to\mathcal{E}$ as:
\begin{align*}
    \text{if }g_f(z) = [x_j, y_{j+1}]\text{ and }x_j, y_j\text{ are strands then }h_f(z) = [x_j^{(\#_f(j))}, y_{j+1}^{(\#_f(j + 1))}] \\
    \text{if }g_f(z) = [x_j, y_{j+1}]\text{ and }x_j \text{ is a strand but } y_{j+1} = n_{j+1}\text{ then }h_f(z) = [x_j^{(\#_f(j))}, n_{j+1}^{\left(\varnothing\right)}]\\
     \text{if }g_f(z) = [x_j, y_{j+1}]\text{ and }y_j \text{ is a strand but } x_j =n_j\text{ then }h_f(z) = [n_j^{\left(\varnothing\right)}, y_{j+1}^{(\#_f(j+1))}]\\
     \text{if }g_f(z) = [x_j, y_{j+1}]\text{ and }x_j =  n_j, y_{j+1} = n_{j+1}\text{ then }h_f(z) = [n_j^{\left(\varnothing\right)}, n_{j+1}^{\left(\varnothing\right)}]
\end{align*}
As an example, consider the $f$ used as an example row earlier. Then,
\begin{align*}
    h_f(x) = [n^{\left(\varnothing\right)}_{2x + 1}, n^{\left(\varnothing\right)}_{2x + 1}]\text{ if }x < 0, x > 1\\
    h_f(0) = [s_1^{(1)}, s_2^{(2)}]\\
    h_f(1) = [n_3^{\left(\varnothing\right)}, l_4^{(3)}]
\end{align*}
We can visualize this as follows:
\begin{center}
    \includegraphics[scale = 0.6]{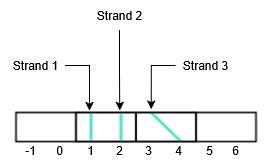}
\end{center}
We will also refer to the symbols $s_j^{(k)}, r_j^{(k)}, l_j^{(k)}$ for $j, k\in\mathbb{Z}$ as strands. We will further extend notation to say that the position of $x_j^{(k)}$ appearing in the image of $h_f$ is $j$. $h_f$ can be thought of as a description of the row $f$ with the additional information of where the strands are in relation to each other and to the first strand. Let $\mathcal{G}$ be the set of $h_f$ such that $f\in\mathcal{F}$. We call an an element of $\mathcal{G}$ a \textbf{generation}. Let $H_n$ be the space of (not necessarily finite) lists over $\mathcal{G}$. Let the list $h = [\delta_1,\delta_2,\dots]\in H_n$. Notice that we now refer to some $h_f\in\mathcal{G}$ where $h_f$ is the $i$th element of the list as $\delta_i$. We do this to make the notation that will be used in following sections clearer. We call $h$ finite if $|h| < \infty$ and $h$ infinite otherwise. If $h$ is finite, define $\length(h) = |h|$. If $h$ is infinite, $\length(h) = \infty$. Define $A = \{1, 2, \dots, z - 1\}$ if $h = [\delta_1,\delta_2,\dots, \delta_z]$ and $A = \mathbb{N}$ if $h$ is not finite. Define $\operatorname{child}_h:\mathcal{E}\times\mathcal{E}\to \mathcal{E}$ such that:
\begin{center}
    For each $i \in A$, if $\delta_i(0) \neq [n^{\left(\varnothing\right)}_{0},l_1^{(1)}], [s_1^{(1)},t]$ for some $t\in\mathcal{E}$, for all $k\in\mathbb{Z}$:
\begin{align*}
    \operatorname{child}(\delta_i(k), \delta_i(k + 1)) = \delta_{i+1}(k+1)
\end{align*}
    And if $\delta_i(0) = [n^{\left(\varnothing\right)}_{0},l_1^{(1)}]$ or $ [s_1^{(1)},t]$ for some $t\in\mathcal{E}$, then for all $k\in\mathbb{Z}$:
\begin{align*}
    \operatorname{child}(\delta_i(k), \delta{}_i(k + 1)) = \delta_{i+1}(k)
\end{align*}
\end{center}
We will often omit the subscript $h$ on $\operatorname{child}_h$ when the list $h$ in question is clear from context. Given an element $h\in H_n, h = [\delta_1,\delta_2,\dots]$, we call $\delta_i$ generation $i$ of $h$ and we call $h$ a \textbf{grid pattern}. We will sometimes refer to a grid pattern of finite length as a ``finite grid pattern" and a grid pattern of infinite length as an ``infinite grid pattern". For convenience, we will often refer to $C_{i, k} = \delta_i(k)$ as a \textbf{cell} of $\delta_i$, and will picture $\delta_i = C_{i,0}C_{i,1}\dots C_{i, m}$ for the last $m$ for which $\delta_i(m)$ contains a strand. To say a cell $C$ is equal to $\delta_i(k)$ for some $k$, we will abuse notation and write $C\in\delta_i$. Notice that $\operatorname{child}(C_{i, k}, C_{i, k+1}) = C_{i+1, k}$.

We now think of grid patterns geometrically. Consider $h\in H_n$ such that $h =[\delta_1,\delta_2,\dots]$, $\delta_1(0)$ is not one of the following:
\begin{itemize}
    \item $[s_1^{(1)}, t]$
    \item $[n_0^{\left(\varnothing\right)}, l_1^{(1)}]$
\end{itemize}
and such that $\delta_2(0)$ is one of the following, for some $t\in\mathcal{E}$
\begin{itemize}
    \item $[s_1^{(1)}, t]$
    \item $[n^{\left(\varnothing\right)}_0, l_1^{(1)}]$
\end{itemize}
Then, we may visualize $h$ as follows: 
\begin{center}
    \includegraphics[scale = 1]{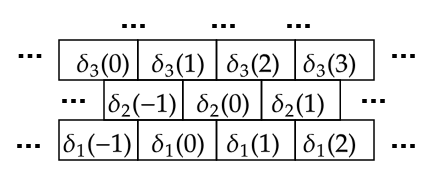}
\end{center}
With the cell notation, this becomes:
\begin{center}
    \includegraphics[scale = 1]{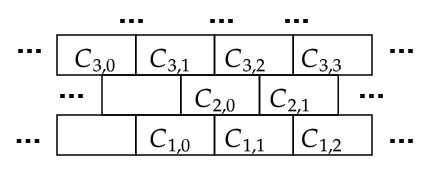}
\end{center}
Notice that, given cells $A, B$, we have the following geometric interpretation of the child function:
\begin{center}
    \includegraphics[scale = .7]{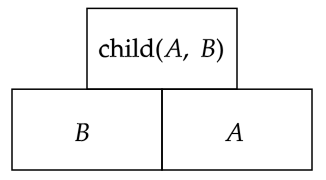}
\end{center}
We say the cell $[r_j^{(m)}, l_{j+1}^{(m+1)}]$ contains a \textbf{Z-crossing}, and that the cell $\overline{[r_j^{(m)}, l_{j+1}^{(m+1)}]}$ contains an \textbf{S-crossing}. Throughout this report, when it is not important whether a cell contains a $S-$ or $Z-$ crossing, we will refer to an arbitrary crossing as being of the form $[x_j^{(k)}, y_{j+1}^{(k+1)}]$ where $x = r, y = l$. We will often refer to Z- and S- crossings as \textbf{crossings} when it is unimportant whether the cell in question contains a Z- or S-crossing.

\subsection{Continuity}
We now aim to define some notion of continuity. Recall that given a grid pattern $g$, $C_{k, 0}$ denotes the first (or ``leftmost" in the geometric interpretation) cell which contains a strand in generation $k$ of $g$. Let $g\in H_n$. If $g$ is finite, let $A = \{1, 2,\dots, \length(g)-1\}$ and if $g$ is infinite, let $A = \mathbb{N}^+$. We use the notation $[\cdot, \cdot]$ to mean a cell with some unknown elements of $\mathcal{U}$ filling the positions $\cdot$, possibly including an overbar. We will enumerate the conditions for continuity, alongside a few examples illustrating each one. We say $g$ is \textbf{continuous} if, for all $i\in A$, $\gen(C_{i, k}, C_{i, k + 1}) = [y, z], y,z\in\mathcal{U}$ where
\begin{itemize}
    \item $[y, z]$ contains two strands if and only if $C_{i, k} = [r_j^{(t)}, \cdot]$ or $C_{i, k} = [\cdot, s_j^{(t)}]$ and $C_{i, k+1} = [\cdot, l_j^{(t)}]$ or $C_{i, k+1} = [s_j^{(t)}, \cdot]$.
    \begin{center}
        \includegraphics[]{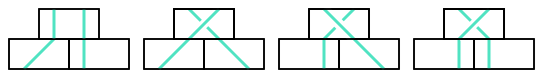}
    \end{center}
    \item $[y, z]$ has $y\neq n_j^{\left(\varnothing\right)}$ for any $j$ and $z = n_{j+1}^{\left(\varnothing\right)}$ for some $j$ if and only if $C_{i, k} = [r_y^{(t)}, \cdot]$ or $C_{i, k} = [\cdot, s_y^{(t)}]$ for some $y, t$ and $C_{i, k+1} \neq [s_y^{(t)}, \cdot]$ and $C_{i,k+1}\neq [\cdot, l_y^{(t)}]$ for some $y, t$.
    \begin{center}
        \includegraphics[]{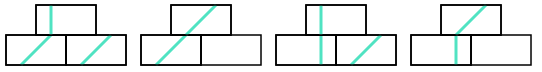}
    \end{center}
    \item $[y, z]$ has $z\neq n_{j+1}^{\left(\varnothing\right)}$  and $y = n_j^{\left(\varnothing\right)}$ for some $j$ if and only if $C_{i, k} \neq [r_y^{(t)},\cdot]$ or $C_{i, k} \neq [\cdot, s_y^{(t)}]$ and $C_{i, k+1} = [s_y^{(t)}, \cdot]$ or $C_{i, k+1} = [\cdot, l_y^{(t)}]$ for some $y, t$.
    \begin{center}
        \includegraphics[]{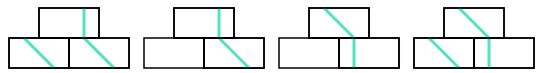}
    \end{center}
    \item $[y, z] = [n_j^{\left(\varnothing\right)}, n^{\left(\varnothing\right)}_{j+1}]$ for some $j\in\mathbb{Z}$ if and only if $C_{i, k} \neq[r_y^{(t)}, \cdot], [\cdot, s_y^{(t)}]$ for some $y, t$ and $C_{i, k+1} \neq [s_y^{(t)},\cdot],[\cdot, l_y^{(t)}]$ for some $y, t$.
    \begin{center}
        \includegraphics[]{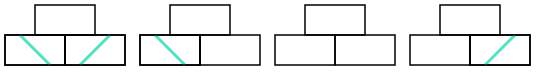}
    \end{center}
\end{itemize}
Finally, we define, for all $n\in\mathbb{N}^+$:
\begin{align*}
    G_n = \left\{h|h\in H_n\text{ such that }h\text{ is continuous}\right\}
\end{align*}
We call $G_n$ the set of continuous grid patterns, and because we will often only be concerned with continuous grid patterns, we will sometimes refer to elements of $G_n$ as grid patterns. It will be made clear from context whether a given grid pattern is continuous or not.

As an example of an element of $G_3$, consider the grid pattern
\begin{align*}
    g = [\delta_1,\delta_2,\delta_3, \delta_4]\\
    \delta_1 = [s_1^{(1)}, s_2^{(2)}][n^{\left(\varnothing\right)}_3, l_4^{(3)}]\\
    \delta_2 = [n^{\left(\varnothing\right)}_0, l_1^{(1)}]\overline{[r_2^{(2)}, l_3^{(3)}]}\\
    \delta_3 = [n^{\left(\varnothing\right)}_0, s_1^{(1)}][n^{\left(\varnothing\right)}_2, s_3^{(2)}][s_4^{(3)}, n^{\left(\varnothing\right)}_5]\\
    \delta_4 = [r_1^{(1)}, n^{\left(\varnothing\right)}_2][r_3^{(2)}, l_4^{(3)}]
\end{align*}
This grid pattern is continuous, and has geometric interpretation:
\begin{center}
    \includegraphics[scale = 1]{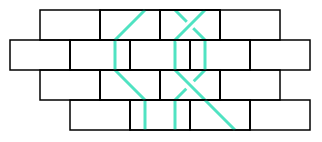}
\end{center}
Notice that all of the strands continue from one generation to the next, justifying the name continuous. Now, consider the grid pattern, which is an element of $H_2$
\begin{align*}
    g = [\delta_1,\delta_2,\delta_3]\\
    \delta_1 = [n^{\left(\varnothing\right)}_0, r_1^{(1)}][n^{\left(\varnothing\right)}_2, l_3^{(2)}]\\
    \delta_2 = \overline{[r_1^{(1)}, l_2^{(2)}]}\\
    \delta_3 = [s_1^{(1)}, n^{\left(\varnothing\right)}_2][s_3^{(2)},n^{\left(\varnothing\right)}_4]
\end{align*}
This grid pattern is not continuous, as can be seen from the geometric interpretation:
\begin{center}
    \includegraphics[scale = 1]{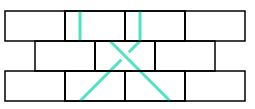}
\end{center}


We now define the notion of \textbf{composition} of grid patterns $\circ:H_n\times H_n\to H_n$. Let $h, g\in H_n$ such that $h$ is a finite grid pattern. Then, $h = [\delta_1,\delta_2,\dots, \delta_k], g = [\delta_1',\delta_2',\dots]$ for some $k\in\mathbb{N}$. Let the infinite list $[\tau_1,\tau_2,\dots] = [\delta_1,\delta_2,\dots, \delta_k,\delta_1',\delta_2',\dots]$ be such that $\tau_i = \delta_i$ if $1\leq i\leq k$ and $\tau_i = \delta_{i-k}'$ for $i > k$. Then, $[\tau_1,\tau_2,\dots]\in H_n$. Let $h\circ g = [\tau_1,\tau_2,\dots]$. For an example of composition of grid patterns, consider the patterns 
\begin{align*}
    g = [\delta_1, \delta_2]\\
    \delta_1 = [s_1^{(1)}, n^{\left(\varnothing\right)}_2][s_3^{(2)}, n^{\left(\varnothing\right)}_4]\\
    \delta_2 = [n^{\left(\varnothing\right)}_0, l_1^{(1)}][n^{\left(\varnothing\right)}_2, l_3^{(2)}]
\end{align*}
\begin{center}
    \includegraphics[scale = 1]{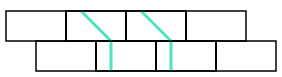}
\end{center}
and
\begin{align*}
    h = [\delta_1']\\
    \delta_1' = [n^{\left(\varnothing\right)}_0, s_1^{(1)}][n^{\left(\varnothing\right)}_2, l_3^{(2)}]
\end{align*}
\begin{center}
    \includegraphics[scale = 1]{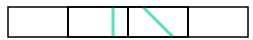}
\end{center}
The composition is:
\begin{align*}
    h\circ g = [\tau_1,\tau_2,\tau_3]\\
    \tau_1 = [s_1^{(1)}, n^{\left(\varnothing\right)}_2][s_3^{(2)}, n^{\left(\varnothing\right)}_4]\\
    \tau_2 = [n^{\left(\varnothing\right)}_0, l_1^{(1)}][n^{\left(\varnothing\right)}_2, l_3^{(2)}]\\
    \tau_3 = [n^{\left(\varnothing\right)}_0, s_1^{(1)}][n^{\left(\varnothing\right)}_2, l_3^{(2)}]
\end{align*}
\begin{center}
    \includegraphics[scale = 1]{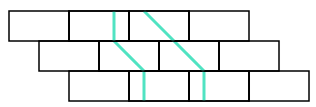}
\end{center}

 We call a finite grid pattern $s = [\delta_1,\dots, \delta_k]$ a \textbf{sublist} of a not necessarily infinite grid pattern $g = [\delta_1',\delta_2',\dots],$ if there exists a $i\leq \operatorname{length}(g)$ such that, for all $1\leq j\leq k, \delta_j = \cdots[\cdots]\delta_{i + j-1}'[\cdots]\cdots$. Notice that by the definition of a child function, this implies $\operatorname{child}_s$ is the restriction of $\operatorname{child}_g$ to $[\delta_{i}',\delta_{i+1}',\dots,\delta_{i + k-1}']$. 

We denote the composition of $a$ with itself $h$ times by $a^h$. We now define a partial function $\per: H_n\to H_n$. If $g$ is the composition of some sublist $a$ with itself $h$ times for $h > 1$ and for all $b$ such that there exists $y\in\mathbb{N}$ with $b^y = g,$ $\length(b)\geq \length(a)$, we write $g = a^h$ and say that $g$ is a \textbf{finite repeating list} and $a$ is a \textbf{repeating part} of $g$. We denote $a$ by $\per(g)$. We define a \textbf{periodic list} $g\in G_n$ to be a list $g$ that is a list $a$ composed with itself infinitely many times such that there exists no shorter sublist $b$ with $b^\infty  = g$. We write $g = a^\infty$, and say that $a$ is a \textbf{period} of $g$, denoted $\per(g)$. If no such $a$ exists, $\per$ is not defined on $g$. Notice that $\per(g)$ is unique as all grid patterns equal to $\per(g)$ must be the same length and have $\delta_1$ as its first generation where $g = [\delta_1,\dots]$.

 Finally, we define the width of a generation. Let $\delta_j$ be a generation of a grid pattern $g$, and let $y$ be the greatest $y$ such that $C_{j, y}$ contains a strand. We define the \textbf{width} of $\delta_j$ to be $\width(\delta_j) = y + 1$. If $\{\width(\delta_j):\delta_j\in g\}$ for a grid pattern $g$ is bounded above in $\mathbb{Z}$, we define $\width(g) = \max_{\delta_j\in g} \width(\delta_j)$.

\section{SCA Facts}\label{sec:SCAFacts}
The Stranded Cellular Automata (SCA) model is a grid of cells, with each cell in one of finitely many states. It employs two cellular automatons, one called the \textbf{turning rule} and the other called the \textbf{crossing rule} \cite{hh}. The turning rule determines whether strands will turn or go straight in a given cell given the information of whether the strands in that cell's neighbor cells turn or go straight. An example is below, similar to that of \cite{hh}:
\begin{center}
    \includegraphics[scale = 1.5]{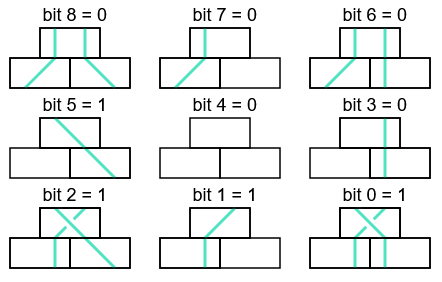}
\end{center}
\begin{center}
    \textbf{Figure 5:} Turning rule rule 39
\end{center}

Similarly, the crossing rule of a grid pattern is a cellular automata that determines whether the crossing of two strands is an $S-$ or $Z-$ crossing if those strands cross \cite{hh}. An example from \cite{hh} is below:
\begin{center}
    \includegraphics[scale = 1.5]{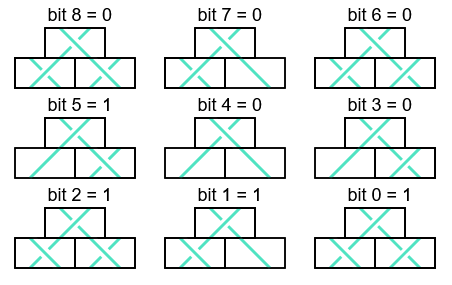}
\end{center}
\begin{center}
    \textbf{Figure 6:} Crossing rule 39
\end{center}

For any crossing or turning rule, we may write the rule as $s_8s_7s_6s_5s_4s_3s_2s_1s_0$ with $s_i$ corresponding to bit $i$. By referring to ``bit $i$", we refer to $s_i$. We call the ``initial conditions" of a bit in the turning rule $b$ the bottom row of cells in the part of figure 5 corresponding to $b$.

We adopt the convention that bit four in the turning rule is always 0, that is, the empty cell is non-turning. We will find it useful to refer to bits of a given turning rule of an SCA pattern by:
\begin{itemize}
    \item The initial conditions for bit 0 will be referred to as $SS$. If bit 0 is 1, it will be referred to as $\frac{T}{SS}$. If it is 0, it will be referred to as $\frac{S}{SS}$.
     \item The initial conditions for bit 1 will be referred to as $SE$. If bit 0 is 1, it will be referred to as $\frac{T}{SE}$. If it is 0, it will be referred to as $\frac{S}{SE}$.
     \item The initial conditions for bit 2 will be referred to as $ST$. If bit 0 is 1, it will be referred to as $\frac{T}{ST}$. If it is 0, it will be referred to as $\frac{S}{ST}$.
    \item The initial conditions for bit 3 will be referred to as $ES$. If bit 3 is 1, it will be referred to as $\frac{T}{ES}$. If it is 0, it will be referred to as $\frac{S}{ES}$.
    \item The initial conditions for bit 5 will be referred to as $ET$. If bit 0 is 1, it will be referred to as $\frac{T}{ET}$. If it is 0, it will be referred to as $\frac{S}{ET}$.
    \item The initial conditions for bit 6 will be referred to as $TS$. If bit 0 is 1, it will be referred to as $\frac{T}{TS}$. If it is 0, it will be referred to as $\frac{S}{TS}$.
    \item The initial conditions for bit 7 will be referred to as $TE$. If bit 0 is 1, it will be referred to as $\frac{T}{TE}$. If it is 0, it will be referred to as $\frac{S}{TE}$.
    \item The initial conditions for bit 8 will be referred to as $TT$. If bit 0 is 1, it will be referred to as $\frac{T}{TT}$. If it is 0, it will be referred to as $\frac{S}{TT}$.
\end{itemize}
A turning (crossing) rule of a pattern is a turning (crossing) rule under which the pattern is the result of the rule applied to the first generation of the pattern for however many generations the pattern has. We now define a \textbf{turning configuration} of the (or corresponding to the) cells $A, B$ and a $\gen(A, B)\neq[n_j^{\left(\varnothing\right)},n^{\left(\varnothing\right)}_{j+1}]$. The turning configuration is a pair of cells $A'$ $B'$ that result from deleting the strands in $A$ and $B$ that do not continue into $\gen(A, B)$. The following is a visual:
\begin{center}
    \includegraphics[scale = 0.9]{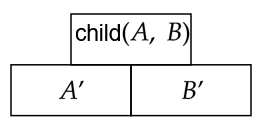}
\end{center}
\begin{center}
    \textbf{Figure 7:} Turning configuration
\end{center}
We will often refer to a turning configuration as a configuration. A turning configuration is said to produce a bit $B$ if all turning rules of the SCA pattern that consists of only the turning configuration have bit $B$. The following is an example:

Consider the pattern $[\delta_1,\delta_2]$ where $\delta_1 = [s_1^{(1)}, s_2^{(2)}][r_3^{(3)}, l_4^{(4)}],\delta_2 = [n_0, s_1^{(1)}]\overline{[r_2^{(2)}, l_3^{(3)}]}[s_4^{(4)}, n_5]$.
\begin{center}
    \includegraphics[scale = 1.2]{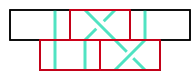}
\end{center}
The turning configuration corresponding to the cells outlined in red is:
\begin{center}
    \includegraphics[scale = 1.2]{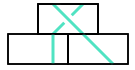}
\end{center}
Crossing configurations are defined similarly. We will write our turning rules for certain patterns as a binary expansion where each bit that does not have any bearing on the pattern is replaced with an ``X". We call this a \textbf{generic turning (crossing) rule}. Notice that a pattern has $X$ in some position of its generic turning rule if and only if there exist two turning rules of the pattern, one with a bit 1 and one with bit 0. Additionally, a generic turning (crossing) rule of a pattern has a bit $m\in\{0,1\}$ in position $i$ if and only if the turning (crossing) configuration for bit $m$ in position $i$ occurs somewhere in the pattern.

\section{Gliders}\label{sec:Gliders}
 We introduce some new terminology and notation. 
Let $g = [\delta_1,\delta_2,\cdots,\delta_m]$ be a finite continuous grid pattern. We first define the \textbf{strand-number} of a strand in generation $\delta_i$. If a strand $x_j^{(k)}$ is contained in some cell of $\delta_i$, we say that $x_j^{(k)}$ is strand $k$ in generation $i$.  We now define the notion of an \textbf{end index} of a strand $x_j^{(k)}$ in generation $\delta_i$, as follows:
\begin{itemize}
    \item If $x = s$, define the end index $y$ of $x_j^{(k)}$ to be $0$.
    \item If $x = l$ and $x_j^{(k)}$ is not such that $[r_{j-1}^{(k-1)}, x_j^{(k)}]$ or $\overline{[r_{j-1}^{(k-1)}, x_j^{(k)}]}$ is a cell in $\delta_i$, define the end index $y$ of $x_j^{(k)}$ to be $1$.
    \item If $x = l$ and $x_j^{(k)}$ is such that $[r_{j-1}^{(k-1)}, x_j^{(k)}]$ or $\overline{[r_{j-1}^{(k-1)}, x_j^{(k)}]}$ is a cell in $\delta_i$, define the end index $y$ of $x_j^{(k)}$ to be $0$.
    \item If $x = r$ and and $x_j^{(k)}$ is not such that $[x_j^{(k)}, l_{j+1}^{(k+1)}]$ or $\overline{[x_j^{(k)}, l_{j+1}^{(k+1)}]}$ is a cell in $\delta_i$, define the end index $y$ of $x_j^{(k)}$ to be $-1$.
    \item If $x = r$ and $x_j^{(k)}$ is such that $[x_j^{(k)}, l_{j+1}^{(k+1)}]$ or $\overline{[x_j^{(k)}, l_{j+1}^{(k+1)}]}$ is a cell in $\delta_i$, define the end index $y$ of $x_j^{(k)}$ to be $0$.
\end{itemize}
The end index measures how far the strand moves to the left in generation $\delta_i$.
It is measured relative to the starting position of the strand $x_j^{(k)}$ in generation $\delta_i$.  We notate this as $\operatorname{end}_i(k) = y$. Again, this is well-defined for all $1\leq k\leq n$. Let $g$ be a grid pattern. Then, the \textbf{global end index} of $g$ is $\sum_{k=1}^m\left(\send_k\left(t_1^{(1)}\right)\right)$ where $t_1^{(1)}$ is the strand at index 1 in generation $\delta_k$. The global end index measures how far the pattern travels to the left of the starting position of strand 1 in generation $\delta_1$. It is measured relative to the starting position of strand 1 in $\delta_1$.

Let $g$ be a pattern with $\per(g) = a$. The \textbf{speed} of $g = a^\infty$ is $\frac{d}{\length(a)}$ where $d$ is the global end index of $a$; we denote this by $\operatorname{Speed}(g)$. Note that $\frac{d}{\length(a)}$ is a symbol and not a rational number, as it conveys the information of the number of positions the glider $g$ moved to the left in exactly $\length(\per(g))$ generations. We do use the notation 1 to denote $\frac{a}{a}$ and $-1$ to denote $\frac{-a}{a}$ for positive $a$. Similarly, we use $0$ to denote $\frac{0}{a}$ for any $a\in\mathbb{N}^+$. As an example, consider the following grid pattern:
\begin{center}
    \includegraphics[]{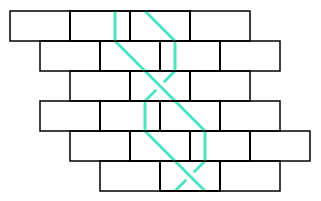}
\end{center}
This grid pattern has speed $\frac{1}{3}$. 

Suppose $a = [\delta_1,\dots,\delta_k]$. Then, we define a \textbf{chain} in $a$ to be a sequence of cells $C_{1, r_1},\dots, C_{k, r_{k}}$ of $a$ such that there exist strands in $C_{1, w}, C_{1, w'}$ for $w < r_1 < w'$ and such that, for each $1\leq i\leq r_{k-1}$, at least one of the following holds:
\begin{itemize}
    \item $\gen\left(C_{i, r_i}, C_{i, r_i + 1}\right) = C_{i+1, r_{i+1}}$\\
    \item $\gen\left(C_{i, r_i - 1}, C_{i, r_i}\right) = C_{i+1, r_{i+1}}$
\end{itemize}
If no element of a chain contains a strand, call that chain a \textbf{null chain}. We define an $n-$\textbf{stranded repeating SCA pattern} on the SCA model to be an SCA pattern $g = a^\infty$ such that $a$ contains no null chains. We define an $n-$\textbf{stranded glider} to be a $n-$stranded repeating SCA pattern with nonzero speed. Intuitively, one can think of an SCA glider as a periodic pattern that moves across the grid over time and stays roughly in one piece.

This definition differs slightly from the one given on \cite[Page 142]{mathBasis}, which would allow patterns such as $([[n^{\left(\varnothing\right)}_0,l_1^{(1)}][n^{\left(\varnothing\right)}_2,l_3^{(2)}][n^{\left(\varnothing\right)}_4,n^{\left(\varnothing\right)}_5][n^{\left(\varnothing\right)}_6,l_7^{(3)}]])^\infty$, pictured below:
\begin{center}
    \includegraphics[scale = 0.9]{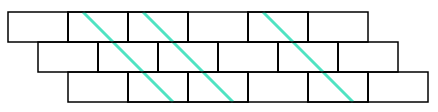}
\end{center}
to be gliders. Because these patterns are, in some sense, built from two separate patterns that do not affect each other, we do not call these patterns gliders.

In studying gliders, we are concerned with deriving turning rules from SCA patterns. As such, we will mostly consider generic turning rules. This allows us to capture all turning rules of a glider in one expansion, however, we will often just refer to it as ``the turning rule" for a given pattern.
\begin{definition}
    Given a finite list of generations $[\delta_1,\delta_2,\dots,\delta_k]$, a \textbf{cyclic permutation} of that list is $[\delta_{i+1},\dots, \delta_k,\delta_1,\dots, \delta_i]$ where $1\leq i\leq k$.
\end{definition}
\begin{definition}
    Given a glider $g$, a \textbf{shift} of $g$ is $a^\infty$ where $a$ is any cyclic permutation of $\per(g)$.
\end{definition}
For example, 
\begin{align*}
    [[n_0, s_1^{(1)}][n_2, l_3^{(2)}], [s_1^{(1)}, s_2^{(2)}],[n_0, l_1^{(1)}][s_2^{(2)}, n_3]]^\infty\text{ is a shift of }[[s_1^{(1)}, s_2^{(2)}],[n_0, l_1^{(1)}][s_2^{(2)}, n_3], [n_0, s_1^{(1)}][n_2, l_3^{(2)}]]^\infty
\end{align*}
    \begin{center}
        \includegraphics[]{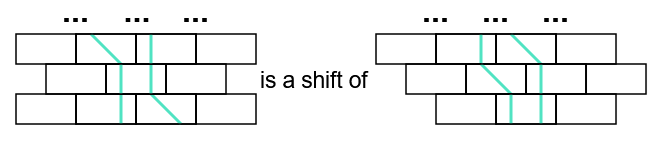}
    \end{center}
We introduce the notion for a \textbf{subpattern}. Define a function $\text{type}:\mathcal{E}\to \mathcal{S}$ by, for all $j,k\in\mathbb{Z}$:
\begin{align*}
    \text{type}(n_j^{\left(\varnothing\right)}) = n\\
    \text{type}(s_j^{(k)}) = s\\
    \text{type}(l_j^{(k)}) = l\\
    \text{type}(r_j^{(k)}) = r\\
\end{align*}
The type function returns the ``kind" of strand (or lack of strand) an element of $\mathcal{E}$. In this way it ``forgets" information.

 We call a (not necessarily infinite) grid pattern $s = [\delta_1,\dots], \delta_i = C_{i, 1}\dots C_{i, r_i}$ a subpattern of a grid pattern $g = [\delta_1',
\dots],\delta_i' = C_{i, 1}'\dots C_{i, v_i}'$ with $\length(s)\leq\length(g)$ if there exists an $i\in\mathbb{N}$ such that for all $1 \leq j \leq \length(s)$ if $\length(s)$ is finite and for all $1\leq j<\infty$ otherwise, there exist $C_{i, j}, k_j\in\mathbb{N}$ such that $C_{j, 1} = [w,x], C_{j + i, C_{i, j}}' = [y, z]$ for some $w, x, y, z\in\mathcal{E}$ and at least one of the following holds:
\begin{itemize}
    \item[1.] $\text{type}(w) = \text{type}(y)$ and $\text{type}(x) = \text{type}(z)$.
    \item[2.] $\text{type}(w) = n$ and $\text{type}(x) = \text{type}(z)$.
    \item[3.] $\text{type}(w) = \text{type}(x) = n$.
\end{itemize}
and such that $C_{j, r_j} = [w,x], C_{j + i, k_j}' = [y, z]$ for some $w, x, y, z\in\mathcal{E}$ and at least one of the following holds:
\begin{itemize}
    \item[1.] $\text{type}(w) = \text{type}(y)$ and $\text{type}(x) = \text{type}(z)$.
    \item[2.] $\text{type}(w) = \text{type}(y)$ and $\text{type}(x) = n$.
    \item[3.] $\text{type}(w) = \text{type}(x) = n$
\end{itemize}
Intuitively, a subpattern of $g$ is a continuous grid pattern obtained by deleting some number of strands in each generation of $g$. An example is as follows. The pattern:
\begin{center}
    \includegraphics[]{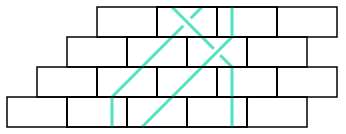}
\end{center}
has a maximal subpattern:
\begin{center}
    \includegraphics[]{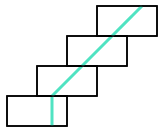}
\end{center}
If an SCA pattern contains no crossings, then we are justified in referring to the subpattern that is the $k$th strand in each generation as the $k$th strand of the pattern. In general, if $g$ is an SCA pattern, let the $k$th strand be the maximal continuous subpattern which contains one strand in each generation such that the first generation contains strand $k$. For example, the above example gives the first strand of the pictured subpattern. If $g = a^\infty$, we call the speed of strand $k$ the speed of the subpattern containing only strand $k$.
We now introduce some shorthand which will be used frequently throughout this section. We say two strands, $s$ and $t$, \textbf{interact} in generation $\delta_i$ if there exists a turning configuration in $\delta_i$ involving strands $s$ and $t$. By ``We have $B$" for some turning rule bit $B$ when it is clear from context that we are referring to a glider $g$ we mean ``There exists a turning rule of $g$ which contains the bit $B$".


\section{One-Stranded Gliders}\label{sec:OneStrand}
First note that we need not consider any crossing rules as there are no crossings in 1-stranded SCA patterns. Because gliders repeat by necessity, it suffices to consider the patterns generated by the following initial conditions. There are 4 conditions that we need to consider: 
\begin{center}
    \includegraphics[scale=1]{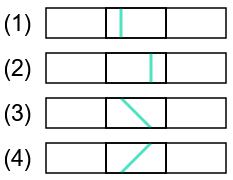}
\end{center}
\textit{Case 1:} $\delta_1$ is:
\begin{center}
    \includegraphics[scale=0.9]{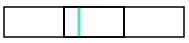}
\end{center}
We must consider the case where $a_1 = \frac{T}{ES}$ and $a_2 = \frac{S}{ES}$. 

\hfill\break
$\bm{a_1 = \frac{S}{ES}}$

$\delta_2 = [n_0, s_1^{(1)}]$
\begin{center}
    \includegraphics[scale=0.9]{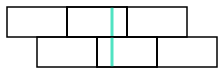}
\end{center}
so we must consider $b_1 = \frac{S}{SE}, b_2 = \frac{T}{SE}$.

\hfill\break
$\bm{b_1 = \frac{S}{SE}}$

$ \delta_3 = [s_1^{(1)}, n_2]$
\begin{center}
    \includegraphics[scale=0.9]{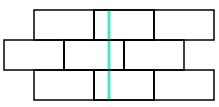}
\end{center}
so we have the repeating SCA pattern $g_1 = ([\delta_1, \delta_2])^\infty$ with speed 0 and turning rule $X0X00XXXX$.

\hfill\break
$\bm{b_2 = \frac{T}{SE}}$

 $ \delta_3 = [r_1^{(1)}, n_2]$
 \begin{center}
    \includegraphics[scale=0.9]{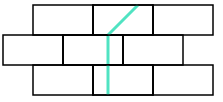}
\end{center}
 so we must consider $c_1 = \frac{T}{TE}, c_2 = \frac{S}{TE}$. 

 \hfill\break
 $\bm{c_1 = \frac{T}{TE}}$
 
 $\delta_4 = [r_1^{(1)}, n_2] = \delta_3$
 \begin{center}
    \includegraphics[scale=0.9]{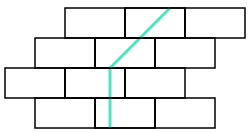}
\end{center}
 so the glider $g_2 = ([\delta_3])^\infty$ has speed $-1$ and turning rule $XXXX0XX1X$.
 
 \hfill\break
 $\bm{c_2 = \frac{S}{SE}}$
 
 $\delta_4 = [s_1^{(1)}, n_2] = \delta_1$
  \begin{center}
    \includegraphics[scale=0.9]{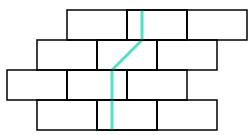}
\end{center}
 so we have a glider $g_3 = ([\delta_1, \delta_2, \delta_3])^\infty$ with speed $-\frac{1}{3}$ and turning rule $X1X00XX0X$. 
 
 \hfill\break
 $\bm{a_2 = \frac{T}{ES}}$
 
 $\delta_2 = [n_0, l_1^{(1)}]$
 \begin{center}
    \includegraphics[scale=0.9]{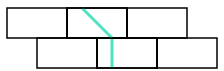}
\end{center}
 so we must consider $d_1 = \frac{T}{ET}, d_2 = \frac{S}{ET}$.
 
 \hfill\break
 $\bm{d_1 = \frac{T}{ET}}$
 
$\delta_3 = [n_0, l_1^{(1)}] = \delta_2$
 \begin{center}
    \includegraphics[scale=0.9]{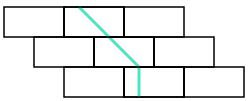}
\end{center}
so we have the glider $g_4 = ([\delta_3])^\infty$ with speed $1$ and turning rule $XXXX01XXX$. 

\hfill\break
$\bm{d_2 = \frac{S}{ET}}$

$\delta_3 = [n_0, s_1^{(1)}]$
 \begin{center}
    \includegraphics[scale=0.9]{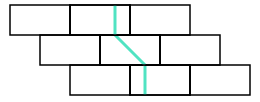}
\end{center}
and we now must consider $e_1 = \frac{T}{SE}, e_2 = \frac{S}{SE}$.

\hfill\break
$\bm{e_1 = \frac{T}{SE}}$

$\delta_4 = [r_1^{(1)}, n_2]$
 \begin{center}
    \includegraphics[scale=0.9]{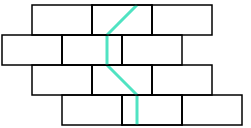}
\end{center}
so we must consider $f_1 = \frac{T}{TE}, f_2 = \frac{S}{TE}$.

\hfill\break
$\bm{f_1 = \frac{T}{TE}}$

$\delta_5 = [r_1^{(1)}, n_2] = \delta_4$
 \begin{center}
    \includegraphics[scale=0.9]{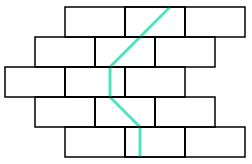}
\end{center}
so we get the glider $g_2$.

\hfill\break
$\bm{f_2 = \frac{S}{TE}}$

We have $\delta_5 = [s_1^{(1)}, n_2] = \delta_1$
 \begin{center}
    \includegraphics[scale=0.9]{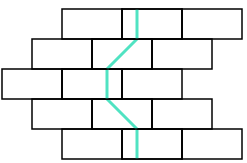}
\end{center}
for a repeating SCA pattern $g_5 = ([\delta_1, \delta_2, \delta_3, \delta_4])^\infty$ with speed 0 and turning rule $X1X100X0X$. 

\hfill\break
$\bm{e_2 = \frac{S}{SE}}$

$\delta_1 = [s_1^{(1)}, n_2], \delta_2 = [n_0, l_1^{(1)}], \delta_3 = [n_0, s_1^{(1)}], \delta_4 = [s_1^{(1)}, n_2] = \delta_1$
 \begin{center}
    \includegraphics[scale=0.9]{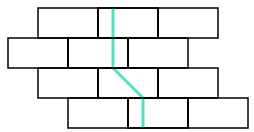}
\end{center}
so we have a glider $g_6 = ([\delta_1, \delta_2, \delta_3])^\infty$ with speed $\frac{1}{3}$ and turning rule $X0X100XXX$. 

\hfill\break
\textit{Case 2}: We start with $\delta_1$
\begin{center}
    \includegraphics[scale=0.9]{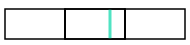}
\end{center}
We immediately must consider $a_1 = \frac{T}{SE}$ and $a_2 = \frac{S}{SE}$.

\hfill\break
$\bm{a_1 = \frac{T}{SE}}$

$\delta_2 = [r_1^{(1)}, n_2]$
\begin{center}
    \includegraphics[scale=0.9]{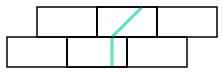}
\end{center}
so we must consider $b_1 = \frac{T}{TE}$ and $b_2 = \frac{S}{TE}$.

\hfill\break
$\bm{b_1 = \frac{T}{TE}}$

$\delta_3 = [r_1^{(1)}, n_2] = \delta_2$
\begin{center}
    \includegraphics[scale=0.9]{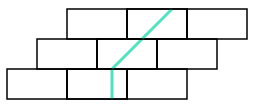}
\end{center}
so we have the glider $g_2$ 

\hfill\break
$\bm{b_2 = \frac{S}{TE}}$

$\delta_3 = [s_1^{(1)}, n_2]$
\begin{center}
    \includegraphics[scale=0.9]{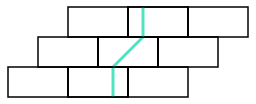}
\end{center}
See case 1.

\hfill\break
$\bm{a_2 = \frac{S}{SE}}$

In this case, $\delta_2 = [s_1^{(1)}, n_2]$
\begin{center}
    \includegraphics[scale=0.9]{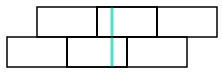}
\end{center}       
so see case 1.

\hfill\break
\textit{Case 3}: $\delta_1$ is
\begin{center}
    \includegraphics[scale=0.9]{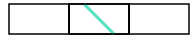}
\end{center}
We now must consider $a_1 = \frac{T}{ET}, a_2 = \frac{S}{TE}$.

\hfill\break
$\bm{a_1 = \frac{T}{ET}}$

$\delta_2 = [n_0, l_1^{(1)}] = \delta_1$
\begin{center}
    \includegraphics[scale=0.9]{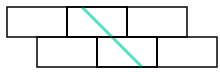}
\end{center}
so we have the glider $g_4$. 

\hfill\break
$\bm{a_2 = \frac{S}{ET}}$

$\delta_2 = [n_0, s_1^{(1)}]$,
\begin{center}
    \includegraphics[scale=0.9]{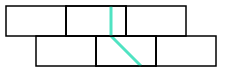}
\end{center}
See case 2.

\hfill\break
\textit{Case 4}: We have $\delta_1$
\begin{center}
    \includegraphics[scale=0.9]{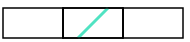}
\end{center}
We must consider $a_1 = \frac{T}{TE}, a_2 = \frac{T}{TE}$. 

\hfill\break
$\bm{a_1 = \frac{T}{TE}}$

$\delta_2 = [r_1^{(1)}, n_2] = \delta_1$,
\begin{center}
    \includegraphics[scale=0.9]{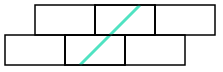}
\end{center}
so we get $g_2$.

\hfill\break
$\bm{a_2 = \frac{S}{TE}}$

$\delta_2 = [n_0, s_1^{(1)}]$
\begin{center}
    \includegraphics[scale=0.9]{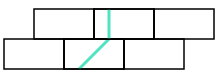}
\end{center}
so consider case 2.

Thus, we have proved the following classification result:
\begin{theorem}\thlabel{1StrandClassification}
    Let $g$ be a 1-stranded SCA pattern. Then $g$ is a repeating SCA pattern if and only if:
    \begin{itemize}
        \item[1.] $g = ([[s_1^{(1)},n_2],[n_0, s_1^{(1)}]])^\infty$ with speed 0 and turning rule $X0X00XXXX$.
    \begin{center}
    \includegraphics[scale=0.9]{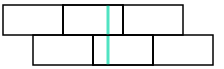}
\end{center}        
        \item[2.] $g = ([[r_1^{(1)}, n_2]])^\infty$ with speed $-1$ and turning rule $XXXX0XX1X$.
        \begin{center}
    \includegraphics[scale=0.9]{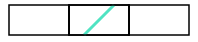}
\end{center}        
        \item[3.] $g = ([[n_0, s_1^{(1)}], [r_1^{(1)}, n_2],[s_1^{(1)}, n_2]])^\infty$ has speed $-\frac{1}{3}$ and turning rule $X1X00XX0X$.
        \begin{center}
    \includegraphics[scale=0.9]{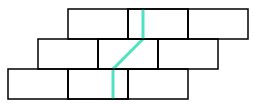}
\end{center}        
        \item[4.] $g = ([[n_0, l_1^{(1)}]])^\infty$ has speed $1$ and turning rule $XXXX01XXX$.
        \begin{center}
    \includegraphics[scale=0.9]{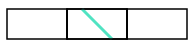}
\end{center}        
        \item[5.] $g = ([[s_1^{(1)}, n_2],[n_0, l_1^{(1)}], [n_0,s_1^{(1)}]])^\infty$ has speed $\frac{1}{3}$ and turning rule $X0X100XXX$.
\begin{center}
    \includegraphics[scale=0.9]{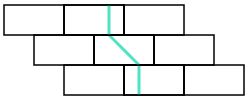}
\end{center}        
        \item[6.] $g = ([[s_1^{(1)},n_2],[n_0,l_1^{(1)}],[n_0,s_1^{(1)}],[r_1^{(1)},n_2]])^\infty$ has speed 0 and turning rule $X1X100X0X$.
\begin{center}
    \includegraphics[scale=0.9]{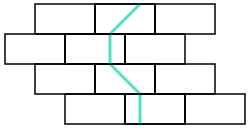}
\end{center}        
    \end{itemize}
\end{theorem}

\section{Two-Stranded Gliders}\label{sec:TwoStrands}
In this section, we aim to classify all 2-stranded SCA repeating patterns, which will also classify all 2-stranded gliders. We aim to prove a general classification result in section \ref{sec:decidabilityAndPure}, and the classification of two-stranded gliders will inform our proof of this result. Recall that a 2-strand SCA repeating pattern is of the form $g = a^\infty$ such that $a$ contains no null chains. If a generic crossing rule is $XXXXXXXXX$, we call it \textbf{trivial}. We first classify all 2-stranded gliders with speed 1 or -1 in \thref{speed2cor}. Recall that we write a cell as $[\cdot,\cdot]$ and a generation as a string of cells, $[\cdot,\cdot][\cdot,\cdot]\cdots$.
\begin{lemma}\thlabel{glidersSpeedctwoCase}
    The 2-stranded pattern $g = [[n_0,l_1^{(1)}][n_2,l_3^{(2)}]]^\infty$ is the unique 2-stranded speed $1$ glider. The 2-stranded pattern $g = [[r_1^{(1)}, n_2][r_3^{(2)}, n_4]]^\infty$ is the unique 2-stranded speed $-1$ glider.
\end{lemma}
\begin{proof}
    We only prove the result for speed 1 gliders as the statement for speed -1 gliders follows analogously. One can see that $g$ is a 2-stranded speed 1 glider. 
\begin{center}
    \includegraphics[]{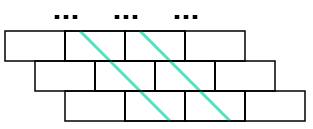}
\end{center}
    Suppose $g$ is a 2-stranded glider with speed 1. Let $a$ be the period of $g$. 
Then, $[n_0, l_1^{(1)}]$
\begin{center}
\includegraphics[]{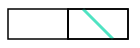}
\end{center}
occurs in some generation $j$ of $g$ since $g$ is traveling to the left, and the speed $\frac{a}{d}$ means that $g$ moves $a$ positions to the left in $|\per(g)|$ generations. Thus, we have two cases:
\begin{center}
\includegraphics[]{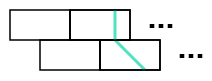}
\end{center}
or
\begin{center}
\includegraphics[]{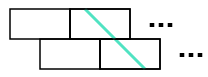}
\end{center}
In the first case, notice that $g$ can move at most $\length(a) - 1$ cells to the left in $\length(a)$ generations, so $\operatorname{Speed}(g)\neq 1,$ a contradiction. Thus, the second case holds, and we have shown that for all generations $\delta_i$, $\delta_i = [n_0, l_1^{(1)}]\dots$. Thus, all turning rules of $g$ have $\frac{T}{ET}$. Since we've been considering the first strand, $g$ contains no crossings. Now, notice that by the fact that gliders cannot contain null chains, there exists some $i$ for which $\delta_i$ is one of the following:
\begin{center}
\includegraphics[]{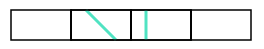}\\
\includegraphics[]{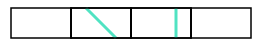}\\
\includegraphics[]{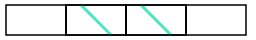}
\end{center}
In the first and second case, notice that the second strand of the pattern can move at most $\length(a) - 1$ cells to the left in $\length(a)$ generations. In the first case, this means that the first two cells of generation $\delta_{i + \length(a)}$ must be one of the following:
\begin{center}
\includegraphics[]{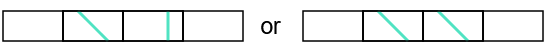}
\end{center}
a contradiction. In the second case, this means that the first two cells of generation $\delta_{i+\length(a)}$ must be:
\begin{center}
\includegraphics[]{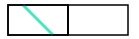}
\end{center}
a contradiction by the same argument as above. Thus, the third case holds. All turning rules of $g$ must have $\frac{T}{TT}$ or $\delta_{i + 1}$ would be:
\begin{center}
    \includegraphics[]{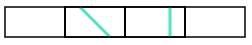}
\end{center}
a contradiction. Thus, if $g$ is a 2-stranded speed 1 glider, $g = [[n_0,l_1^{(1)}][n_2,l_3^{(2)}]]^\infty$. 
\end{proof}
\begin{corollary}\thlabel{speed2cor}
    Let $g\in G_2$ be an SCA pattern. $g$ is a repeating SCA pattern of speed 1 or -1 if and only if:
    \begin{enumerate}
        \item $g = [[n_0, l_1^{(1)}][n_2, l_3^{(2)}]]^\infty,$ which has turning rule $XXXX01XXX$, a trivial crossing rule, and speed 1.
\begin{center}
    \includegraphics[scale=0.9]{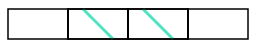}
\end{center}
        \item $g = [[r_1^{(1)}, n_2][r_3^{(2)}, n_4]]^\infty$, which has turning rule $XXXX0XX1X,$ a trivial crossing rule, and speed -1.
\begin{center}
    \includegraphics[scale=0.9]{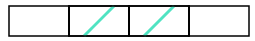}
        \end{center}
    \end{enumerate}
\end{corollary}
We will prove that any 2-stranded repeating SCA pattern with a strand with speed $1$ or $-1$, respectively, is a speed $1$ or $-1$ glider, respectively. Recall that speed $\pm1$ specifically means speed $\frac{\pm1}{1}$.
\begin{lemma}\thlabel{speedCCriteria}
    Let $g$ be a two-stranded repeating SCA pattern. If there exists a $N\in\mathbb{N}$ such that for all $k\geq N$, the first strand in $\delta_k$ turns left or for all $k\geq N$, the second strand in $\delta_k$ turns left, $g$ has speed 1. If there exists a $N\in\mathbb{N}$ such that for all $k\geq N$, the first strand in $\delta_k$ turns right or for all $k\geq N$, the second strand in $\delta_k$ turns right, $g$ has speed -1.
\end{lemma}
\begin{proof}
    We only prove the result for the cases where there exists a $N\in\mathbb{N}$ such that for all $k\geq N$, the first strand in $\delta_k$ turns left or for all $k\geq N$, the second strand in $\delta_k$ turns left; as the other cases follow analogously. Suppose there exists a $N\in\mathbb{N}$ such that for all $k\geq N$, the first strand in $\delta_k$ turns left. Because $g$ is a glider, there exists some $\delta_k, k\geq N$ such that one of the following holds:
\begin{center}
    \includegraphics[]{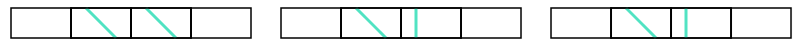}
\end{center}
    Notice that $\delta_k = \delta_{k + \length(\per(g))}$, so in cases two and three, strand 2 must travel $\length(\per(g))$ positions to to the right in $\length(\per(g)) - 1$ generations, a contradiction. In case 1, notice that strand 2 must turn in every generation between $k$ and $\length(\per(g))$, so $\delta_{k + 1} = [n_0,l_1^{(1)}][n_2,l_3^{(2)}].$ Thus, $g = ([n_0,l_1^{(1)}][n_2,l_3^{(2)}])^\infty$ because $g$ is a repeating SCA pattern, so $g$ has speed 1.

    Now suppose there exists a $N\in\mathbb{N}$ such that for all $k\geq N$, the second strand in $\delta_k$ turns left. If $g$ contains a crossing, we have one of the following
\begin{center}
    \includegraphics[]{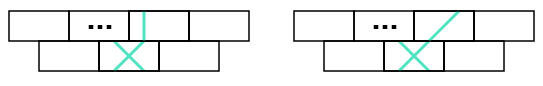}
\end{center}
    both contradictions, so $g$ can contain no crossings. From here, the proof is nearly identical to first case.
\end{proof}
Recall that if a turning rule 
\begin{lemma}\thlabel{notspeedc2strands}
    Suppose $g$ is a 2-stranded repeating SCA pattern that does not have speed $1$ or $-1$. Then there exist turning rules of $g$ with bits $\frac{S}{ET}, \frac{S}{TE}$.
\end{lemma}
\begin{proof}
    For contradiction, suppose otherwise. We only prove the result in the case where all turning rules of $g$ have $\frac{T}{ET}$, as the case where all turning rules have bit $\frac{T}{TE}$ is analogous. Then, there exist generations $i$ and $i + 1$ in which
    \begin{center}
    \includegraphics[scale=0.9]{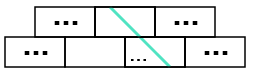}
\end{center}
    occurs. If the strand pictured is the first strand in generation $i$, then the cells to the left of the strand pictured must be empty, so $g$ must have speed $1$. This is a contradiction. $g$ also cannot contain a crossing for this reason, so we may call the first strand in the first generation the first strand throughout all of $g$. We now consider two cases:

    \textbf{Case 1:} Some turning rule of $g$ has $\frac{T}{ES}$. In this case the cells $[s_1^{(1)},...], [n_0, l_1^{(1)}]$ cannot occur in $g$. Thus, only the cells $[n_0, s_1^{(1)}], [r_1^{(1)}, n_2]$ can describe the first strand of $g$ in any given generation. If every generation contained $[n_0, s_1^{(1)}]$, $g$ would not be continuous, so there exists some generation $i$ for which $[r_1^{(1)}, n_2]$ occurs. $g$ does not have speed $-1$, so there exists a least generation $i + k > 1$ for which strand 1 is not given by $[r_1^{(1)}, n_2]$. Thus, $\delta_{i+k-1}, \delta_{i+k}$ must contain:
   \begin{center}
    \includegraphics[scale=0.9]{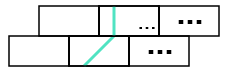}
\end{center}
    This contradicts the fact that $g$ cannot contain $[s_1^{(1)},...]$.

    \textbf{Case 2:} Some turning rule of $g$ has $\frac{S}{ES}$. By our assumption for contradiction, there exists some generation $\delta_i$ for which $[n^{\left(\varnothing\right)}_{j-1}, l_j^{(m)}]$ occurs for some $m\in\{1,2\}$. If $m = 1$, by \thref{speedCCriteria} $g$ would have speed 1, a contradiction, so suppose $m = 2$. There must be some $k\in\mathbb{N}_{\geq2}$ for which the second strand in $\delta_k$ turns left but the second strand in $\delta_{k-1}$ does not turn left, or else the second strand would turn left in all generations of $g$, a contradiction to \thref{speedCCriteria}. Thus, the options for $\delta_{k-1}$ are as follows:
\begin{center}
    \includegraphics[]{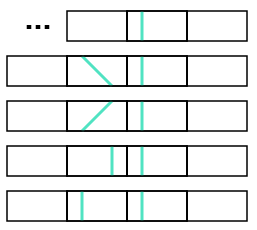}
\end{center}
    In the first case, because strand 2 turns left in the next generation, we get that all turning rules of $g$ have $\frac{T}{ES}$, a contradiction. In the second case, $g$ has speed 1 by the fact that $g$ has the rule $\frac{T}{ET}$ and \thref{speedCCriteria}. In the third and fourth cases, because strand 2 turns left in the next generation, $g$ contains a crossing, a contradiction. In the fifth case, by $\frac{S}{ES}$, we have:
\begin{center}
    \includegraphics[]{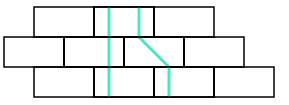}
\end{center}
    Thus, all rules of $g$ have $\frac{S}{ST}, \frac{T}{SS}$. Now, consider $\delta_{k+2}$. Strand 2 either turns or does not turn. If strand 2 does not turn, we get:
\begin{center}
    \includegraphics[]{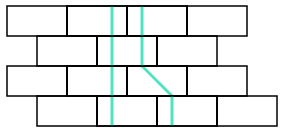}
\end{center}
    Thus, $\delta_{k + 3}$ contains a crossing because we have the rule $\frac{T}{SS}$, so suppose strand 2 turns in generation $\delta_{k+2}$
\begin{center}
    \includegraphics[]{scaStuff/sca11-18_6.png}
\end{center}
    If strand 2 turns in generation $\delta_{k + 3}$ then $g$ has speed -1 by \thref{speedCCriteria}, so we only must consider the behavior of strand 1 in $\delta_{k + 3}$. By $\frac{S}{ST}$, we have:
\begin{center}
    \includegraphics[]{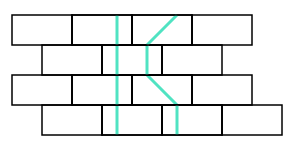}
\end{center}
Thus, this pattern repeats in $g$ for all future generations, so because $g$ is periodic, $g = [\delta_{k-1},\dots,\delta_{k+3}]^\infty$. Thus, there exist turning rules of $g$ with the bit $\frac{S}{ET}$, a contradiction to our assumption.
\end{proof}
We aim to classify all 31 2-stranded repeating SCA patterns in \thref{2StrandClassification}. To prove the classification result at the end of this section, we first consider the case where the glider has speed $1$ or $-1$. If the glider has speed $1$, then it must be $([[n_0, l_1^{(1)}][n_2, l_3^{(2)}]])^\infty$ with turning rule $XXXX01XX1$ by \thref{speed2cor}. If it has speed $-1$, it must be $([[r_1^{(1)}, n_2][r_3^{(2)}, n_4]])^\infty$ with turning rule $XXXX0XX11$ by \thref{speed2cor}.

We may assume that the bits $\frac{S}{TE}, \frac{S}{ET}$ appear in our rules throughout the classification by \thref{notspeedc2strands}. As every repeating SCA pattern must have the strands within four positions of each other during at least one point in its period, it suffices to classify all gliders arising from a first generation that contains one of the following configurations:
\begin{center}
    \includegraphics[scale=0.8]{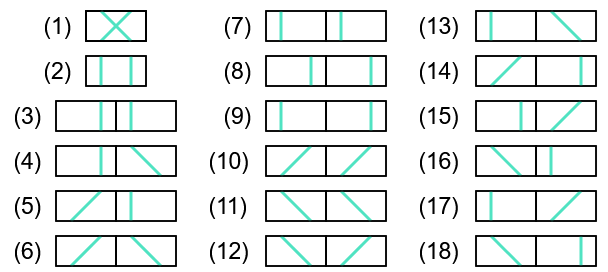}
\end{center}
To classify all 2-stranded gliders and SCA repeating patterns, we consider one case for each configuration listed above. We split each case into two parts: The first part assuming the bit $\frac{T}{ES}$, necessary for a positive-speed glider, and the second part assuming the bit $\frac{S}{ES}$. Within each part, we consider the different possibilities for bits in the generic turning rule. In each part, we set up a binary tree of assumptions for the possible bits of the turning rule, where letters $a, b, c,\dots$ indicate the depth of the tree at any given point and subscripts denote the branch taken at the vertex at the level of the letter. For example, the following tree depicts part 1 case 1:

\[\begin{tikzcd}
	&& \bullet \\
	{\bm{a_1}} & {} && {} & {\bm{a_2}} \\
	& {\bm{b_1}} && {} & {} &&& {\bm{b_2}} \\
	{\bm{c_1}} & {} & {\bm{c_2}} &&& {\bm{d_1}} &&&& {\bm{d_2}} \\
	&&&& {\bm{e_1}} && {\bm{e_2}} && {\bm{f_1}} && {\bm{f_2}} \\
	&&&&&&& {\bm{h_1}} && {\bm{h_2}}
	\arrow[no head, from=1-3, to=2-1]
	\arrow[no head, from=1-3, to=2-5]
	\arrow[no head, from=2-5, to=3-2]
	\arrow[no head, from=2-5, to=3-8]
	\arrow[no head, from=3-2, to=4-1]
	\arrow[no head, from=3-2, to=4-3]
	\arrow[no head, from=3-8, to=4-6]
	\arrow[no head, from=3-8, to=4-10]
	\arrow[no head, from=4-6, to=5-5]
	\arrow[no head, from=4-6, to=5-7]
	\arrow[no head, from=4-10, to=5-9]
	\arrow[no head, from=4-10, to=5-11]
	\arrow[no head, from=5-9, to=6-8]
	\arrow[no head, from=5-9, to=6-10]
\end{tikzcd}\]
We may also assume that all crossings are a generic crossing which we will refer to as an $X-$crossing, as there is exactly one bit in the crossing rule which can possibly be used with 2-stranded SCA patterns. Thus, if a 2-stranded glider or repeating SCA pattern $g$ with a crossing exists, then a 2-stranded glider or SCA pattern identical to $g$ except with the crossings replaced by S-crossing exists and a 2-stranded glider or repeating SCA pattern identical to $g$ except with the crossings replaced by Z-crossing exists. There exists no 2-stranded SCA pattern with both S- and Z-crossings, as that would violate bit 4 of the crossing rule. Notice that every 2-stranded repeating pattern must be a shift of a 2-stranded repeating pattern found via this process.

\hfill\break
\textbf{Case 1}: We start with $\delta_1 = [r_1^{(1)}, l_2^{(2)}]$.
\begin{center}
    \includegraphics[scale=0.95]{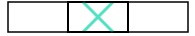}
\end{center}
\hfill\break
\textit{Part 1:} We have $\frac{S}{TE}, \frac{S}{ET}$, so $\delta_2 = [n_0, s_1^{(1)}][s_2^{(2)}, n_3]$:
\begin{center}
    \includegraphics[scale=0.6]{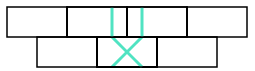}
\end{center}
so consider $a_1 = \frac{T}{SS}$ and $a_2 = \frac{S}{SS}$. 

\hfill\break
$\bm{a_1 = \frac{T}{SS}}$:

$\delta_3 = \delta_1$, so we have a glider $g_1 = ([\delta_1, \delta_2])^\infty$ with speed 0 and turning rule $1XXX00X0X$. 

\hfill\break
$\bm{a_2 = \frac{S}{SS}}$

$ \delta_3 = [s_1^{(1)}, s_2^{(2)}]$
\begin{center}
    \includegraphics[scale=0.6]{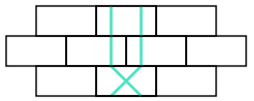}
\end{center}
so we need to consider $b_1 = \frac{T}{SE}, b_2 = \frac{S}{SE}$.

\hfill\break
$\bm{b_1 = \frac{T}{SE}}$

$\delta_4 = [n_0, l_1^{(1)}][r_{2}^{(2)}, n_3]$:
\begin{center}
    \includegraphics[scale=0.9]{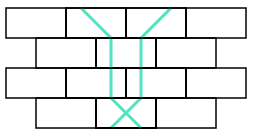}
\end{center}
and because we have $\frac{S}{TE}, \frac{S}{ET}$, $\delta_5 = [n_0, s_1^{(1)}][n_2, n_3][s_4^{(2)}, n_5], \delta_6 = [r_1^{(1)}, n_2][n_3, l_4^{(2)}]$:
\begin{center}
    \includegraphics[scale=0.6]{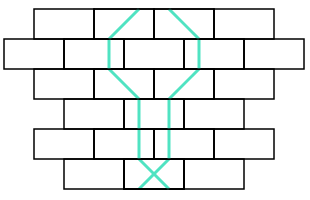}
\end{center}
so we must consider $c_1 = \frac{T}{TT}, c_2 = \frac{S}{TT}$.

\hfill\break
$\bm{c_1 = \frac{T}{TT}}$

$\delta_7 = \delta_1$, so we have $g_2 = ([\delta_1, \delta_2, \delta_3, \delta_4, \delta_5])^\infty$ with speed 0 and turning rule $01X100X01$. 

\hfill\break
$\bm{c_2 = \frac{S}{TT}}$

$\delta_7 = [s_1^{(1)}, s_2^{(2)}] = \delta_3$
\begin{center}
    \includegraphics[scale=0.6]{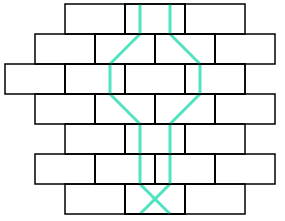}
\end{center}
so we have $g_3 = ([\delta_3, \delta_4, \delta_5, \delta_6])^\infty$ with speed 0 and turning rule $X1X100X00$. 

\hfill\break
$\bm{b_2 = \frac{S}{SE}}$

$\delta_1 = [r_1^{(1)}, l_2^{(2)}]$. We have $\frac{S}{TE}, \frac{S}{ET}$, so $\delta_2 = [n_0, s_1^{(1)}][s_2^{(2)}, n_3]$
\begin{center}
    \includegraphics[scale=0.6]{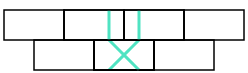}
\end{center}
and we have $a_2$, so $\delta_3 = [s_1^{(1)}, s_2^{(2)}]$
\begin{center}
    \includegraphics[scale=0.6]{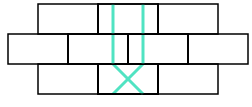}
\end{center}
Because we have $b_2$, $\delta_4 = [n_0, l_1^{(1)}][s_2^{(2)}, n_3]$
\begin{center}
    \includegraphics[scale=0.6]{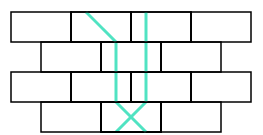}
\end{center}
We now need to consider $d_1 = \frac{T}{TS}, d_2 = \frac{S}{TE}$. 

\hfill\break
$\bm{d_1 = \frac{T}{TS}}$

We have $\delta_5 = [n_0, s_1^{(1)}][n_2, r_3^{(2)}]$
\begin{center}
    \includegraphics[scale=0.6]{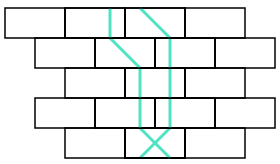}
\end{center}
so we must consider $e_1 = \frac{T}{ST}, e_2 = \frac{S}{ST}$. 

\hfill\break
$\bm{e_1 = \frac{T}{ST}}$

We get $\delta_6 = [r_1^{(1)}, l_2^{(2)}] = \delta_1$
\begin{center}
    \includegraphics[scale=0.6]{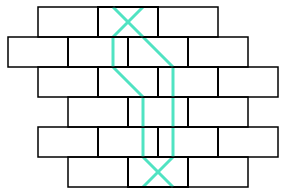}
\end{center}
so $g_4 = ([\delta_1, \delta_2, \delta_3, \delta_4, \delta_5])^\infty$ with speed $\frac{1}{5}$ and turning rule $00110010X$. 

\hfill\break
$\bm{e_2 = \frac{S}{ST}}$
We have $a_2, b_2, d_1, e_2$, so $\delta_6 = [s_1^{(1)}, s_2^{(2)}] = \delta_3$
\begin{center}
    \includegraphics[scale=0.6]{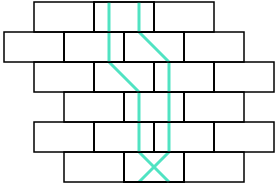}
\end{center}
and $g_5 = ([\delta_3, \delta_4, \delta_5])^\infty$ with speed $\frac{1}{3}$, initial condition $[s_1^{(1)}, s_2^{(2)}]$ and turning rule $X001001XX$.

\hfill\break
$\bm{d_2 = \frac{S}{TS}}$

We now consider $a_2, b_2, d_2$, where $\delta_1 = [r_1^{(1)}, l_2^{(2)}], \delta_2 = [n_0, s_1^{(1)}][s_2^{(2)}, n_3], \delta_3 = [s_1^{(1)}, s_2^{(2)}], \delta_4 = [n_0, l_1^{(1)}][s_2^{(2)}, n_3], \delta_5 = [n_0, s_1^{(1)}][n_2, s_3^{(2)}], \delta_6 = [s_1^{(1)}, n_2][s_3^{(2)}, n_4],$$ \delta_7 = [n_0, l_1^{(1)}][n_2, s_3^{(2)}],$\\
$ \delta_8 = [n_0, s_1^{(1)}][n_2, n_3][s_4^{(2)}, n_5], \delta_9 = [s_1^{(1)}, n_2][n_3, l_4^{(2)}]$
\begin{center}
    \includegraphics[scale=0.6]{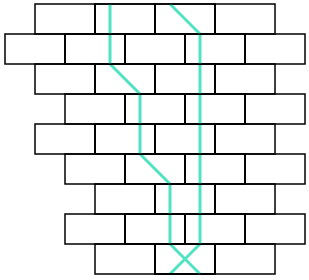}
\end{center}
Notice that we now need to consider $f_1 = \frac{T}{ST}, f_2 = \frac{S}{ST}$.

\hfill\break
$\bm{f_1 = \frac{T}{ST}}$

 $\delta_{10} = [n_0, l_1^{(1)}][n_2, l_3^{(2)}]$
 \begin{center}
    \includegraphics[scale=0.6]{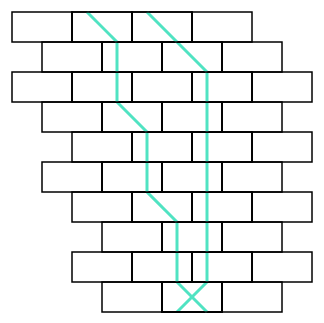}
\end{center}
so we must consider $h_1 = \frac{S}{TT}, h_2 = \frac{T}{TT}$. 

\hfill\break
$\bm{h_1 = \frac{S}{TT}}$

$a_2, b_2, d_2, f_1$ gives $\delta_{11} = [n_0, s_1^{(1)}][n_2, s_3^{(2)}] = \delta_5$
\begin{center}
    \includegraphics[scale=0.6]{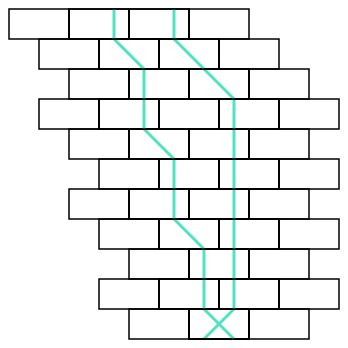}
\end{center}
for a glider $g_6 = ([\delta_5, \delta_6, \delta_7, \delta_8, \delta_9, \delta_{10}])^\infty$ with speed $\frac{2}{6},$ initial condition $[n_0, s_1^{(1)}][n_2, s_3^{(2)}]$ and turning rule $001100XX0$. 

\hfill\break
$\bm{h_2 = \frac{T}{TT}}$

We now consider $a_2, b_2, d_2, g_1, h_2$ in which $\delta_i$ is the same as in $a_2, b_2, d_2, g_1, h_1$ for $1\leq i\leq 10$, but $\delta_{11} = [n_0, s_1^{(1)}][n_2, l_3^{(2)}]$ and $\delta_{12} = [r_1^{(1)}, l_2^{(2)}] = \delta_1$
\begin{center}
    \includegraphics[scale=0.6]{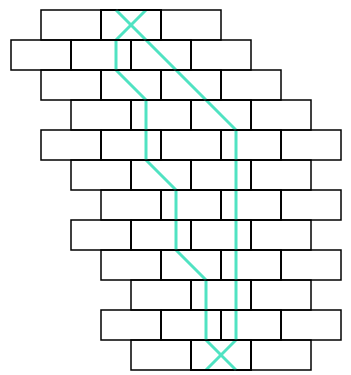}
\end{center}
for a glider $g_7 = ([\delta_1, \delta_2, \delta_3, \delta_4, \delta_5, \delta_6, \delta_7, \delta_8, \delta_9, \delta_{10}, \delta_{11}])^\infty$ with speed $\frac{3}{11}$, initial condition $[r_1^{(1)}, l_2^{(2)}]$, and turning rule $001100001$.

\hfill\break
$\bm{f_2 = \frac{S}{ST}}$

We now must consider $a_2, b_2, d_2, f_2$. We get $\delta_i$ the same as in the case with rules $a_2, b_2, d_2, g_1, h_2$ for $1\leq i\leq 9$, but $\delta_{10} = [n_0, l_1^{(1)}][n_2, s_3^{(2)}] = \delta_{7}$
\begin{center}
    \includegraphics[scale=0.6]{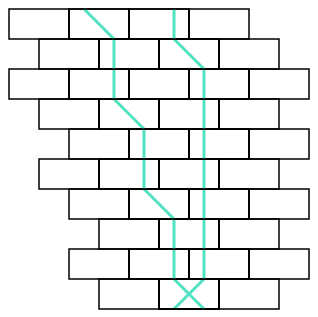}
\end{center}
for a glider $g_8 = ([[n_0, l_1^{(1)}][n_2, s_3^{(2)}], $$[n_0, s_1^{(1)}][n_2, n_3][s_4^{(2)}, n_5],$$ [s_1^{(1)}, n_2][n_3, l_4^{(2)}]])^\infty$
\hfill\break
with speed $\frac{1}{3}$, initial condition $[n_0, l_1^{(1)}][n_2, s_3^{(2)}]$, and turning rule $X00100XXX$.

\hfill\break
\textit{Part 2:}

Now we enter the second part of this case, where we instead assume $\frac{S}{ES}$. We have $\delta_2 = [n_0, s_1^{(1)}][s_2^{(2)}, n_3]$, so we must consider $a_1 = \frac{T}{SS}, a_2 = \frac{S}{SS}$. 

\hfill\break
$\bm{a_1 = \frac{T}{SS}}$

With $a_1, \delta_3 = \delta_1$, which is $g_1$. 

\hfill\break
$\bm{a_2 = \frac{S}{SS}}$

$\delta_3 = [s_1^{(1)}, s_2^{(2)}]$
\begin{center}
    \includegraphics[scale=0.6]{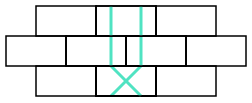}
\end{center}
so we must consider $b_1 = \frac{T}{SE}, b_2 = \frac{S}{SE}$. 

\hfill\break
$\bm{b_1 = \frac{T}{SE}}$

With $a_2, b_1, \delta_4 = [n_0, s_1^{(1)}][r_2^{(2)}, n_3]$
\begin{center}
    \includegraphics[scale=0.6]{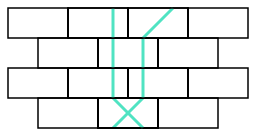}
\end{center}
We must consider $c_1 = \frac{S}{ST}, c_2 = \frac{T}{ST}$. 

\hfill\break
$\bm{c_1 = \frac{S}{ST}}$

For $a_2, b_1, c_1,$ $\delta_5 = [s_1^{(1)}, n_2][s_3^{(2)}, n_4], \delta_6 = [n_0, s_1^{(1)}][n_2, s_3^{(2)}], \delta_7 = [s_1^{(1)}, n_2][r_3^{(2)}, n_4],$ $\delta_8 = [n_0, s_1^{(1)}][n_2, n_3][s_4^{(2)}, n_5], \delta_9 = [r_1^{(1)}, n_2][n_3, s_4^{(2)}]$
\begin{center}
    \includegraphics[scale=0.6]{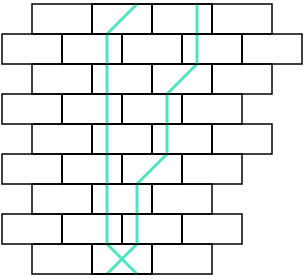}
\end{center}
so we must consider $d_1 = \frac{T}{TS}, d_2 = \frac{S}{TS}$.

\hfill\break
$\bm{d_1 = \frac{T}{TS}}$

We get $\delta_{10} = [r_1^{(1)}, n_2][r_3^{(2)}, n_4]$
\begin{center}
    \includegraphics[scale=0.6]{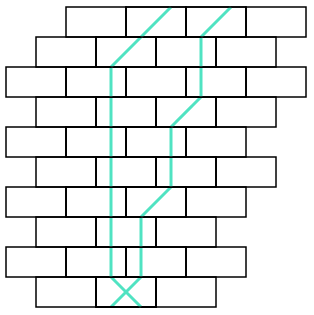}
\end{center}
so consider $e_1 = \frac{T}{TT}, e_2 = \frac{S}{TT}$.

\hfill\break
$\bm{e_1 = \frac{T}{TT}}$

 $\delta_{11} = [r_1^{(1)}, n_2][s_3^{(2)}, n_4], \delta_{12} = \delta_1$
 \begin{center}
    \includegraphics[scale=0.6]{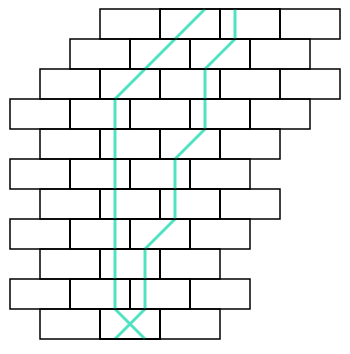}
\end{center}
 Thus, $g_9 = ([\delta_1, \delta_2, \cdots, \delta_{11}])^\infty$ is a glider with speed $-\frac{3}{11}$ and turning rule $010000101$. 
 
 \hfill\break
 $\bm{e_2 = \frac{S}{TT}}$

 $\delta_i$ is identical to the $e_1$ case for $1\leq i\leq 10$ but $\delta_{11} = \delta_5$
 \begin{center}
    \includegraphics[scale=0.6]{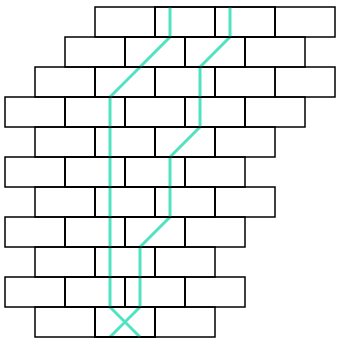}
\end{center}
so $g_{10} = ([\delta_5, \delta_6,\cdots, \delta_{10}])^\infty$ is a glider with speed $-\frac{2}{6}$ and turning rule $01X00X100$. 

\hfill\break
$\bm{d_2 = \frac{S}{TS}}$

$\delta_i$ the same as in the $e_2$ case for $1\leq i \leq 9$ and $\delta_{10} = \delta_7$
\begin{center}
    \includegraphics[scale=0.6]{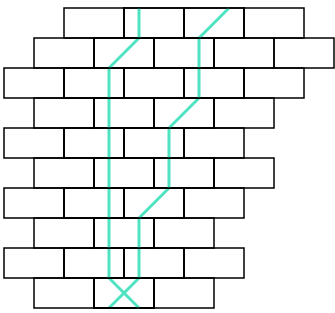}
\end{center}
so $g_{11} = ([\delta_7, \delta_8, \delta_9])^\infty$ is a glider with speed $-\frac{1}{3}$ and turning rule $X1X00X00X$. 

\hfill\break
$\bm{c_2 = \frac{T}{ST}}$

$\delta_i$ is the same as in the $d_2$ case for $1\leq i\leq 4$ but $\delta_5 = [r_1^{(1)}, n_2][s_3^{(2)}, n_4]$
\begin{center}
    \includegraphics[scale=0.6]{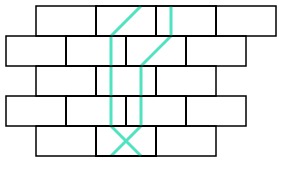}
\end{center}
so we must consider $f_1 = \frac{T}{TS}, f_2 = \frac{S}{TS}$.

\hfill\break
$\bm{f_1 = \frac{T}{TS}}$

We have $\delta_6 = \delta_1$ 
\begin{center}
    \includegraphics[scale=0.6]{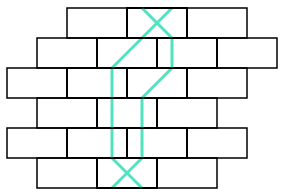}
\end{center}
for $g_{12} = ([\delta_1,\cdots,\delta_5])^\infty$ a glider with speed $-\frac{1}{5}$ and turning rule $01100010X$. 

\hfill\break
$\bm{f_2 = \frac{S}{TS}}$

$\delta_i$ is the same as in the $f_1$ case for $1\leq i\leq 5$ but $\delta_6 = [s_1^{(1)}, s_2^{(2)}] = \delta_3$
\begin{center}
    \includegraphics[scale=0.6]{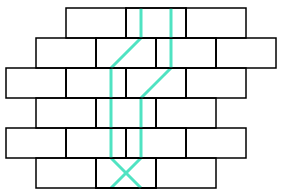}
\end{center}
so the glider $g_{13} = ([\delta_3, \delta_4, \delta_5])^\infty$ has speed $-\frac{1}{3}$ and turning rule $X1100X00X$. 

\hfill\break
$\bm{b_2 = \frac{S}{SE}}$

$\delta_2 = [n_0, s_1^{(1)}][s_2^{(2)}, n_3], \delta_3 = [s_1^{(1)}, s_2^{(2)}]$
\begin{center}
    \includegraphics[scale=0.6]{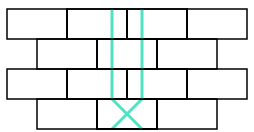}
\end{center}
and $\delta_4 = \delta_2$ for a SCA pattern $g_{14} = ([\delta_2, \delta_3])^\infty$ with speed 0 and turning rule $00X00XXXX$.\\

\hfill\break
\textbf{Case 2}: We now turn to patterns beginning with the initial condition $\delta_1 = [s_1^{(1)}, s_2^{(2)}]$. 
\begin{center}
    \includegraphics[scale=0.6]{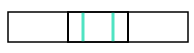}
\end{center}
\hfill\break
\textit{Part 1:} $\frac{T}{ES}$. 

We immediately must consider $a_1 = \frac{T}{SE}, a_2 = \frac{S}{SE}$.

\hfill\break
$\bm{a_1 = \frac{T}{SE}}$

$\delta_2 = [n_0, l_1^{(1)}][r_2^{(2)}, n_3], \delta_3 = [n_0, s_1^{(1)}][n_2, n_3][s_4^{(2)}, n_5], \delta_4 = [r_1^{(1)}, n_2][n_3, l_4^{(2)}]$
\begin{center}
    \includegraphics[scale=0.6]{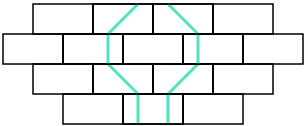}
\end{center}
We have two cases for $\delta_5: \delta_5 = [s_1^{(1)}, s_2^{(2)}]$ in which case we have the SCA pattern $g_3$, or $\delta_5 = [r_1^{(1)}, l_2^{(2)}]$, in which case see case 1. 

\hfill\break
$\bm{a_2 = \frac{S}{SE}}$

$\delta_1 = [s_1^{(1)}, s_2^{(2)}], \delta_2 = [n_0, l_1^{(1)}][s_2^{(2)}, n_3]$
\begin{center}
    \includegraphics[scale=0.6]{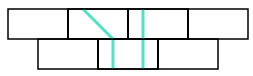}
\end{center}
so we must consider $b_1 = \frac{T}{TS}, b_2 = \frac{S}{TS}$.

\hfill\break
$\bm{b_1 = \frac{T}{TS}}$

$\delta_3 = [n_0, s_1^{(1)}][n_2, l_3^{(2)}]$
\begin{center}
    \includegraphics[scale=0.6]{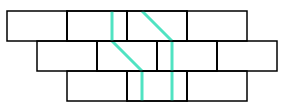}
\end{center}
Note that there are two options for $\delta_4: \delta_4 = [s_1^{(1)}, s_2^{(2)}]$ or $\delta_4 = [r_1^{(1)}, l_2^{(2)}]$. If we have $\delta_4: \delta_4 = [s_1^{(1)}, s_2^{(2)}]$, then we get the glider $g_5$,and if $\delta_4 = [r_1^{(1)}, l_2^{(2)}]$, see case 1.

\hfill\break
$\bm{b_2 = \frac{S}{TS}}$

We have $\delta_2 = [n_0, l_1^{(1)}][s_2^{(2)}, n_3], \delta_3 = [n_0, s_1^{(1)}][n_2, s_3^{(2)}]$
\begin{center}
    \includegraphics[scale=0.6]{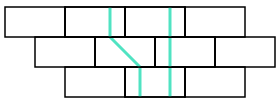}
\end{center}
so we need to consider $c_1 = \frac{T}{SS}, c_2 = \frac{S}{SS}$.

\hfill\break
$\bm{c_1 = \frac{T}{SS}}$

$\delta_4 = [r_1^{(1)}, n_2][s_3^{(2)}, n_4], \delta_5 = [s_1^{(1)}, s_2^{(2)}]$
\begin{center}
    \includegraphics[scale=0.6]{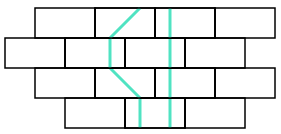}
\end{center}
so $g_{15} = ([[s_1^{(1)}, s_2^{(2)}], [n_0, l_1^{(1)}][r_2^{(2)}, n_3], [n_0, s_1^{(1)}][n_2, n_3][s_4^{(2)}, n_5], [r_1^{(1)}, n_2][n_3, l_4^{(2)}]])^\infty$ is a SCA pattern with speed 0, initial condition $[s_1^{(1)}, s_2^{(2)}]$, and turning rule $10X1000XX$. 

\hfill\break
$\bm{c_2 = \frac{S}{SS}}$

$\delta_2 = [n_0, l_1^{(1)}][s_2^{(2)}, n_3], \delta_3 = [n_0, s_1^{(1)}][n_2, s_3^{(2)}], \delta_4 = [s_1^{(1)}, n_2][s_3^{(2)}, n_4], \delta_5 = [n_0, l_1^{(1)}][n_2, s_3^{(2)}],$
 
$ \delta_6 = [n_0, s_1^{(1)}][n_2, n_3][s_4^{(2)}, n_5], \delta_7 = [s_1^{(1)}, n_2][n_3, l_4^{(2)}]$
\begin{center}
    \includegraphics[scale=0.6]{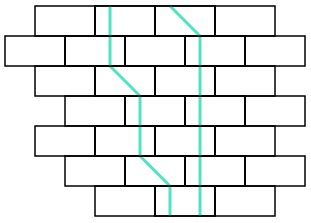}
\end{center}
We must now consider $d_1 = \frac{T}{ST}, d_2 = \frac{S}{ST}$. 

\hfill\break
$\bm{d_1 = \frac{T}{ST}}$

$\delta_8 = [n_0, l_1^{(1)}][n_2, l_3^{(2)}]$
\begin{center}
    \includegraphics[scale=0.6]{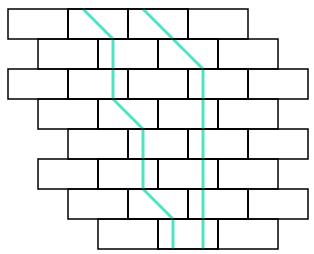}
\end{center}
and we need to consider $e_1 = \frac{T}{TT}, e_2 = \frac{S}{TT}$. 

\hfill\break
$\bm{e_1 = \frac{T}{TT}}$

We get $\delta_9 = [n_0, s_1^{(1)}][n_2, l_3^{(2)}]$ and $\delta_{10} = [r_1^{(1)}, l_2^{(2)}]$
\begin{center}
    \includegraphics[scale=0.6]{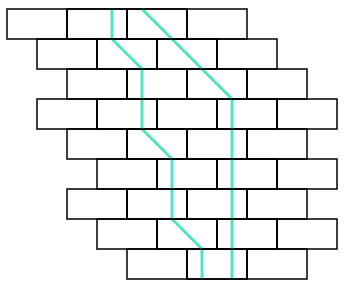}
\end{center}
so see case 1.

\hfill\break
$\bm{e_2 = \frac{S}{TT}}$

We have $\delta_i$ is the same as in the $e_1$ case for $1\leq i\leq 8$, but $\delta_9 = [n_0, s_1^{(1)}][n_2, s_3^{(2)}] = \delta_3$
\begin{center}
    \includegraphics[scale=0.6]{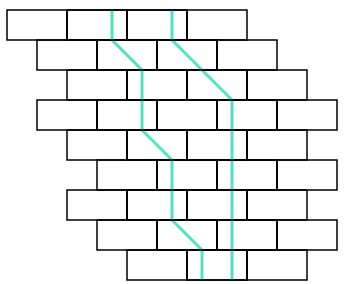}
\end{center}
so we have glider $g_6$. 

\hfill\break
$\bm{d_2 = \frac{S}{ST}}$

We get $\delta_i$ is the same as in the $e_1$ case for $1\leq i\leq 7$, but $\delta_8 = [n_0, l_1^{(1)}][n_2, s_3^{(2)}] = \delta_5$
\begin{center}
    \includegraphics[scale=0.6]{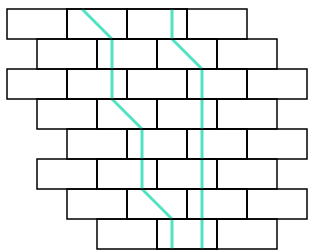}
\end{center}
Thus, this produces $g_8$. 

\hfill\break
\textit{Part 2:} $\frac{S}{ES}$. 

We first must consider $a_1 = \frac{T}{SE}, a_2 = \frac{S}{SE}$.

\hfill\break
$\bm{a_1 = \frac{T}{SE}}$

If we have $a_1, \delta_2 = [n_0, s_1^{(1)}][r_2^{(2)}, n_3]$
\begin{center}
    \includegraphics[scale=0.6]{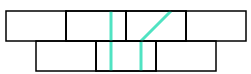}
\end{center}
so we must consider $b_1 = \frac{T}{ST}, b_2 = \frac{S}{ST}$. 

\hfill\break
$\bm{b_1 = \frac{T}{ST}}$

We get $\delta_3 = [r_1^{(1)}, n_2][s_3^{(2)}, n_4]$
\begin{center}
    \includegraphics[scale=0.6]{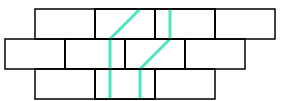}
\end{center}
We now must consider $c_1 = \frac{T}{TS}, c_2 = \frac{S}{TS}$.

\hfill\break
$\bm{c_1 = \frac{T}{TS}}$

$\delta_4 = [r_1^{(1)}, l_2^{(2)}]$, see case 1. 

\hfill\break
$\bm{c_2 = \frac{S}{TS}}$

$\delta_i$ the same as in the $c_1$ case for $1\leq i\leq 3$ with $\delta_4 = \delta_1$, which is the glider $g_{13}$. 

\hfill\break
$\bm{b_2 = \frac{S}{ST}}$

We have $\delta_2 = [n_0, s_1^{(1)}][n_2, s_3^{(2)}]$
\begin{center}
\includegraphics[scale=0.6]{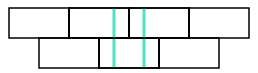}
\end{center}
so we have two options for $\delta_3$: $\delta_3 = \delta_1$, in which case we have a shift of $g_{14}$ or $\delta_3 = [r_1^{(1)}, l_2^{(2)}]$ in which case consider case 1. 

\hfill\break
$\bm{a_2 = \frac{S}{SE}}$

$\delta_2$ is the same as in the $a_1, b_2$ case, so the same results hold. Thus, we have completed the case where $\delta_1 = [s_1^{(1)}, s_2^{(2)}]$.\\

\hfill\break
\textbf{Case 3}: We now consider the case where $\delta_1 = [n_0, s_1^{(1)}][s_2^{(2)}, n_3]$
\begin{center}
    \includegraphics[scale=0.6]{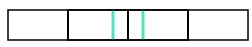}
\end{center}
Note that there are two options for $\delta_2$: $\delta_2 = [r_1^{(1)}, l_2^{(2)}]$, in which case see case 1, or $\delta_2 = [s_1^{(1)}, s_2^{(2)}]$, in which case see case 2. 
\textbf{Case 4} ($\delta_1 = [n_0, s_1^{(1)}][n_2, l_3^{(2)}]$),
\begin{center}
    \includegraphics[scale=0.6]{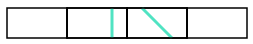}
\end{center}
\textbf{Case 5} ($\delta_1 = [r_1^{(1)}, n_2][s_3^{(2)}, n_4]$),
\begin{center}
    \includegraphics[scale=0.6]{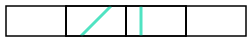}
\end{center}
and \textbf{Case 6} ($\delta_1 = [r_1^{(1)}, n_2][n_3, l_4^{(2)}]$)
\begin{center}
    \includegraphics[scale=0.6]{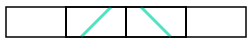}
\end{center}
follow similarly.\\

\hfill\break
\textbf{Case 7}: 
\hfill\break
We now consider the case where $\delta_1 = [s_1^{(1)}, n_2][s_3^{(2)}, n_4]$.
\begin{center}
    \includegraphics[scale=0.6]{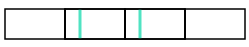}
\end{center}
\hfill\break
\textit{Part 1:}  Suppose $\frac{T}{ES}$. We must consider $a_1 = \frac{T}{SS}, a_2 = \frac{S}{SS}$

\hfill\break
$\bm{a_1 = \frac{T}{SS}}$

With $a_1, \delta_2 = [n_0, l_1^{(1)}][n_2, l_3^{(2)}]$
\begin{center}
    \includegraphics[scale=0.6]{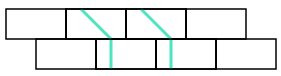}
\end{center}
As such, we need to consider $b_1 = \frac{T}{TT}, b_2 = \frac{S}{TT}$. 

\hfill\break
$\bm{b_1 = \frac{T}{TT}}$

$\delta_3 = [n_0, s_1^{(1)}][n_2, l_3^{(2)}]$
\begin{center}
    \includegraphics[scale=0.6]{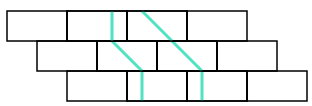}
\end{center}
Then, $\delta_4 = [r_1^{(1)}, l_2^{(2)}]$, in which case consider case 1, or $\delta_4 = [s_1^{(1)}, s_2^{(2)}]$, in which case consider case 2. 

\hfill\break
$\bm{b_2 = \frac{S}{TT}}$

$\delta_i$ is as in the $b_1$ case for $1\leq i\leq 2$, $\delta_3 = [n_0, s_1^{(1)}][n_2, s_3^{(2)}], \delta_4 = [r_1^{(1)}, n_2][s_3^{(2)}, n_4]$
\begin{center}
    \includegraphics[scale=0.6]{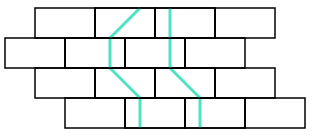}
\end{center}
There are two options for $\delta_5: \delta_5 = [r_1^{(1)}, l_2^{(2)}]$ in which case see case 1 or $\delta_5 = [s_1^{(1)}, s_2^{(2)}]$ in which case see case 2.

\hfill\break
$\bm{a_2 = \frac{S}{SS}}$

Now consider $a_2$. We have $\delta_2 = [n_0, l_1^{(1)}][n_2, s_3^{(2)}]$
\begin{center}
    \includegraphics[scale=0.6]{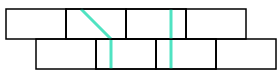}
\end{center}
so we must consider $c_1 = \frac{T}{SE}, c_2 = \frac{S}{SE}$.

\hfill\break
$\bm{c_1 = \frac{T}{SE}}$

$\delta_3 = [n_0, s_1^{(1)}][n_2, n_3][s_4^{(2)}, n_5]$ and $\delta_4 = [r_1^{(1)}, n_2][n_3, l_4^{(2)}]$
\begin{center}
    \includegraphics[scale=0.6]{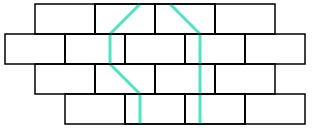}
\end{center}
so see case 6. 

\hfill\break
$\bm{c_2 = \frac{S}{SE}}$

We have $\delta_i$ the same as in the $c_1$ case for $1\leq i\leq 3$, and $\delta_4 = [s_1^{(1)}, n_2][n_3, l_4^{(2)}]$
\begin{center}
    \includegraphics[scale=0.6]{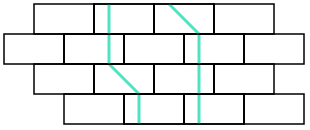}
\end{center}
so we must consider $d_1 = \frac{T}{ST}, d_2 = \frac{S}{ST}$. 

\hfill\break
$\bm{d_1 = \frac{T}{ST}}$

 $\delta_5 = [n_0, l_1^{(1)}][n_2, l_3^{(2)}]$
 \begin{center}
    \includegraphics[scale=0.6]{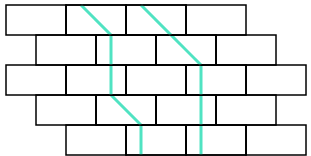}
\end{center}
 so we must consider $e_1 = \frac{T}{TT}, e_2 = \frac{S}{TT}$.
 
\hfill\break
$\bm{e_1 = \frac{T}{TT}}$

$\delta_6 = [n_0, s_1^{(1)}][n_2, l_3^{(2)}]$ and $\delta_7 = [r_1^{(1)}, l_2^{(2)}]$
\begin{center}
    \includegraphics[scale=0.6]{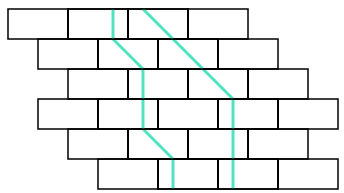}
\end{center}
so see case 1. 

\hfill\break
$\bm{e_2 = \frac{S}{TT}}$

$\delta_i$ is the same as in the $e_1$ case for $1\leq i\leq 5$. $\delta_6 = [n_0, s_1^{(1)}][n_2, s_3^{(2)}]$
\begin{center}
    \includegraphics[scale=0.6]{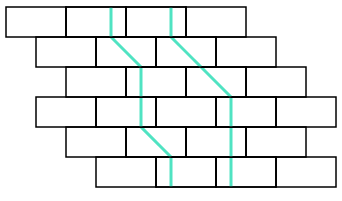}
\end{center}
so $\delta_7 = [s_1^{(1)}, n_2][s_3^{(2)}, n_4] = \delta_1$
\begin{center}
    \includegraphics[scale=0.6]{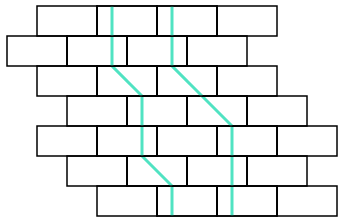}
\end{center}
which is a shift of the glider $g_6$. 

\hfill\break
$\bm{d_2 = \frac{S}{ST}}$

$\delta_i$ is the same as in the $e_2$ case for $1\leq i\leq 4$, but $\delta_5 = [n_0, l_1^{(1)}][n_2, s_3^{(2)}] = \delta_2$
\begin{center}
    \includegraphics[scale=0.6]{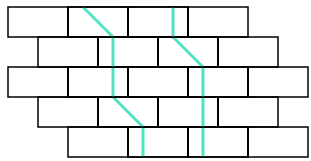}
\end{center}
for the glider $g_8$. 

\hfill\break
\textit{Part 2:} We now assume $\frac{S}{ES}$. Thus, we must consider $a_1 = \frac{T}{SS}, a_2 = \frac{S}{SS}$.

\hfill\break
$\bm{a_1 = \frac{T}{SS}}$

$\delta_2 = [n_0, s_1^{(1)}][n_2, l_3^{(2)}]$
\begin{center}
    \includegraphics[scale=0.6]{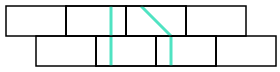}
\end{center}
so see case 4. 

\hfill\break
$\bm{a_2 = \frac{S}{SS}}$

$\delta_2 = [n_0, s_1^{(1)}][n_2, s_3^{(2)}]$
\begin{center}
    \includegraphics[scale=0.6]{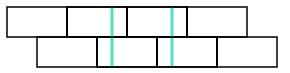}
\end{center}
so we must consider $b_1 = \frac{T}{SE}, b_2 = \frac{S}{SE}$. 

\hfill\break
$\bm{b_1 = \frac{T}{SE}}$

We have $\delta_3 = [s_1^{(1)}, n_2][r_3^{(2)}, n_4], \delta_4 = [n_0, s_1^{(1)}][n_2, n_3][s_4^{(2)}, n_5], \delta_5 = [r_1^{(1)}, n_2][n_3, s_4^{(2)}]$ 
\begin{center}
    \includegraphics[scale=0.6]{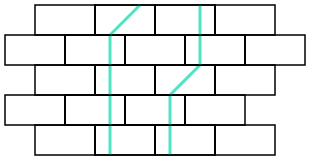}
\end{center}
so we must consider $c_1 = \frac{T}{TS}, c_2 = \frac{S}{TS}$. 

\hfill\break
$\bm{c_1 = \frac{T}{TS}}$

$\delta_6 = [r_1^{(1)}, n_2][r_3^{(2)}, n_4]$ 
\begin{center}
    \includegraphics[scale=0.6]{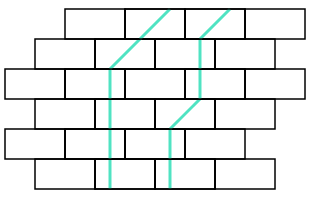}
\end{center}
so consider $d_1 = \frac{T}{TT}, d_2 = \frac{S}{TT}$.

\hfill\break
$\bm{d_1 = \frac{T}{TT}}$

We have $\delta_7 = [r_1^{(1)}, n_2][s_3^{(2)}, n_4]$
\begin{center}
    \includegraphics[scale=0.6]{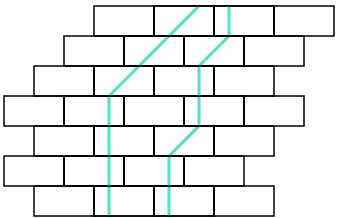}
\end{center}
so see case 5. 

\hfill\break
$\bm{d_2 = \frac{S}{TT}}$

$\delta_i$ is the same as in the $d_1$ case for $1\leq i\leq 6, \delta_7 = \delta_1$
\begin{center}
    \includegraphics[scale=0.6]{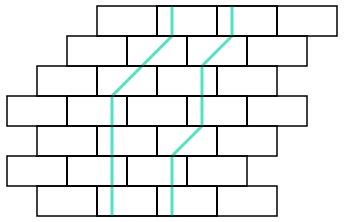}
\end{center}
for a glider $g_{16} = ([\delta_1, \cdots, \delta_6])^\infty$ with speed $-\frac{2}{6}$ and turning rule $01X00X100$. 

\hfill\break
$\bm{c_2 = \frac{S}{TS}}$

$\delta_i$ is the same as in the $d_2$ case for $1\leq i\leq 5$ but $\delta_6 = \delta_3$
\begin{center}
    \includegraphics[scale=0.6]{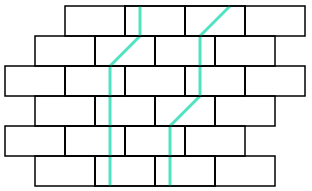}
\end{center}
for the glider $g_{11}$.

\hfill\break
$\bm{b_2 = \frac{S}{SE}}$

$\delta_2= [n_0, s_1^{(1)}][n_2, s_3^{(2)}], \delta_3 =\delta_1$
\begin{center}
    \includegraphics[scale=0.6]{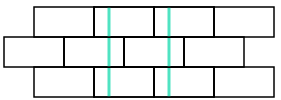}
\end{center}
for an SCA pattern $g_{17} = ([\delta_1, \delta_2])^\infty$ with speed 0 and turning rule $00X00XXXX$. This concludes case 7.\\

\hfill\break
\textbf{Case 8}: 
\hfill\break
We consider the case where $\delta_1 = [n_0, s_1^{(1)}][n_2, s_3^{(2)}]$
\begin{center}
    \includegraphics[scale=0.6]{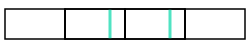}
\end{center}
\hfill\break
\textit{Part 1:} Assume $\frac{T}{ES}$. We immediately must consider $a_1 = \frac{T}{SE}, a_2 = \frac{S}{SE}$ and $b_1 = \frac{T}{SS}, b_2 = \frac{S}{SS}$.

\hfill\break
$\bm{a_1 = \frac{T}{SE}, b_1 = \frac{T}{SS}}$

$\delta_2 = [r_1^{(1)}, n_2][r_3^{(2)}, n_4]$
\begin{center}
    \includegraphics[scale=0.6]{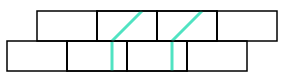}
\end{center}
so we must consider $c_1 = \frac{T}{TT}, c_2 = \frac{S}{TT}$.

\hfill\break
$\bm{c_1 = \frac{T}{TT}}$

We have $\delta_3 = [r_1^{(1)}, n_2][s_3^{(2)}, n_4]$
\begin{center}
    \includegraphics[scale=0.6]{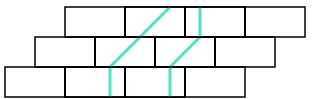}
\end{center}
so consider case 5. 

\hfill\break
$\bm{c_2 = \frac{S}{TT}}$

$\delta_i$ is the same as in the $c_2$ case for $1\leq i\leq 2$ and $\delta_3 = [s_1^{(1)}, n_2][s_3^{(1)}, n_4]$
\begin{center}
    \includegraphics[scale=0.6]{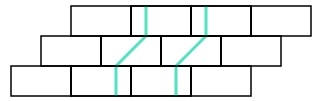}
\end{center}
so consider case 7.

\hfill\break
$\bm{a_2 = \frac{S}{SE}, b_1 = \frac{T}{SS}}$

$\delta_2 = [r_1^{(1)}, n_2][s_3^{(1)}, n_4]$ 
\begin{center}
    \includegraphics[scale=0.6]{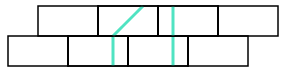}
\end{center}
so consider case 5.

\hfill\break
$\bm{a_1 = \frac{T}{SE}, b_2 = \frac{S}{SS}}$

$\delta_2 = [s_1^{(1)}, n_2][r_3^{(2)}, n_4], \delta_3 = [n_0, l_1^{(1)}][n_2, n_3][s_4^{(2)}, n_5], \delta_4 = [n_0, s_1^{(1)}][n_2, n_3][n_4, l_5^{(2)}], \delta_5 = [r_1^{(1)}, n_2][n_3, s_4^{(2)}]$
\begin{center}
    \includegraphics[scale=0.6]{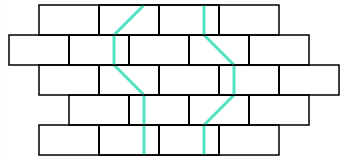}
\end{center}
so we must consider $d_1 = \frac{S}{TS}$ and $d_2 = \frac{T}{TS}$.

\hfill\break
$\bm{d_1 = \frac{S}{TS}}$

We get $\delta_6 = [s_1^{(1)}, n_2][r_3^{(2)}, n_4] = \delta_2$
\begin{center}
    \includegraphics[scale=0.6]{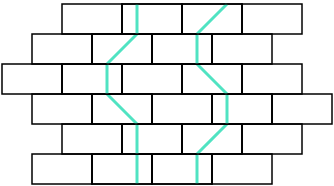}
\end{center}
for the speed 0 SCA pattern $g_{18} = ([[s_1^{(1)}, n_2][r_3^{(2)}, n_4], [n_0, l_1^{(1)}][n_2, n_3][s_4^{(2)}, n_5],$

$ [n_0, s_1^{(1)}][n_2, n_3][n_4, l_5^{(2)}],$

$[r_1^{(1)}, n_2][n_3, s_4^{(2)}]])^\infty$ with turning rule $X1X10000X$ and initial condition $[s_1^{(1)}, n_2][r_3^{(2)}, n_4]$.

\hfill\break
$\bm{d_2 = \frac{T}{TS}}$

$\delta_i$ is the same as in the $d_1$ case for $1\leq i\leq 5$, and $\delta_6 = [r_1^{(1)}, n_2][r_3^{(2)}, n_4]$
\begin{center}
    \includegraphics[scale=0.6]{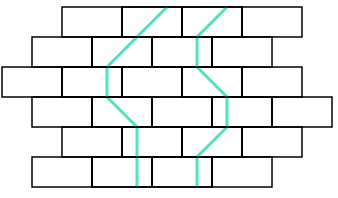}
\end{center}
so consider $e_1 = \frac{T}{TT}, e_2 = \frac{S}{TT}$.

\hfill\break
$\bm{e_1 = \frac{T}{TT}}$

 $\delta_7 = [r_1^{(1)}, n_2][s_3^{(2)}, n_4]$ 
 \begin{center}
    \includegraphics[scale=0.6]{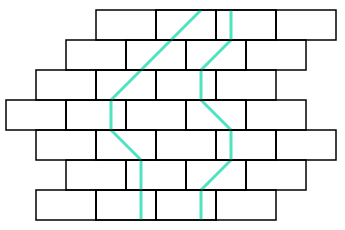}
\end{center}
 so consider case 5. 

\hfill\break
$\bm{e_2 = \frac{S}{TT}}$
 
$\delta_i$ is the same as in the $e_1$ case for $1\leq i\leq 6$ but $\delta_7 = [s_1^{(1)}, n_2][s_3^{(2)}, n_4]$
\begin{center}
    \includegraphics[scale=0.6]{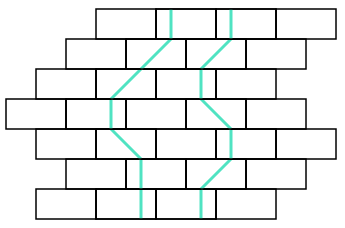}
\end{center}
so consider case 7.

\hfill\break
$\bm{a_2 = \frac{S}{SE}, b_2 = \frac{S}{SS}}$

 We get $\delta_2 = [s_1^{(1)}, n_2][s_3^{(2)}, n_4]$
 \begin{center}
    \includegraphics[scale=0.6]{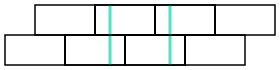}
\end{center}
 so consider case 7.

\hfill\break
\textit{Part 2:} We now suppose $\frac{S}{ES}$. We must immediately consider $a_1 = \frac{T}{SE}, a_2 = \frac{S}{SE}, b_1 = \frac{T}{SS}, b_2 = \frac{S}{SS}$.

\hfill\break
$\bm{a_1 = \frac{T}{SE}, b_1 = \frac{T}{SS}}$

$\delta_2 = [r_1^{(1)}, n_2][r_3^{(2)}, n_4]$
 \begin{center}
    \includegraphics[scale=0.6]{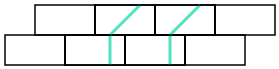}
\end{center}
so we need to consider $c_1 = \frac{T}{TT}, c_2 = \frac{S}{TT}$.

\hfill\break
$\bm{c_1 = \frac{T}{TT}}$

$\delta_3 = [r_1^{(1)}, n_2][s_3^{(2)}, n_4]$
 \begin{center}
    \includegraphics[scale=0.6]{scaStuff/sca731_23.png}
\end{center}
so consider case 5. 

\hfill\break
$\bm{c_2 = \frac{S}{TT}}$

$\delta_3 = [s_1^{(1)}, n_2][s_3^{(2)}, n_4]$
 \begin{center}
    \includegraphics[scale=0.6]{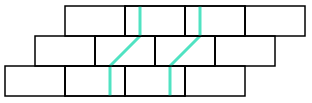}
\end{center}
so consider case 7. 

\hfill\break
$\bm{a_1 = \frac{T}{SE}, b_2 = \frac{S}{SS}}$

$\delta_2 = [s_1^{(1)}, n_2][r_3^{(2)}, n_4], \delta_3 = [n_0, s_1^{(1)}][n_2, n_3][s_4^{(2)}, n_5], \delta_4 = [r_1^{(1)}, n_2][n_3, s_4^{(2)}]$
 \begin{center}
    \includegraphics[scale=0.6]{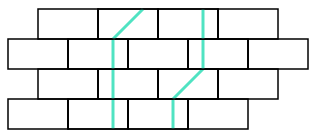}
\end{center}
so we must consider $d_1 = \frac{T}{TS}, d_2 = \frac{S}{TS}$. 

\hfill\break
$\bm{d_1 = \frac{T}{TS}}$

$\delta_5 = [r_1^{(1)}, n_2][r_3^{(2)}, n_4]$
 \begin{center}
    \includegraphics[scale=0.6]{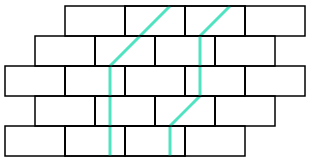}
\end{center}
so consider $e_1 = \frac{T}{TT}, e_2 = \frac{S}{TT}$.

\hfill\break
$\bm{e_1 = \frac{T}{TT}}$

We get $\delta_6 = [r_1^{(1)}, n_2][s_3^{(2)}, n_4]$
 \begin{center}
    \includegraphics[scale=0.6]{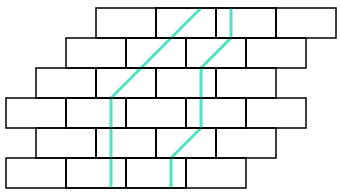}
\end{center}
so consider case 5.

\hfill\break
$\bm{e_2 = \frac{S}{TT}}$

$\delta_i$ is the same as in the $e_1$ case for $1\leq i\leq 5$ so $\delta_6 = [s_1^{(1)}, n_2][s_3^{(2)}, n_4]$
 \begin{center}
    \includegraphics[scale=0.6]{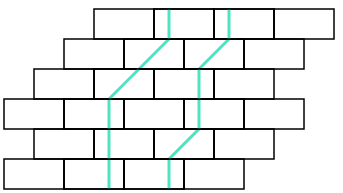}
\end{center}
See case 7. 

\hfill\break
$\bm{d_2 = \frac{S}{TS}}$

 $\delta_i$ the same as in the $e_2$ case for $1\leq i\leq 4$ but $\delta_5 = \delta_2$
  \begin{center}
    \includegraphics[scale=0.6]{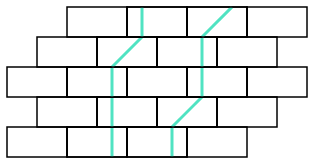}
\end{center}
 for $g_{11}$.
 \hfill\break
$\bm{a_2 = \frac{S}{SE}, b_1 = \frac{T}{SS}}$
 
 If we have $a_2, b_1$ we get $\delta_2 = [r_1^{(1)}, n_2][s_3^{(2)}, n_4]$
  \begin{center}
    \includegraphics[scale=0.6]{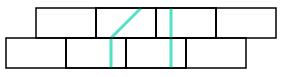}
\end{center}
 so see case 5. 

 \hfill\break
$\bm{a_1 = \frac{S}{SE}, b_2 = \frac{S}{SS}}$

This yields $\delta_2 = [s_1^{(1)}, n_2][s_3^{(2)}, n_4]$
 \begin{center}
    \includegraphics[scale=0.6]{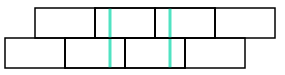}
\end{center}
so see case 7.\\

\hfill\break
\textbf{Case 9}: Let $\delta_1 = [s_1^{(1)}, n_2][n_3, s_4^{(2)}]$.
\begin{center}
    \includegraphics[scale=0.6]{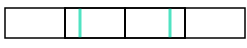}
\end{center}
\hfill\break
\textit{Part 1:} Suppose $\frac{T}{ES}$. We immediately must consider $a_1 = \frac{T}{SE}$ and $a_2 = \frac{S}{SE}$.

\hfill\break
$\bm{a_1 = \frac{T}{SE}}$

$\delta_2 = [n_0, l_1^{(1)}][n_2, n_3][r_4^{(2)}, n_5], \delta_3 = [n_0, s_1^{(1)}][n_2, n_3][n_4, n_5][s_6^{(1)}, n_7],$ 

$\delta_4 = [r_1^{(1)}, n_2][n_3, n_4][n_5, l_6^{(2)}], \delta_5 = [s_1^{(1)}, n_2][n_3, s_4^{(2)}] = \delta_1$
\begin{center}
    \includegraphics[scale=0.6]{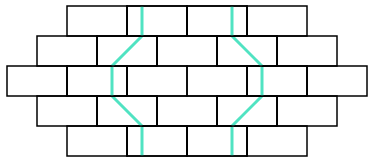}
\end{center}
This is the speed 0 SCA pattern $g_{19} = ([[s_1^{(1)}, n_2][n_3, s_4^{(2)}], [n_0, l_1^{(1)}][n_2, n_3][r_4^{(2)}, n_5],$\\
$ [n_0, s_1^{(1)}][n_2, n_3][n_4, n_5][s_6^{(1)}, n_7], $
$[r_1^{(1)}, n_2][n_3, n_4][n_5, l_6^{(2)}]])^\infty$ with turning rule $X1X100X0$ and initial condition $[s_1^{(1)}, n_2][n_3, s_4^{(2)}]$. 

\hfill\break
$\bm{a_2 = \frac{S}{SE}}$

We have $\delta_2 = [n_0, l_1^{(1)}][n_2, n_3][s_4^{(2)}, n_5], \delta_3 = [n_0, s_1][n_2, n_3][n_4,l_5^{(2)}], \delta_4 = [s_1^{(1)}, n_2][n_3, s_4^{(2)}] = \delta_1$
\begin{center}
    \includegraphics[scale=0.6]{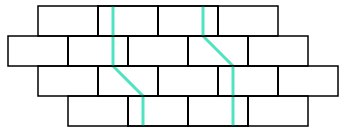}
\end{center}
Thus, we have the speed $\frac{1}{3}$ glider $g_{20} = ([[s_1^{(1)}, n_2][n_3, s_4^{(2)}],$\\
$ [n_0, l_1^{(1)}][n_2, n_3][s_4^{(2)}, n_5], [n_0, s_1][n_2, n_3][n_4,l_5^{(2)}]])^\infty$ with turning rule $X0X100XXX$ and initial condition $[s_1^{(1)}, n_2][n_3, s_4^{(2)}]$.

\hfill\break
\textit{Part 2:} If we have $\frac{S}{ES}$, we immediately must consider $a_1 = \frac{T}{SE}, a_2 = \frac{S}{SE}$.

\hfill\break
$\bm{a_1 = \frac{T}{SE}}$

For $a_1, \delta_2 = [n_0, s_1^{(1)}][n_2, n_3][r_4^{(2)}, n_5], \delta_3 = [r_1^{(1)}, n_2][n_3, n_4][s_5^{(2)}, n_6], \delta_4 = \delta_1$
\begin{center}
    \includegraphics[scale=0.6]{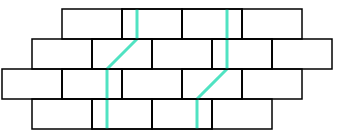}
\end{center}
for a glider $g_{21} = ([\delta_1,\cdots,\delta_3])^\infty$ with speed $-\frac{1}{3}$ and turning rule $X1X00XX0X$. 

\hfill\break
$\bm{a_2 = \frac{S}{SE}}$

$\delta_2 = [n_0, s_1^{(1)}][n_2, n_3][s_4^{(2)}, n_5], \delta_3 = \delta_1$
\begin{center}
    \includegraphics[scale=0.6]{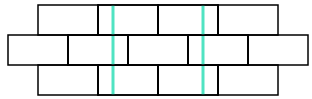}
\end{center}
so we have a SCA pattern $g_{22} = ([\delta_1, \delta_2])^\infty$ with speed 0 and turning rule $X0X00XXXX$.\\

\hfill\break
\textbf{Case 10:} Let $\delta_1 = [r_1^{(1)}, n_2][r_3^{(2)}, n_4]$.
\begin{center}
    \includegraphics[scale=0.95]{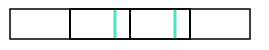}
\end{center}
We must consider $a_1 = \frac{T}{TT}, a_2 = \frac{S}{TT}$.

\hfill\break
$\bm{a_1 = \frac{T}{TT}}$

In the case of $a_1, \delta_2 = [r_1^{(1)}, n_2][s_3^{(2)}, n_4]$
\begin{center}
    \includegraphics[scale=0.6]{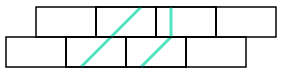}
\end{center}
so consider case 5. 

\hfill\break
$\bm{a_2 = \frac{S}{TT}}$

If we have $a_2, \delta_2 = [s_1^{(1)}, n_2][s_3^{(2)}, n_4]$
\begin{center}
    \includegraphics[scale=0.6]{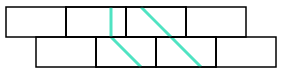}
\end{center}
so consider case 7.\\

\hfill\break
\textbf{Case 11}: Let $\delta_1 = [n_0, l_1^{(1)}][n_2, l_3^{(2)}]$. 
\begin{center}
    \includegraphics[scale=0.95]{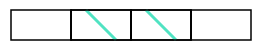}
\end{center}
We must consider $a_1 = \frac{T}{TT}, a_2 = \frac{S}{TT}$.

\hfill\break
$\bm{a_1 = \frac{T}{TT}}$

We have $\delta_2 = [n_0, s_1^{(1)}][n_2, l_3^{(2)}]$
\begin{center}
    \includegraphics[scale=0.6]{scaStuff/sca82_7.png}
\end{center}
so consider case 4. 

\hfill\break
$\bm{a_2 = \frac{S}{TT}}$

$\delta_2 = [n_0, s_1^{(1)}][n_2, s_3^{(2)}]$
\begin{center}
    \includegraphics[scale=0.6]{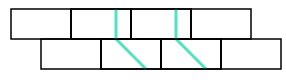}
\end{center}
 so consider case 8.\\

\hfill\break
\textbf{Case 12:} Let $\delta_1 = [n_0, l_1^{(1)}][r_2^{(2)}, n_3]$
\begin{center}
    \includegraphics[scale=0.6]{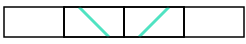}
\end{center}
\hfill\break
\textit{Part 1:} Assume $\frac{T}{ES}$.
Then $\delta_2 = [n_0, s_1^{(1)}][n_2, n_3][s_4^{(2)}, n_5]$.
\begin{center}
    \includegraphics[scale=0.6]{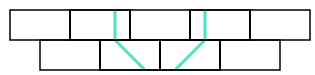}
\end{center}
We must immediately consider $a_1 = \frac{T}{SE}, a_2 = \frac{S}{SE}$. 

\hfill\break
$\bm{a_1 = \frac{T}{SE}}$

$\delta_3 = [r_1^{(1)}, n_2][n_3, l_4^{(2)}]$
\begin{center}
    \includegraphics[scale=0.6]{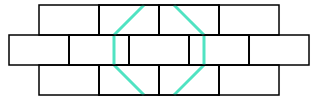}
\end{center}
so consider case 6. 

\hfill\break
$\bm{a_2 = \frac{S}{SE}}$

For $1\leq i\leq 2$, $\delta_i$ is as in the $a_1$ case and $\delta_3 = [s_1, n_2][n_3, l_4^{(2)}]$
\begin{center}
    \includegraphics[scale=0.6]{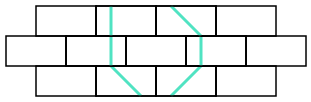}
\end{center}
so we must consider $b_1 = \frac{T}{ST}, b_2 = \frac{S}{ST}$. 

\hfill\break
$\bm{b_1 = \frac{T}{ST}}$

$\delta_4 = [n_0, l_1^{(1)}][n_2, l_3^{(2)}]$
\begin{center}
    \includegraphics[scale=0.6]{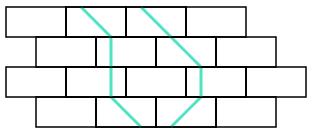}
\end{center}
so consider case 11. 

\hfill\break
$\bm{b_2 = \frac{S}{ST}}$

$\delta_4 = [n_0, l_1^{(1)}][n_2, s_3^{(2)}], \delta_5 = [n_0, s_1^{(1)}][n_2, n_3][s_4^{(2)}, n_5] = \delta_2$
\begin{center}
    \includegraphics[scale=0.6]{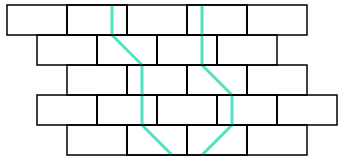}
\end{center}
so we have the glider $g_8$. 

\hfill\break
\textit{Part 2:} Assume $\frac{S}{ES}$. $\delta_2 = [n_0, s_1^{(1)}][n_2, n_3][s_4^{(2)}, n_5]$
\begin{center}
    \includegraphics[scale=0.6]{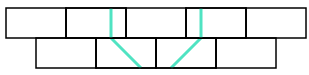}
\end{center}
We must immediately consider $a_1 = \frac{T}{SE}, a_2 = \frac{S}{SE}$.

\hfill\break
$\bm{a_1 = \frac{T}{SE}}$

$\delta_3 = [r_1^{(1)}, n_2][n_3, s_4^{(2)}]$
\begin{center}
    \includegraphics[scale=0.6]{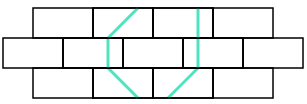}
\end{center}
so we must consider $b_1 = \frac{T}{TS}, b_2 = \frac{S}{TS}$.

\hfill\break
$\bm{b_1 = \frac{T}{TS}}$

$\delta_5 = [r_1^{(1)}, n_2][r_3^{(2)}, n_4]$ 
\begin{center}
    \includegraphics[scale=0.6]{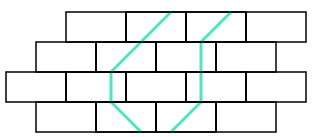}
\end{center}
see case 10. 

\hfill\break
$\bm{b_2 = \frac{S}{TS}}$

$\delta_i$ is the same as the $b_1$ case for $1\leq i\leq 3$, but $\delta_4 = [s_1^{(1)}, n_2][r_3^{(2)}, n_4], \delta_5 = \delta_2$
\begin{center}
    \includegraphics[scale=0.6]{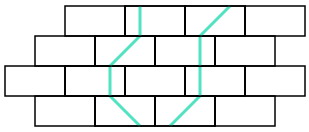}
\end{center}
for the glider $g_{11}$.

\hfill\break
$\bm{a_2 = \frac{S}{SE}}$

Finally, $\delta_2 = [n_0, s_1^{(1)}][n_2, n_3][s_4^{(2)}, n_5], \delta_3 = [s_1^{(1)}, n_2][n_3, s_4^{(2)}]$
\begin{center}
    \includegraphics[scale=0.6]{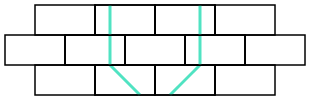}
\end{center}
so see case 9.\\

\hfill\break
\textbf{Case 13:} Let $\delta_1 = [s_1^{(1)}, n_2][n_3, l_4^{(2)}]$.
\begin{center}
    \includegraphics[scale=0.6]{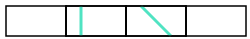}
\end{center}
\textit{Part 1:} Assume $\frac{T}{ES}$. We immediately must consider $a_1 = \frac{T}{ST}, a_2 = \frac{S}{ST}$.

\hfill\break
$\bm{a_1 = \frac{T}{ST}}$

$\delta_2 = [n_0, l_1^{(1)}][n_2, l_3^{(2)}]$
\begin{center}
    \includegraphics[scale=0.6]{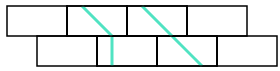}
\end{center}
so see case 11.

\hfill\break
$\bm{a_2 = \frac{S}{ST}}$

$\delta_2 = [n_0, l_1^{(1)}][n_2, s_3^{(2)}]$
\begin{center}
    \includegraphics[scale=0.6]{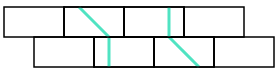}
\end{center}
We must consider $b_1 = \frac{T}{SE}, b_2 = \frac{S}{SE}$.

\hfill\break
$\bm{b_1 = \frac{T}{SE}}$

 We have $\delta_3 = [n_0, s_1^{(1)}][n_2, n_3][r_4^{(1)}, n_5], \delta_4 = [r_1^{(1)}, n_2][n_3, n_4][s_5^{(2)}, n_6], \delta_5 = [s_1^{(1)}, n_2][n_2, l_4^{(2)}] = \delta_1$
 \begin{center}
    \includegraphics[scale=0.6]{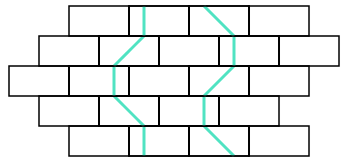}
\end{center}for a SCA pattern $g_{23} = ([[s_1^{(1)}, n_2][n_3, l_4^{(2)}], [n_0, l_1^{(1)}][n_2, s_3^{(2)}], $\\
$[n_0, s_1^{(1)}][n_2, n_3][r_4^{(1)}, n_5], [r_1^{(1)}, n_2][n_3, n_4][s_5^{(2)}, n_6]])^\infty$ with speed 0, turning rule $X10100X0X$, and initial condition $[s_1^{(1)}, n_2][n_3, l_4^{(2)}]$. 

\hfill\break
$\bm{b_2 = \frac{S}{SE}}$

$\delta_2 = [n_0, l_1^{(1)}][n_2, s_3^{(2)}], \delta_3 = [n_0, s_1^{(1)}][n_2, n_3][s_4^{(2)}, n_5], \delta_4 = [s_1^{(1)}, n_2][n_3, l_4^{(2)}] = \delta_1$ 
\begin{center}
    \includegraphics[scale=0.6]{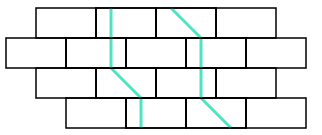}
\end{center}
for the glider $g_{24} = ([[s_1^{(1)}, n_2][n_3, l_4^{(2)}], [n_0, l_1^{(1)}][n_2, s_3^{(2)}],$\\
$[n_0, s_1^{(1)}][n_2, n_3][s_4^{(2)}, n_5]])^\infty$ that has speed $\frac{1}{3}$, turning rule $X00100XXX$, and initial condition $[s_1^{(1)}, n_2][n_3, l_4^{(2)}]$.

\hfill\break
\textit{Part 2:} Assume $\frac{S}{ES}$. We must consider $a_1 = \frac{T}{ST}, a_2 = \frac{S}{ST}$.

\hfill\break
$\bm{a_1 = \frac{T}{ST}}$

$\delta_2 = [n_0, s_1^{(1)}][n_2, l_3^{(2)}]$
\begin{center}
    \includegraphics[scale=0.6]{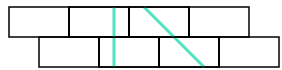}
\end{center}
 so see case 4.

 \hfill\break
 $\bm{a_2 = \frac{S}{ST}}$
 
$\delta_2 = [n_0, s_1^{(1)}][n_2, s_3^{(2)}]$
\begin{center}
    \includegraphics[scale=0.6]{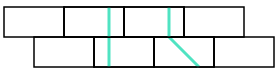}
\end{center}
 so see case 8.\\

\hfill\break
\textbf{Case 14}: Let $\delta_1 = [r_1^{(1)}, n_2][n_3, s_4^{(2)}]$
\begin{center}
    \includegraphics[scale=0.6]{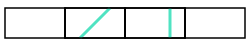}
\end{center}
\textit{Part 1:} Assume $\frac{T}{ES}$. We must consider $a_1 = \frac{T}{TS}, a_2 = \frac{S}{TS}, b_1 = \frac{T}{SE}, b_2 = \frac{S}{SE}$.

\hfill\break
$\bm{a_1 = \frac{T}{TS}, b_1 = \frac{T}{SE}}$

If we have $a_1, b_1$ we get $\delta_2 = [r_1^{(1)}, n_2][r_3^{(2)}, n_4]$
\begin{center}
    \includegraphics[scale=0.6]{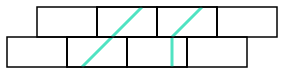}
\end{center}
See case 10. 

\hfill\break
$\bm{a_1 = \frac{T}{TS}, b_2 = \frac{S}{SE}}$

We have $\delta_2 = [r_1^{(1)}, n_2][s_3^{(2)}, n_4]$
\begin{center}
    \includegraphics[scale=0.6]{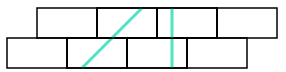}
\end{center}
so see case 5. 

\hfill\break
$\bm{a_2 = \frac{S}{TS}, b_2 = \frac{S}{SE}}$

$\delta_2 = [s_1^{(1)}, n_2][s_3^{(2)}, n_4]$
\begin{center}
    \includegraphics[scale=0.6]{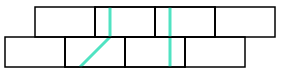}
\end{center}
See case 7. 

\hfill\break
$\bm{a_2 = \frac{S}{TS}, b_1 = \frac{T}{SE}}$

$\delta_2 = [s_1^{(1)}, n_2][r_3^{(1)}, n_4], \delta_3 = [n_0, l_1^{(1)}][n_2, n_3][s_4^{(2)}, n_5], \delta_4 = [n_0, s_1^{(1)}][n_2, n_3][n_4, l_5^{(2)}],$ \\
$ \delta_5 = [r_1^{(1)}, n_2][n_3, s_4^{(2)}] = \delta_1$
\begin{center}
    \includegraphics[scale=0.6]{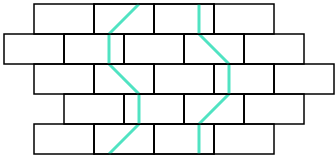}
\end{center}
$g_{25}=([r_1^{(1)}, n_2][n_3, s_4^{(2)}], [s_1^{(1)}, n_2][r_3^{(2)}, n_4], [n_0, l_1^{(1)}][n_2, n_3][s_4^{(2)}, n_5], [n_0, s_1^{(1)}][n_2, n_3][n_4, l_5^{(2)}]])^\infty$ with speed 0, turning rule $X1X10000X$, and initial condition $[r_1^{(1)}, n_2][n_3, s_4^{(2)}]$.

\hfill\break
\textit{Part 2:} Assume $\frac{T}{ES}$. We have $a_1 = \frac{T}{TS}, a_2 = \frac{S}{TS}, b_1 = \frac{T}{SE}, b_2 = \frac{S}{SE}$.

\hfill\break
$\bm{a_1 = \frac{T}{TS}, b_1 = \frac{T}{SE}}$

$\delta_2 = [r_1^{(1)}, n_2][r_3^{(2)}, n_4]$ 
\begin{center}
    \includegraphics[scale=0.6]{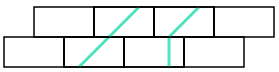}
\end{center}
See case 10. 

\hfill\break
$\bm{a_1 = \frac{T}{TS}}, b_2 = \frac{S}{SE}$

$\delta_2 = [r_1^{(1)}, n_2][s_3^{(2)}, n_4]$
\begin{center}
    \includegraphics[scale=0.6]{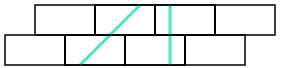}
\end{center}
See case 5. 

\hfill\break
$\bm{a_2 = \frac{S}{TS}, b_1 = \frac{T}{SE}}$

We have $\delta_2 = [s_1^{(1)}, n_2][r_3^{(2)}, n_4], \delta_3 = [n_0, s_1^{(1)}][n_2, n_3][s_4^{(2)}, n_5], \delta_4 = \delta_1$
\begin{center}
    \includegraphics[scale=0.6]{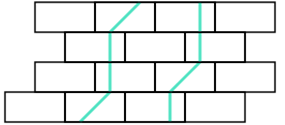}
\end{center}
which is the glider $g_{11}$.

\hfill\break
$\bm{a_2 = \frac{S}{TS}, b_2 = \frac{S}{SE}}$

$\delta_2 = [s_1^{(1)}, n_2][s_3^{(2)}, n_4]$
\begin{center}
    \includegraphics[scale=0.6]{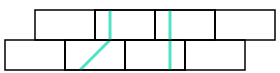}
\end{center}
See case 7.

\hfill\break
\textbf{Case 15:} Let $\delta_1 = [n_0, s_1^{(1)}][r_2^{(2)}, n_3]$
\begin{center}
    \includegraphics[scale=0.6]{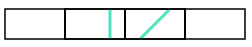}
\end{center}
We must consider $a_1 = \frac{T}{ST}, a_2 = \frac{S}{ST}$.

\hfill\break
$\bm{a_1 = \frac{T}{ST}}$

$\delta_2 = [r_1^{(1)}, n_2][s_3^{(2)}, n_4]$
\begin{center}
    \includegraphics[scale=0.6]{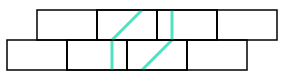}
\end{center}
See case 5. 

\hfill\break
$\bm{a_2 = \frac{S}{ST}}$

$\delta_2 = [s_1^{(1)}, n_2][s_3^{(2)}, n_4]$
\begin{center}
    \includegraphics[scale=0.6]{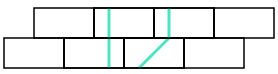}
\end{center}
See case 7.

\hfill\break
\textbf{Case 16:} Let $\delta_1 = [n_0, l_1^{(1)}][s_2^{(2)}, n_3]$. 
\begin{center}
    \includegraphics[scale=0.6]{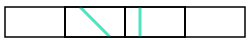}
\end{center}
We consider $a_1 = \frac{T}{TS}, a_2 = \frac{S}{TS}$.

\hfill\break
$\bm{a_1 = \frac{T}{TS}}$

$\delta_2 = [n_0, s_1^{(1)}][n_2, l_3^{(2)}]$
\begin{center}
    \includegraphics[scale=0.6]{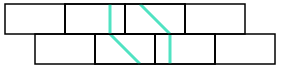}
\end{center}
so examine case 4. 

\hfill\break
$\bm{a_2 = \frac{S}{TS}}$

$\delta_2 = [n_0, s_1^{(1)}][n_2, s_3^{(2)}]$
\begin{center}
    \includegraphics[scale=0.6]{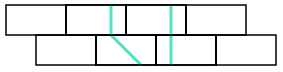}
\end{center}
See case 8.

\hfill\break
\textbf{Case 17:} Let $\delta_1 = [s_1^{(1)}, n_2][r_3^{(2)}, n_4]$
\begin{center}
    \includegraphics[scale=0.6]{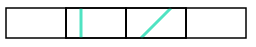}
\end{center}
\textit{Part 1:} Assume $\frac{T}{ES}$. We must consider $a_1 = \frac{T}{SE}, a_2 = \frac{S}{SE}$.

\hfill\break
$\bm{a_1 = \frac{T}{SE}}$

$\delta_2 = [n_0, l_1^{(1))}][n_2, n_3][s_4^{(2)}, n_5], \delta_3 = [n_0, s_1^{(1)}][n_2, n_3][n_4, l_5^{(2)}], \delta_4 = [r_1^{(1)}, n_2][n_3, s_4^{(2)}]$
\begin{center}
    \includegraphics[scale=0.6]{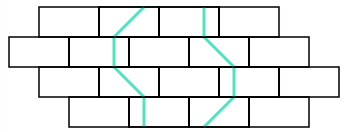}
\end{center}
See case 14. 

\hfill\break
$\bm{a_2 = \frac{S}{SE}}$

$\delta_2 = [n_0, l_1^{(1)}][n_2, n_3][s_4^{(1)}, n_5], \delta_3 = [n_0, s_1^{(1)}][n_2, n_3][n_4, l_5^{(2)}],$ $ \delta_4 = [s_1^{(1)}, n_2][n_3, s_4^{(2)}]$
\begin{center}
    \includegraphics[scale=0.6]{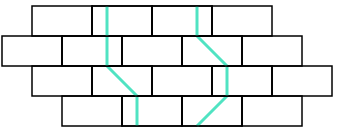}
\end{center}
so consider case 9.

\hfill\break
\textit{Part 2:} Assume $\frac{S}{ES}$. $\delta_2 = [n_0, s_1^{(1)}][n_2, n_3][s_4^{(2)}, n_5]$
\begin{center}
    \includegraphics[scale=0.6]{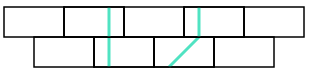}
\end{center}
Then, we must consider $a_1 = \frac{T}{SE}, a_2 = \frac{S}{SE}$.

\hfill\break
$\bm{a_1 = \frac{T}{SE}}$

$\delta_3 = [r_1^{(1)}, n_2][n_3, s_4^{(2)}]$
\begin{center}
    \includegraphics[scale=0.6]{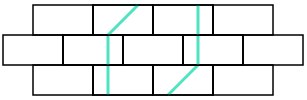}
\end{center}
so consider case 14. 

\hfill\break
$\bm{a_2 = \frac{S}{SE}}$

$ \delta_3 = [s_1^{(1)}, n_2][n_3, s_4^{(2)}]$
\begin{center}
    \includegraphics[scale=0.6]{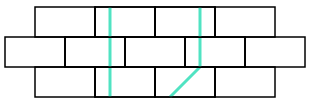}
\end{center}
See case 9.

\hfill\break
\textbf{Case 18:} Let $\delta_1 = [n_0, l_1^{(1)}][n_2, s_3^{(2)}]$
\begin{center}
    \includegraphics[scale=0.6]{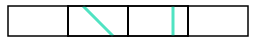}
\end{center}
\textit{Part 1:} Assume $\frac{T}{ES}$. We must consider $a_1 = \frac{T}{SE}, a_2 = \frac{S}{SE}$. 

\hfill\break
$\bm{a_1 = \frac{T}{SE}}$

$\delta_2 = [n_0, s_1^{(1)}][n_2, n_3][r_4^{(2)}, n_5], \delta_3 = [r_1^{(1)}, n_2][n_3, n_4][s_5^{(1)}, n_6]$ and $\delta_4 = [s_1^{(1)}, n_2][n_3, l_4^{(2)}]$
\begin{center}
    \includegraphics[scale=0.6]{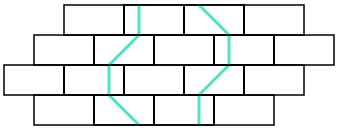}
\end{center}
so see case 13. 

\hfill\break
$\bm{a_2 = \frac{S}{SE}}$

We have $\delta_2 = [n_0, s_1^{(1)}][n_2, n_3][s_4^{(2)}, n_5], \delta_3 = [s_1^{(1)}, n_2][n_3, l_4^{(2)}]$
\begin{center}
    \includegraphics[scale=0.6]{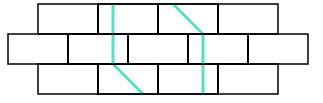}
\end{center}
See case 13.\\

\hfill\break
\textit{Part 2:} Assume $\frac{S}{ES}$. Then we must consider $a_1 = \frac{T}{SE}, a_2 = \frac{S}{SE}$.

\hfill\break
$\bm{a_1 = \frac{T}{SE}}$

$\delta_2 = [n_0, s_1^{(1)}][n_2, n_3][r_4^{(2)}, n_5], \delta_3 = [r_1^{(1)}, n_2][n_3, n_4][s_5^{(2)}, n_6], \delta_4 = [s_1^{(1)}, n_2][n_3, s_4^{(2)}]$
\begin{center}
    \includegraphics[scale=0.6]{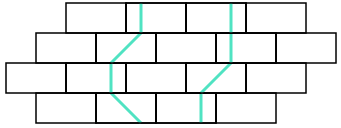}
\end{center}
Consider case 9. 

\hfill\break
$\bm{a_2 = \frac{S}{SE}}$

$\delta_2 = [n_0, s_1^{(1)}][n_2, n_3][s_4^{(2)}, n_5], \delta_3 = [s_1^{(1)}, n_2][n_3, s_4^{(2)}]$
\begin{center}
    \includegraphics[scale=0.6]{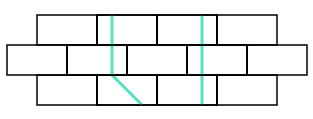}
\end{center}
so consider case 9.

\hfill\break
This classification proves the following result:
\begin{theorem}\thlabel{2StrandClassification}
    Let $g\in G_2$ be a SCA pattern. $g$ is a repeating SCA pattern if and only if it is a shift of one of:
    \begin{itemize}
        \item[1] $([[r_1^{(1)}, l_2^{(2)}], [n_0, s_1^{(1)}][s_2^{(2)}, n_3]])^\infty$, which has turning rule $1XXX00X0X$, crossing rule\\
        $XXXX1XXXX$, and speed 0.
        \begin{center}
    \includegraphics[scale=0.6]{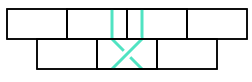}
        \end{center}
        \item[2] $([\overline{[r_1^{(1)}, l_2^{(2)}]}, [n_0, s_1^{(1)}][s_2^{(2)}, n_3]])^\infty$, which has turning rule $1XXX00X0X$, crossing rule\\
        $XXXX0XXXX$, and speed 0.
        \begin{center}
    \includegraphics[scale=0.6]{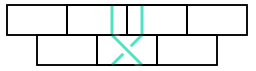}
        \end{center}
        \item[3] $([[r_1^{(1)}, l_2^{(2)}], [n_0, s_1^{(1)}][s_2^{(2)}, n_3], [s_1^{(1)}, s_2^{(2)}], [n_0, l_1^{(1)}][r_2^{(2)}, n_3], [n_0, s_1^{(1)}][n_2, n_3][s_4^{(2)}, n_5],$\\
        $ [r_1^{(1)}, n_2][n_3, l_4^{(2)}]])^\infty$, which has turning rule $01X100X01$, crossing rule $XXXX1XXXX$, and speed 0.
        \begin{center}
    \includegraphics[scale=0.6]{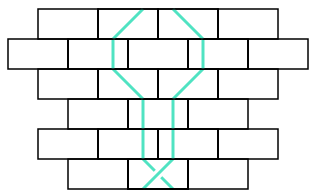}
        \end{center}
         \item[4] $([\overline{[r_1^{(1)}, l_2^{(2)}]}, [n_0, s_1^{(1)}][s_2^{(2)}, n_3], [s_1^{(1)}, s_2^{(2)}], [n_0, l_1^{(1)}][r_2^{(2)}, n_3], [n_0, s_1^{(1)}][n_2, n_3][s_4^{(2)}, n_5],$\\
        $ [r_1^{(1)}, n_2][n_3, l_4^{(2)}]])^\infty$, which has turning rule $01X100X01$, crossing rule $XXXX0XXXX$, and speed 0.
        \begin{center}
    \includegraphics[scale=0.6]{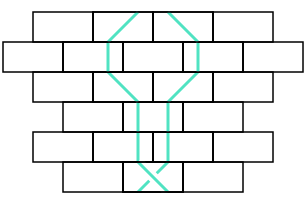}
        \end{center}
        \item[5] $([[s_1^{(1)}, s_2^{(2)}], [n_0, l_1^{(1)}][r_2^{(2)}, n_3], [n_0, s_1^{(1)}][n_2, n_3][s_4^{(2)}, n_5], [r_1^{(1)}, n_2][n_3, l_4^{(2)}]])^\infty$, which has turning rule $X1X100X00$, trivial crossing rule, and speed 0.
        \begin{center}
    \includegraphics[scale=0.6]{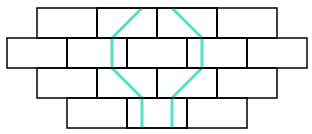}
        \end{center}
        \item[6] $([[s_1^{(1)}, s_2^{(2)}], [n_0, l_1^{(1)}][s_2^{(2)}, n_3], [n_0, s_1^{(1)}][n_2, s_3^{(2)}], [r_1^{(1)}, n_2][s_3^{(2)}, n_4]])^\infty$, which has turning rule $10X1000XX$, trivial crossing rule, and speed 0.
        \begin{center}
    \includegraphics[scale=0.6]{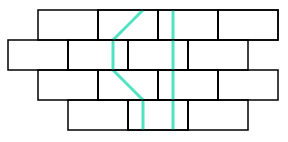}
        \end{center}
        \item[7] $([[s_1^{(1)}, s_2^{(2)}], [n_0, s_1^{(1)}][r_2^{(2)}, n_3], [s_1^{(1)}, n_2][s_3^{(2)}, n_4], [n_0, s_1^{(1)}][n_2, l_4^{(2)}]])^\infty$, which has turning rule $11000XX0X$, trivial crossing rule, and speed 0.
        \begin{center}
    \includegraphics[scale=0.6]{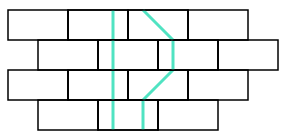}
        \end{center}
        \item[8] $([[s_1^{(1)}, n_2][r_3^{(2)}, n_4], [n_0, l_1^{(1)}][n_2, n_3][s_4^{(2)}, n_5], [n_0, s_1^{(1)}][n_2, n_3][n_4, l_5^{(2)}], [r_1^{(1)}, n_2][n_3, s_4^{(2)}]])^\infty$, which has turning rule $X1X10000X$, trivial crossing rule, and speed 0.
        \begin{center}
    \includegraphics[scale=0.6]{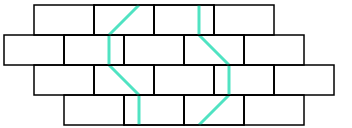}
        \end{center}
        \item[9] $([[s_1^{(1)}, n_2][n_3, l_4^{(2)}], [n_0, l_1^{(1)}][n_2, s_3^{(2)}], [n_0, s_1^{(1)}][n_2, n_3][r_4^{(2)}, n_5], [r_1^{(1)}, n_2][n_3, n_4][s_5^{(2)}, n_6]])^\infty$, which has turning rule $X10100X0X$, trivial crossing rule and speed 0.
        \begin{center}
    \includegraphics[scale=0.6]{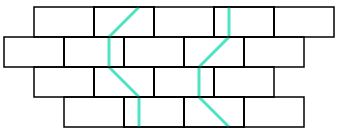}
        \end{center}
        \item[10] $([[s_1^{(1)}, n_2][n_3, s_4^{(2)}], [n_0, l_1^{(1)}][n_2, n_3][r_4^{(2)}, n_5], [n_0, s_1^{(1)}][n_2, n_3][n_4, n_5][s_6^{(2)}, n_7],$\\
        $ [r_1^{(1)}, n_2][n_3, n_4][n_5, l_6^{(2)}]])^\infty$, which has turning rule $X1X100X0$, trivial crossing rule, and speed 0.
        \begin{center}
    \includegraphics[scale=0.6]{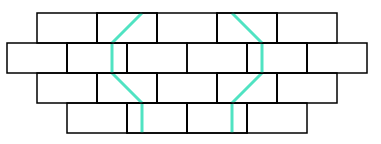}
        \end{center}
        \item[11] $([[n_0, s_1^{(1)}][n_2, s_3^{(2)}], [s_1^{(1)}, n_2][s_3^{(2)}, n_4]])^\infty$, which has turning rule $00X00XXXX$, trivial crossing rule and speed 0.
        \begin{center}
    \includegraphics[scale=0.6]{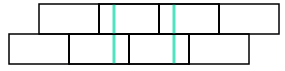}
        \end{center}
        \item[12] $([[n_0, s_1^{(1)}][s_2^{(2)}, n_3], [s_1^{(1)}, s_2^{(2)}]])^\infty$, which has turning rule $00X00XXXX$, trivial crossing rule and speed 0.
        \begin{center}
    \includegraphics[scale=0.6]{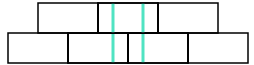}
        \end{center}
        \item[13] $([[s_1^{(1)}, n_2][n_3, s_4^{(2)}], [n_0, s_1^{(1)}][n_2, n_3][s_4^{(2)}, n_5]])^\infty$, which has turning rule $X0X00XXXX$, trivial crossing rule, and speed 0.
        \begin{center}
    \includegraphics[scale=0.6]{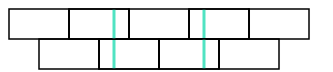}
        \end{center}
        \item[14] $([[r_1^{(1)}, l_2^{(2)}], [n_0, s_1^{(1)}][s_2^{(2)}, n_3], [s_1^{(1)}, s_2^{(2)}], [n_0, l_1^{(1)}][s_2^{(2)}, n_3], [n_0, s_1^{(1)}][n_2, l_3^{(2)}]])^\infty$, which has turning rule $00110010X$, crossing rule $XXXX1XXXX$, and speed $\frac{1}{5}$.
        \begin{center}
    \includegraphics[scale=0.6]{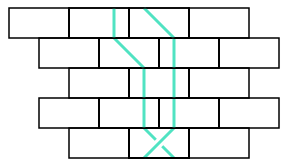}
        \end{center}
        \item[15] $([\overline{[r_1^{(1)}, l_2^{(2)}]}, [n_0, s_1^{(1)}][s_2^{(2)}, n_3], [s_1^{(1)}, s_2^{(2)}], [n_0, l_1^{(1)}][s_2^{(2)}, n_3], [n_0, s_1^{(1)}][n_2, l_3^{(2)}]])^\infty$, which has turning rule $00110010X$, crossing rule $XXXX0XXXX$, and speed $\frac{1}{5}$.
        \begin{center}
    \includegraphics[scale=0.6]{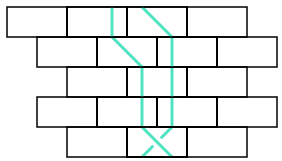}
        \end{center}
        \item[16] $([[r_1^{(1)}, l_2^{(2)}], [n_0, s_1^{(1)}][s_2^{(2)}, n_3], [s_1^{(1)}, s_2^{(2)}], [n_0, s_1^{(1)}][r_2^{(2)}, n_3], [r_1^{(1)}, n_2][s_3^{(2)}, n_4]])^\infty$, which has turning rule $01100010X$, crossing rule $XXXX1XXXX$, and speed $-\frac{1}{5}$.
        \begin{center}
    \includegraphics[scale=0.6]{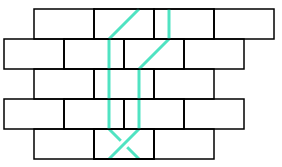}
        \end{center}
        \item[17] $([\overline{[r_1^{(1)}, l_2^{(2)}]}, [n_0, s_1^{(1)}][s_2^{(2)}, n_3], [s_1^{(1)}, s_2^{(2)}], [n_0, s_1^{(1)}][r_2^{(2)}, n_3], [r_1^{(1)}, n_2][s_3^{(2)}, n_4]])^\infty$, which has turning rule $01100010X$, crossing rule $XXXX0XXXX$, and speed $-\frac{1}{5}$.
        \begin{center}
    \includegraphics[scale=0.6]{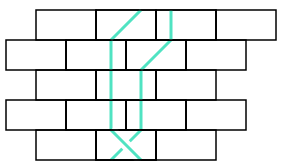}
        \end{center}
        \item[18] $([[s_1^{(1)}, s_2^{(2)}], [n_0, l_1^{(1)}][s_2^{(2)}, n_3], [n_0, s_1^{(1)}][n_2, l_3^{(2)}]])^\infty$, which has turning rule $X001001XX$, crossing rule $XXXXXXXX$, and speed $\frac{1}{3}$.
        \begin{center}
    \includegraphics[scale=0.6]{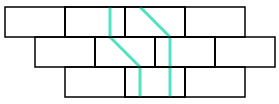}
        \end{center}
        \item[19] $([[n_0, l_1^{(1)}][n_2, s_3^{(2)}], [n_0, s_1^{(1)}][n_2, n_3][s_4^{(2)}, n_5], [s_1^{(1)}, n_2][n_3, l_4^{(2)}]])^\infty$, which has turning rule\\
        $X00100XXX$, trivial crossing rule, and speed $\frac{1}{3}$.
        \begin{center}
    \includegraphics[scale=0.6]{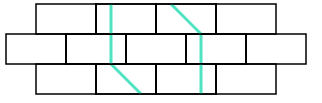}
        \end{center}
        \item[20] $([[s_1^{(1)}, n_2][n_3, s_4^{(2)}], [n_0, l_1^{(1)}][n_2, n_3][s_4^{(2)}, n_5], [n_0, s_1^{(1)}][n_2, n_3][n_4, l_5^{(2)}]])^\infty$, which has turning rule $X0X100XXX$, trivial crossing rule, and speed $\frac{1}{3}$.
        \begin{center}
    \includegraphics[scale=0.6]{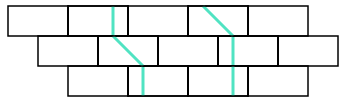}
        \end{center}
        \item[21] $([[s_1^{(1)}, n_2][r_3^{(2)}, n_4], [n_0, s_1^{(1)}][n_2, n_3][s_4^{(2)}, n_5], [r_1^{(1)}, n_2][n_3, s_4^{(2)}]])^\infty$, which has turning rule\\
        $X1X00X00X$, trivial crossing rule, and speed $-\frac{1}{3}$.
        \begin{center}
    \includegraphics[scale=0.6]{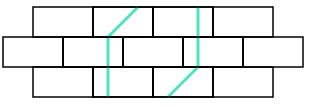}
        \end{center}
        \item[22] $([[s_1^{(1)}, s_2^{(2)}], [n_0, s_1^{(1)}][r_2^{(2)}, n_3], [r_1^{(1)}, n_2][s_3^{(2)}, n_4]])^\infty$, which has turning rule $X1100X00X$, trivial crossing rule, and speed $-\frac{1}{3}$.
        \begin{center}
    \includegraphics[scale=0.6]{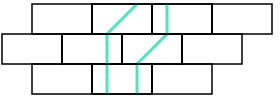}
        \end{center}
        \item[23] $([[s_1^{(1)}, n_2][n_3, s_4^{(2)}], [n_0, s_1^{(1)}][n_2, n_3][r_4^{(2)}, n_5], [r_1^{(1)}, n_2][n_3, n_4][s_5^{(2)}, n_6]])^\infty$, which has turning rule $X1X00XX0X$, trivial crossing rule, and speed $-\frac{1}{3}$.
        \begin{center}
    \includegraphics[scale=0.6]{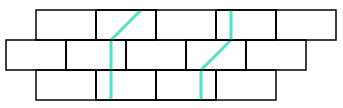}
        \end{center}
        \item[24] $([[n_0, s_1^{(1)}][n_2, s_3^{(2)}], [s_1^{(1)}, n_2][s_3^{(2)}, n_4], [n_0, l_1^{(1)}][n_2, s_3^{(2)}], [n_0, s_1^{(1)}][n_2, n_3][s_4^{(2)}, n_5], $\\
        $ [s_1^{(1)}, n_2][n_3, l_4^{(2)}], [n_0, l_1^{(1)}][n_2, l_3^{(2)}]])^\infty$, which has turning rule $001100XX0$, trivial crossing rule and speed $\frac{2}{6}$.
        \begin{center}
    \includegraphics[scale=0.6]{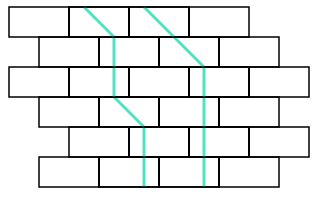}
        \end{center}
        \item[25] $([[s_1^{(1)}, n_2][s_3^{(2)}, n_4], [n_0, s_1^{(1)}][n_2, s_3^{(2)}], [s_1^{(1)}, n_2][r_3^{(2)}, n_4], [n_0, s_1^{(1)}][n_2, n_3][s_4^{(2)}, n_5],$\\
        $ [r_1^{(1)}, n_2][n_3, s_4^{(2)}], [r_1^{(1)}, n_2][r_3^{(2)}, n_4]])^\infty$, which has turning rule $01X00X100$, crossing rule $XXXXXXXX$, and speed $-\frac{2}{6}$.
        \begin{center}
    \includegraphics[scale=0.6]{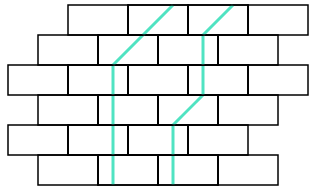}
        \end{center}
        \item[26] $([[r_1^{(1)}, l_2^{(2)}], [n_0, s_1^{(1)}][s_2^{(2)}, n_3], [s_1^{(1)}, s_2^{(2)}], [n_0, l_1^{(1)}][s_2^{(2)}, n_3], [n_0, s_1^{(1)}][n_2, s_3^{(2)}], [s_1^{(1)}, n_2][s_3^{(2)}, n_4],$\\
        $ [n_0, l_1^{(1)}][n_2, s_3^{(2)}], [n_0, s_1^{(1)}][n_2, n_3][s_4^{(2)}, n_5], [s_1^{(1)}, n_2][n_3, l_4^{(2)}], [n_0, l_1^{(1)}][n_2, l_3^{(2)}],$\\
        $ [n_0, s_1^{(1)}][n_2, l_3^{(2)}]])^\infty$, which has turning rule $001100001$, crossing rule $XXXX1XXXX$, and speed $\frac{3}{11}$.
        \begin{center}
    \includegraphics[scale=0.6]{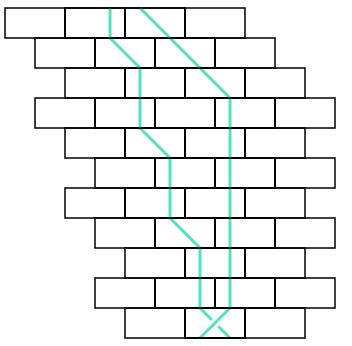}
        \end{center}
        \item[27] $([\overline{[r_1^{(1)}, l_2^{(2)}]}, [n_0, s_1^{(1)}][s_2^{(2)}, n_3], [s_1^{(1)}, s_2^{(2)}], [n_0, l_1^{(1)}][s_2^{(2)}, n_3], [n_0, s_1^{(1)}][n_2, s_3^{(2)}], [s_1^{(1)}, n_2][s_3^{(2)}, n_4],$\\
        $ [n_0, l_1^{(1)}][n_2, s_3^{(2)}], [n_0, s_1^{(1)}][n_2, n_3][s_4^{(2)}, n_5], [s_1^{(1)}, n_2][n_3, l_4^{(2)}], [n_0, l_1^{(1)}][n_2, l_3^{(2)}],$\\
        $ [n_0, s_1^{(1)}][n_2, l_3^{(2)}]])^\infty$, which has turning rule $001100001$, crossing rule $XXXX0XXXX$, and speed $\frac{3}{11}$.
        \begin{center}
    \includegraphics[scale=0.6]{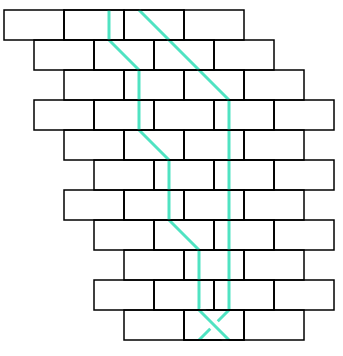}
        \end{center}
        \item[28] $([[r_1^{(1)}, l_2^{(2)}], [n_0, s_1^{(1)}][s_2^{(2)}, n_3], [s_1^{(1)}, s_2^{(2)}], [n_0, s_1^{(1)}][r_2^{(2)}, n_3], [s_1^{(1)}, n_2][s_3^{(2)}, n_4], [n_0, s_1^{(1)}][n_2, s_3^{(2)}], $ \\
        $ [s_1^{(1)}, n_2][r_3^{(2)}, n_4], [n_0, s_1^{(1)}][n_2, n_3][s_4^{(2)}, n_5], [r_1^{(1)}, n_2][n_3, s_4^{(2)}],$        $[r_1^{(1)}, n_2][r_3^{(2)}, n_4], $\\
        $ [r_1^{(1)}, n_2][s_3^{(2)}, n_4]])^\infty$ which has turning rule $010000101$, crossing rule $XXXX1XXXX$, and speed $-\frac{3}{11}$.
        \begin{center}
    \includegraphics[scale=0.6]{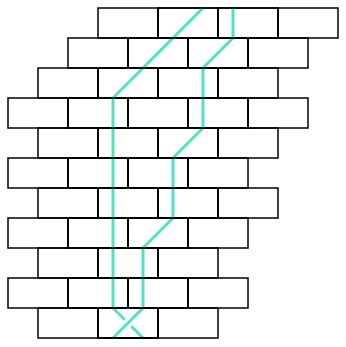}
        \end{center}
        \item[29] $([\overline{[r_1^{(1)}, l_2^{(2)}]}, [n_0, s_1^{(1)}][s_2^{(2)}, n_3], [s_1^{(1)}, s_2^{(2)}], [n_0, s_1^{(1)}][r_2^{(2)}, n_3], [s_1^{(1)}, n_2][s_3^{(2)}, n_4], [n_0, s_1^{(1)}][n_2, s_3^{(2)}], $ \\
        $ [s_1^{(1)}, n_2][r_3^{(2)}, n_4], [n_0, s_1^{(1)}][n_2, n_3][s_4^{(2)}, n_5], [r_1^{(1)}, n_2][n_3, s_4^{(2)}],$        $[r_1^{(1)}, n_2][r_3^{(2)}, n_4], $\\
        $ [r_1^{(1)}, n_2][s_3^{(2)}, n_4]])^\infty$ which has turning rule $010000101$, crossing rule $XXXX0XXXX$, and speed $-\frac{3}{11}$.
        \begin{center}
    \includegraphics[scale=0.6]{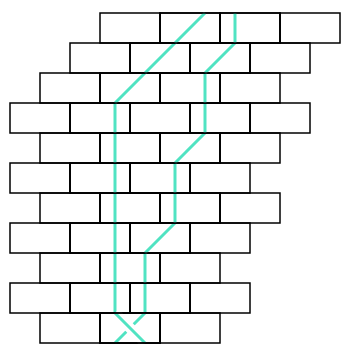}
        \end{center}
        \item[30] $([[n_0,l_1^{(1)}], [n_2,l_3^{(2)}]])^\infty$ which has turning rule $XXXX01XXX$, trivial crossing rule, and speed $1$.
\begin{center}
    \includegraphics[scale=0.9]{scaStuff/sca815_6.png}
        \end{center}
        \item[31] $([[r_1^{(1)}, n_2],[r_3^{(2)}, n_4]])^\infty$ which has turning rule $XXXX0XX1X$, trivial crossing rule, and speed $-1$.
\begin{center}
    \includegraphics[scale=0.9]{scaStuff/sca815_7.png}
        \end{center}
    \end{itemize}
\end{theorem}
\begin{corollary}
    Let $g$ be a 2-stranded SCA pattern. $g$ is a glider if it is a shift of one of items $14$-$31$ of \thref{2StrandClassification}.
\end{corollary}
Notice that the process we used to obtain this result cannot be easily generalized to $n$-stranded repeating SCA patterns where $n > 2$, as we used an essential fact about $2$-stranded repeating SCA patterns in this classification: The fact that the two strands must interact with each other at some point, which means that every $2$-stranded repeating SCA pattern must have a generation with width at most 2 cells. Because there is no guarantee when pairs of strands will interact in $n$-stranded repeating SCA patterns for $n > 2$ two different pairs of strands could interact in separate generations, which means that we no longer have an easy upper bound on the width of some generation of a $n$-stranded SCA pattern. In section \ref{sec:decidabilityAndPure}, we demonstrate a subset of gliders for which a similar method can still be used.
\section{Pure Gliders}\label{sec:Pure}
We must first define a non-crossing glider:
\begin{definition}
    A \textbf{non-crossing glider} $g = ([\delta_1, \delta_2, \cdots, \delta_m])^\infty$ is a glider where there does not exist $C\in\delta_k$ for $1\leq k\leq m$ and $C = \overline{[r_i^{(j)}, l_{i+1}^{(j+1)}]}$.
\end{definition}
We now define pure gliders:
\begin{definition}
    A \textbf{pure positive glider} $g$ is a glider with positive speed that has a turning rule with the bit $\frac{S}{SE}$.
\end{definition}
\begin{definition}
    A \textbf{pure negative glider} $g$ is a glider with negative speed that has a turning rule with the bit $\frac{S}{ES}$.
\end{definition}
A \textbf{pure glider} is a pure positive glider or a pure negative glider. Intuitively, a pure glider is one that doesn't go back-and-forth as time goes on. As we will soon see, the set of pure gliders is a subset of another type of glider. To describe this type of glider, we introduce a few more definitions:
\begin{definition}\thlabel{def}
    The $k-$\textbf{stranded left subpattern} of an $n$-stranded glider $g$, denoted $g_k^l$ for $k\leq n$, is the subpattern that contains exactly the $k$ strands ordered lowest by index in the first generation.
\end{definition}
\begin{definition}
    The $k-$\textbf{stranded right subpattern} of an $n$-stranded glider $g$, denoted $g_k^r$ for $k\leq n$, is the subpattern that contains exactly the $k$ strands ordered highest by index in the first generation.
\end{definition}

\begin{definition}
    A \textbf{positive nested} $n-$stranded glider $g$ is a glider with positive speed such that $g_k^l$ is a glider for all $1\leq k\leq n$.
\end{definition}
\begin{definition}
    A \textbf{negative nested} $n$-stranded glider $g$ is a glider with negative speed such that $g_k^r$ is a glider for all $1\leq k\leq n$.
\end{definition}
A \textbf{nested glider} is a glider that is a positive nested glider or a negative nested glider. Notice that it is not immediate that all gliders are nested gliders, as it is theoretically possible to have a positive (negative) glider such that $g_k^l$ ($g_k^r$) is not an SCA pattern for some $k$.

As we will see, the existence of a non-nested non-crossing SCA glider would imply the existence of a positive non-crossing SCA glider with all turning rules containing $\frac{T}{SE}$. We conjecture that the other direction also holds, and that no such gliders exist. We aim to prove that all pure gliders are nested in \thref{pureisNested}. Intuitively, the following result proves that all speed 1 gliders look like this:
\begin{center}
    \includegraphics[]{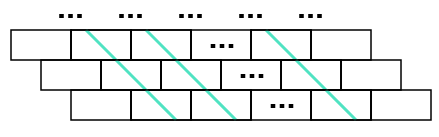}
\end{center}
and that all speed -1 gliders look like this:
\begin{center}
    \includegraphics[]{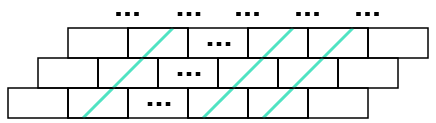}
\end{center}
\begin{lemma}\thlabel{speedcgliders}
    For any $n\geq 1$, there is exactly one speed $1$ glider and exactly one speed $-1$ glider. The $n$-stranded speed $1$ glider is given by $[[n_0, l_1^{(1)}][n_2, l_3^{(2)}]\cdots[n^{\left(\varnothing\right)}_{2(n-1)}, l_{2n - 1}^{(n)}]]^\infty$ and the $n$-stranded speed $-1$ glider is given by $[[r_1^{(1)}, n_2][r_3^{(2)}, n_4]\cdots[r_{2(n-1)}^{(n)}, n^{\left(\varnothing\right)}_{2n - 1}]]^\infty$. For all $n\in\mathbb{N}^+$, all speed $1$ gliders have the same set of turning and crossing rules (turning rule: $XXXX01XX1$ and trivial crossing rule) and are nested gliders. For all $n\in\mathbb{N}^+$, all speed $-1$ gliders have the same set of turning and crossing rules (turning rule: $XXXX0XX11$ and trivial crossing rule) and are nested gliders. 
\end{lemma}
\begin{proof}
    We prove the result for positive speed gliders, as it follows analogously for negative speed gliders. We first demonstrate the existence of a speed $1$ glider on $n$ strands. Let $a = [n_0, l_1^{(1)}][n_2, l_3^{(2)}]\cdots[n^{\left(\varnothing\right)}_{2(n-1)}, l_{2n - 1}^{(n)}]$ and set $g = a^\infty$. By the definition of speed, $g$ has speed $1$. $g$ only uses $\frac{T}{TT}, \frac{T}{ET}$ by its definition, so it has turning rule $XXXX01XX1$. It contains no crossings, so it has trivial crossing rule. Thus, $g$ is an $n$-stranded repeating SCA pattern. Notice that every strand of $g$ is adjacent to another strand of $g$ in every generation of $g$, so it is an $n$-stranded SCA glider. Now consider $g_l^{(k)}$ for some $1\leq k\leq n$. Notice $g_l^{(k)} = [n_0, l_1^{(1)}][n_2, l_3^{(2)}]\cdots[n^{\left(\varnothing\right)}_{2(k-1)}, l_{2k - 1}^{(k)}]$, so $g_l^{(k)}$ has speed $1$ and every strand of $g_l^{(k)}$ is adjacent to another strand of $g_l^{(k)}$ in every generation of $g$. If $k\neq 1$, $g_l^{(k)}$ only uses $\frac{T}{TT}, \frac{T}{ET}$ so it has turning rule $XXXX01XX1$, and it contains no crossings, so it has trivial crossing rule, so it is a $k$-stranded SCA glider. If $k = 1$, $g_l^{(k)}$ only uses $\frac{T}{ET}$ and contains no crossings, so it has turning rule $XXXX01XXX$, a superset of $XXXX01XX1$ and trivial crossing rule. Thus, it is a $1$-stranded SCA glider, so by the definition of nested glider, $g$ is a nested glider. We now demonstrate the uniqueness of $g$.

    For contradiction, suppose there exists a $n\in\mathbb{N}^+$ such that there exist two different speed $1$ gliders on $n$ strands. 
    Let $g$ be the glider constructed above, and let $h$ be a speed $1$ glider that is distinct from $g$.

    We first consider the case where $h$ contains a crossing $[r_z^{(k)}, l_{z+1}^{(k+1)}]$ in $\delta_i$ where strand $k + 1$ starts at index $t$. Let $j$ be the first $j > i$ such that $\delta_j = \delta_i$. We claim that none of the $n - k$ strands which end at an index greater than or equal to $t$ in $\delta_i$ will reach the starting indices of the first $k + 1$ strands in generation $j$, which will contradict the continuity of $h$. We first consider strand $k$, which ends at index $t  + 1$ in $\delta_i$ and ends at index at least $t + 1$ in $\delta_{i+1}$, as it can either turn right or go straight in $\delta_{i+1}$. Thus, to reach any of the first $k + 1$ strands in $\delta_j$, it must move at least $|t - ((j - i) + 1) - (t + 1)| = j - i$ cells to the left in $j - i - 2$ generations, which is impossible. If there is a strand starting in index $t + 2$ in $\delta_i$, it must be straight or turning right, so it must move at least $|t - (j - i) + 1 - (t + 2)|  = (j - i) + 1$ cells to the left in $(j - i) - 1$ generations, which is also impossible. If there is a strand starting in index $t + m$ for $m\geq 3$, it must move at least $|t - (j - i) + 1 - (t + m)| \geq (j - i) + 2$ cells to the left in $j - i$ generations, also impossible. Thus, $h$ contains no crossings. 

    Now suppose there exists a cell $C_{i, u}'$ $h$ such that $C_{i, u}'\neq[n^{\left(\varnothing\right)}_{2(u-1)}, l_{2u-1}^{(k)}]$ for any $k$ but still contains a strand, call it $m$. Let $j$ be the next generation identical to $i$. Then, because $h$ is speed $1$ and contains no crossings, strand $m$ must move $j - i$ indices to the left in $j - i$ generations for the pattern to be continuous. Because it does not move left in generation $i$, it must move $j-i$ indices to the left in at least $j-i-1$ generations, which is impossible. Thus, $h$ contains no crossings and every cell of $h$ that contains a strand has the strand turning left. Now, because $h\neq g$, there must exist some $C_{i,u}' = [n^{\left(\varnothing\right)}_{2(u-1)}, n^{\left(\varnothing\right)}_{2u-1}]$ and strands $k, k+1$ such that $k$ is in cell $C_{i, t}'$ and $k + 1$ is in cell $C_{i, j}'$ for some $t < u < j$. We claim strands $k$ and $k + 1$ are never in adjacent cells or the same cell. For contradiction, suppose there exists some generation $r$ such that they are. Then, there exists a least $r' > i$ such that $\delta_{r'}' = \delta_r'$ by the definition of glider. By the fact that every cell that contains a strand is of the form $[n_z^{\left(\varnothing\right)}, l_{z+1}^y]$ for some $z, y$, we must have that strand $k$ moves $r' - i$ cells to the left between generations $i$ and $r'$, and that $C_{r', u}' = C_{i, t}'$. This means strand $k + 1$ must move $r' - i + 1$ cells to the left in $r' - i$ generations, a contradiction. Thus, we must have $\delta_i = [n_0, l_1^{(1)}]\cdots[n^{\left(\varnothing\right)}_{2n-2}, l_{2n-1}^{(n)}]$ for all $i$, so $h = [[n_0, l_1^{(1)}]\cdots[n^{\left(\varnothing\right)}_{2n-2}, l_{2n-1}^{(n)}]]^\infty = g$. Thus, $g$ is unique, and the claim follows.
\end{proof}
Recall that ``a glider $g$ has speed $\frac{a}{b}$" means that $g$ moves $a$ positions to the left in $|\per(g)| = b$ generations.

We now introduce the \textbf{valuation} for speed, $\nu$. If $\frac{a}{b}$ is a speed, $\nu(\frac{a}{b})$ returns the rational number corresponding to $\frac{a}{b}$. We call the valuation of a speed the valuated speed.
\begin{lemma}\thlabel{gliderSpeed}
    Let $g$ be a non-crossing SCA repeating grid pattern, and suppose the first strand of $g$ has speed $\frac{a}{b}$ for some $a > 0, b\geq 1$. Then, the speed of $g$ is $\frac{ad}{bd}$ for some $d\geq 1$. If $g$ is a non-crossing glider and the last strand of $g$ has speed $-\frac{a}{b}$ for some $a> 0, b\geq 1$, the speed of $g$ is $-\frac{ad}{bd}$ for some $d\geq 1$. If the speed of the first or last strand of $g$ is 0, the speed of $g$ is 0. Also, all strands have the same valuated speed.
\end{lemma}
\begin{proof}
    Let $g = h^\infty$ for some $h = [\delta_1',\cdots,\delta_m']$. Let the subgrid pattern $u = v^\infty$ where $v = [\delta_1^*,\cdots,\delta_m^*]$ be the subgrid pattern generated by removing all but the first strand of $g$ from each $\delta_i'$ in $h$. If $\delta_i' = [x, \cdots]\cdots$ then $\delta_i^* = [x, n_2]\cdots$ and if $\delta_i' = [\cdots, x]\cdots$ then $\delta_i^* = [n_0, x]\cdots$ for all $1\leq i\leq m$. Note $[\delta_1^*,\cdots, \delta_m^*] = [\delta_1^*,\cdots, \delta_u^*]^d$ for some $1\leq u\leq m$ and $d\geq 1$. Pick the least such $u$, which results in the greatest $d\geq 1$.

   We first prove the lemma for the case where $g$ is a positive glider. The speed of $[\delta_1^*,\cdots, \delta_u^*] $ must then be $\frac{a}{b}$ for some $a\geq 1, b\geq 1$.  Then, by the definition of speed of $g$, the speed of $g$ is $\frac{ad}{bd}$, so the claim for positive non-crossing gliders is satisfied.
     
    Let $g$ be an $n$-stranded negative glider. The speed of $[\delta_1^*,\cdots, \delta_u^*]$ must then be $-\frac{a}{b}$ for some $a\geq 1, b\geq 1$. Because, for any $i\geq 1, \delta_i = \delta_{im + 1}$ in $g$, the leftmost strand moves right $da$ positions in $db = m$ generations, so the speed of $g$ is $-\frac{da}{db}$, and the claim for negative non-crossing gliders is satisfied.

    Let $g$ be an $n$-stranded SCA pattern such that the first strand of $g$ is speed 0. The speed of $[\delta_1^*,\cdots, \delta_u^*] $ must then be $\frac{a}{b}$ for some $a\geq 1, b\geq 1$, but because the first strand has speed 0, $a = 0$. Then, by the definition of speed of $g$, the speed of $g$ is $\frac{ad}{bd} = 0$, so the claim for a glider such that the first strand of $g$ is satisfied. The case where the $n$th strand of $g$ is speed 0 follows similarly to the case where the speed of the $n$th strand is negative.

    From here, one can see that all strands of a noncrossing glider have the same valuated speed.
\end{proof}
\begin{lemma}\thlabel{mustHaveSET}
    Let $g$ be a noncrossing glider that does not have speed $1$ or $-1$. Then there exist turning rules of $g$ with the bits $\frac{S}{TE}, \frac{S}{ET}$.
\end{lemma}
\begin{proof}
    We will prove the lemma only for positive noncrossing gliders, as the negative case is analogous. For contradiction, suppose otherwise. Then, all turning rules of some $n-$stranded glider $g = h^\infty$ contain at least one of $\frac{T}{ET}$ or $\frac{T}{TE}$.      
    
    First consider the case where all turning rules of $g$ contain $\frac{T}{ET}$. Because $g$ is a positive glider, there exists some $\delta_i$ with $\delta_i = [n_0, l_1^{(1)}]\cdots$. 

    \hfill\break
    \textbf{Claim 1}: For all $j\geq i, \delta_j = [n_0, l_1^{(1)}]\cdots$.
\hfill\break
    \textbf{Proof}: We proceed by induction.

\hfill\break
    \textbf{Base case}: $\delta_i = [n_0, l_1^{(1)}]\cdots$, so because we have $\frac{T}{ET}$, $\delta_{i+1} = [n_0, l_1^{(1)}]\cdots$.
    
\hfill\break
    \textbf{Inductive step}: Suppose, for some $j\in\mathbb{N}$, that $\delta_j = [n_0, l_1^{(1)}]\cdots$. Then, because we have the bit $\frac{T}{ET}$, $\delta_{j + 1} = [n_0, l_1^{(1)}]\cdots$ and the inductive hypothesis is confirmed.
\begin{flushright}
$\dashv_{\text{Claim}}$
\end{flushright}
    By claim 1, for all $k\geq i, \delta_k = [n_0, l_1^{(1)}]\cdots$. Because $h = [\delta_1', \cdots, \delta_m']$ has finite length $l$, there exists some $r > 0$ such that $lr > i$. Thus, for all $\delta_k$ in $h^{lr}$ with $k> l(r-1)$, $\delta_k = [n_0, l_1^{(1)}]\cdots$, so the first strand of $h$ turns left in every generation. Thus, by \thref{gliderSpeed} $g$ has speed $1$, a contradiction. We now suppose every turning rule of $g$ contains $\frac{T}{TE}$. We break the remainder of the proof into cases:

    \hfill\break
\textbf{Case 1:} $g$ is pure

\hfill\break
    \textbf{Claim 2:} There exists some $i\geq 1, 1\leq k\leq n$ such that $\length(h)\leq 2\length(h)$ and $\delta_{i+1} = \cdots[r_j^{(k)}, n^{\left(\varnothing\right)}_{j+1}][n^{\left(\varnothing\right)}_{j+2}, n^{\left(\varnothing\right)}_{j+3}]\cdots$.
    
    \textbf{Proof:} To prove the claim, first note that $h = [\delta_1',\cdots, \delta_m']$ cannot have $\delta_y' = \cdots[r_j^{(k)}, n^{\left(\varnothing\right)}_{j+1}][n^{\left(\varnothing\right)}_{j+2}, n^{\left(\varnothing\right)}_{j+3}]\cdots$ for all $1\leq y\leq m$ or else $g$ would have speed $-1$ by \thref{gliderSpeed}. Thus, there must exist some $1\leq y\leq m$ such that the configuration $[r_j^{(k)}, n^{\left(\varnothing\right)}_{j+1}][n^{\left(\varnothing\right)}_{j+2}, n^{\left(\varnothing\right)}_{j+3}]$ does not appear in $\delta_y'$.
    \begin{flushright}
$\dashv_{\text{Claim}}$
\end{flushright}
    If there exists a $1\leq x\leq m$ such that $\delta_x'$ does not contain the configuration $[r_j^{(k)}, n^{\left(\varnothing\right)}_{j+1}][n^{\left(\varnothing\right)}_{j+2}, n^{\left(\varnothing\right)}_{j+3}]$ but $\delta_{x+1}'$ does, pick $i = \length(h) + x$. Otherwise, $\delta_1' = \cdots[r_j^{(k)}, n^{\left(\varnothing\right)}_{j+1}][n^{\left(\varnothing\right)}_{j+2}, n^{\left(\varnothing\right)}_{j+3}]\cdots$ and $\delta_m'$ does not contain the configuration (by claim 2) $[r_j^{(k)}, n^{\left(\varnothing\right)}_{j+1}]$, in which case set $i = \length(h) = m$. Because $g$ is a non-crossing glider, strand $k$ in generation $1$ is the same strand in all generations of $g$. Notice that $\delta_{i+1} = \cdots[r_j^{(k)}, n^{\left(\varnothing\right)}_{j+1}][n^{\left(\varnothing\right)}_{j+2}, n^{\left(\varnothing\right)}_{j+3}]\cdots$, $\delta_{i} = \cdots[\cdots][n^{\left(\varnothing\right)}_{v+1}, n^{\left(\varnothing\right)}_{v+2}]\cdots$ such that $\gen([\cdots][n^{\left(\varnothing\right)}_{v+1}, n^{\left(\varnothing\right)}_{v+2}]) = [r_j^{(k)}, n^{\left(\varnothing\right)}_{j+1}]$ and $[\cdots]$ does not contain a right turning strand (as strand $k$ denotes the same strand in generation $i+1$ as in generation $i$). This implies, by the continuity of $g$, that we must have $\delta_i = \cdots[\cdots, s_{v}^{(k)}][n^{\left(\varnothing\right)}_{v+1}, n^{\left(\varnothing\right)}_{v+2}]$, so the rule $\frac{T}{SE}$ is used, a contradiction to the purity of $g$.

    \hfill\break
    \textbf{Case 2:} $g$ is not pure
    \hfill\break
    Consider strand $n$. Because $g$ has positive speed and is noncrossing, there must exist a generation $\delta_i$ in which $[n_j, l_{j+1}^{(n)}]$ occurs in $\delta_i$ by \thref{gliderSpeed} for some $j$. Thus, because strand $n$ cannot turn left in every generation (or else we would reach a contradiction to \thref{gliderSpeed}), there exists some $k > j$ such that $[\dots, s_j^{(n)}][n_{j+1}^{(\varnothing)}, n_{j+2}^{(\varnothing)}]$ occurs in $\delta_k$ for some $j$. Thus, there exists $u\in\mathbb{N}$ such that $[r_u^{(n)}, n_{u+1}^{(\varnothing)}][n_{u+2}^{(\varnothing)},n_{u+3}^{(\varnothing)}]$ occurs in $\delta_{k+1}$. Then, for all $m > k$, there exists a $u_m\in\mathbb{N}$ such that $[r_{u_m}^{(n)}, n_{u_m+1}^{(\varnothing)}][n_{u_m+2}^{(\varnothing)},n_{u_m+3}^{(\varnothing)}]$ occurs in $\delta_{m}$, so there exists no $\delta_m, m  > k$, with $\delta_m = \delta_k$. Thus, $g$ is not a glider, a contradiction.
\end{proof}
\begin{lemma}\thlabel{noconsective3Turns}
    Let $g$ be an $n$-stranded noncrossing glider that does not have speed $1$ or $-1$ and is such that there exist turning rules with the bit $\frac{T}{TT}$. Then, if $g$ is positive, the pattern:
    \begin{center}
        \includegraphics[scale=0.6]{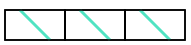}
    \end{center}
    does not occur in $g$, and if $g$ is negative, the pattern:
    \begin{center}
        \includegraphics[scale=0.9]{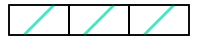}
    \end{center}
    does not occur in $g$.
\end{lemma}
\begin{proof} We only prove the case where $g$ is a positive glider, as the case where $g$ is a negative glider is analogous. We proceed by contradiction. Notice that there are five cases to consider, depending on the cell directly to the left of the leftmost turn which is part of at least three consecutive left turns in a generation:
\begin{enumerate}
    \item The case where $[n_k^{\left(\varnothing\right)}, n^{\left(\varnothing\right)}_{k+1}][n^{\left(\varnothing\right)}_{k+2}, l_{k+3}^{(u)}][n^{\left(\varnothing\right)}_{k+4}, l_{k+5}^{(u+1)}][n^{\left(\varnothing\right)}_{k+6}, l_{k+7}^{(u+2)}]$ occurs as a sublist of  $g$.
    \item The case where $[s_k^{(u)}, n^{\left(\varnothing\right)}_{k+1}][n^{\left(\varnothing\right)}_{k+2}, l_{k+3}^{(u + 1)}][n^{\left(\varnothing\right)}_{k+4}, l_{k+5}^{(u+2)}][n^{\left(\varnothing\right)}_{k+6}, l_{k+7}^{(u+3)}]$ occurs as a sublist of $g$.
    \item The case where $[n_k^{\left(\varnothing\right)}, s_{k+1}^{(u)}][n^{\left(\varnothing\right)}_{k+2}, l_{k+3}^{(u + 1)}][n^{\left(\varnothing\right)}_{k+4}, l_{k+5}^{(u+2)}][n^{\left(\varnothing\right)}_{k+6}, l_{k+7}^{(u+3)}]$ occurs as a sublist of $g$.
    \item The case where $[s_k^{(u)}, s_{k+1}^{(u+1)}][n^{\left(\varnothing\right)}_{k+2}, l_{k+3}^{(u + 2)}][n^{\left(\varnothing\right)}_{k+4}, l_{k+5}^{(u+3)}][n^{\left(\varnothing\right)}_{k+6}, l_{k+7}^{(u+4)}]$ occurs as a sublist of $g$.
    \item The case where $[r_k^{(u)}, n_{k+1}^{(\varnothing)}][n^{\left(\varnothing\right)}_{k+2}, l_{k+3}^{(u + 1)}][n^{\left(\varnothing\right)}_{k+4}, l_{k+5}^{(u+2)}][n^{\left(\varnothing\right)}_{k+6}, l_{k+7}^{(u+3)}]$ occurs as a sublist of $g$.
\end{enumerate}
Notice that case 5 leads to a contradiction because $\frac{T}{TT}$ forces $g$ to contain a crossing.

\hfill\break
    \textbf{Case 1:}  The pattern $[n_k^{\left(\varnothing\right)}, n^{\left(\varnothing\right)}_{k+1}][n^{\left(\varnothing\right)}_{k+2}, l_{k+3}^{(u)}][n^{\left(\varnothing\right)}_{k+4}, l_{k+5}^{(u+1)}][n^{\left(\varnothing\right)}_{k+6}, l_{k+7}^{(u+2)}]$ occurs as a sublist of $g$.
    
    Choose $\delta_i = \cdots[n_k^{\left(\varnothing\right)}, n^{\left(\varnothing\right)}_{k+1}][n^{\left(\varnothing\right)}_{k+2}, l_{k+3}^{(u)}][n^{\left(\varnothing\right)}_{k+4}, l_{k+5}^{(u+1)}][n^{\left(\varnothing\right)}_{k+6}, l_{k+7}^{(u+2)}]\cdots$ such that $i\geq 4$. This is possible by the fact that $g$ is a glider. By $\frac{S}{ET}$ and $\frac{T}{TT}$, we get:
      \begin{center}
        \includegraphics[scale=0.9]{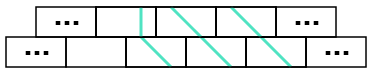}
    \end{center}
    Thus, to avoid a crossing, we must have the rule $\frac{S}{ST}$. Note that the outlined cell must be empty:
      \begin{center}
        \includegraphics[scale=0.6]{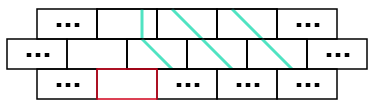}
    \end{center}
    So we must have:
      \begin{center}
        \includegraphics[scale=0.6]{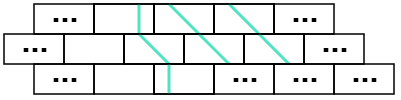}
    \end{center}
    Which, because we have $\frac{S}{ST}$, means we must have:
    \begin{center}
        \includegraphics[scale=0.6]{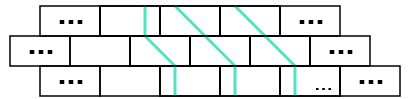}
    \end{center}
    So we also have $\frac{T}{SS}$. By $\frac{T}{SS}$ and $\frac{T}{TT}$, we must have one of the following options:
    \begin{center}
        \includegraphics[scale=0.6]{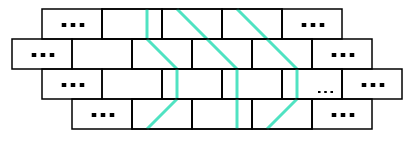}
    \end{center}
    or 
     \begin{center}
        \includegraphics[scale=0.6]{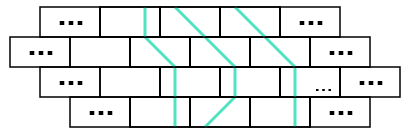}
    \end{center}
    In the first case, we get:
   \begin{center}
        \includegraphics[scale=0.6]{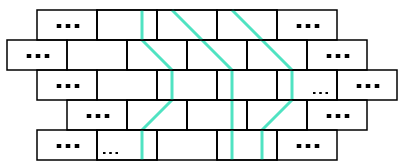}
    \end{center}
    By the fact that we have a turning rule of $g$ with $\frac{S}{TE}$ (\thref{mustHaveSET}). This is a contradiction to the assumption that the glider has positive speed, as it forces all turning rules of $g$ to have $\frac{S}{ES}$. In the second case, we get:
    \begin{center}
        \includegraphics[scale=0.6]{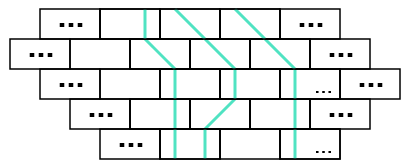}
    \end{center}
    because $g$ is a non-crossing glider. This is also a contradiction, as it forces all turning rules of $g$ to have $\frac{S}{ES}$, which contradicts the fact that all turning rules of $g$ have $\frac{T}{ES}$. Having reached contradictions in all parts of case 1, we move to case 2.

\hfill\break
    \textbf{Case 2:} The sublist $[s_k^{(u)}, n^{\left(\varnothing\right)}_{k+1}][n^{\left(\varnothing\right)}_{k+2}, l_{k+3}^{(u + 1)}][n^{\left(\varnothing\right)}_{k+4}, l_{k+5}^{(u+2)}][n^{\left(\varnothing\right)}_{k+6}, l_{k+7}^{(u+3)}]$ occurs in $g$.
    
    Choose $\delta_i = \cdots[s_k^{(u)}, n^{\left(\varnothing\right)}_{k+1}][n^{\left(\varnothing\right)}_{k+2}, l_{k+3}^{(u + 1)}][n^{\left(\varnothing\right)}_{k+4}, l_{k+5}^{(u+2)}][n^{\left(\varnothing\right)}_{k+6}, l_{k+7}^{(u+3)}]\cdots$ such that $i\geq 4$. This is possible because $g$ is a glider. We have two options because we have $\frac{T}{TT}$:
     \begin{center}
        \includegraphics[scale=0.6]{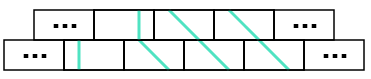}
    \end{center}
    or 
    \begin{center}
        \includegraphics[scale=0.6]{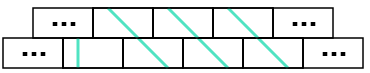}
    \end{center}
     In the first case, because we have $\frac{S}{ET}$ and $\frac{S}{ST}$, we must get:
   \begin{center}
        \includegraphics[scale=0.6]{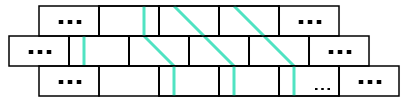}
    \end{center}
     Because we have $\frac{T}{SS}$ and $\frac{T}{TT}$, this means we must have one of the following:
     \begin{center}
        \includegraphics[scale=0.6]{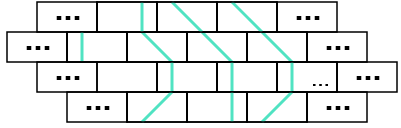}
    \end{center}
    or 
    \begin{center}
        \includegraphics[scale=0.6]{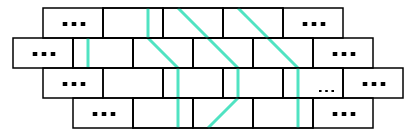}
    \end{center}
     In both cases, we must have $\frac{S}{ST}$. In the first case, we have $\frac{S}{TE}$ and $\frac{S}{ST}$, so we get:
    \begin{center}
        \includegraphics[scale=0.6]{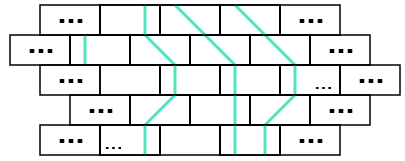}
    \end{center}
    by \thref{gliderSpeed}. This contradicts $\frac{T}{ES}$. In the second case, we have no crossings in $g$, so we get:
   \begin{center}
        \includegraphics[scale=0.6]{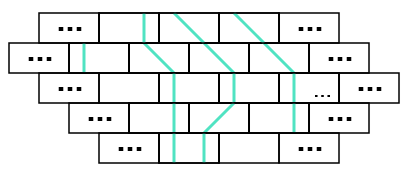}
    \end{center}
    If $g$ is pure, this is an immediate contradiction, so consider the case where all turning rules of $g$ have $\frac{T}{SE}$. Because $g$ has $\frac{T}{SE}$, we get:
    \begin{center}
        \includegraphics[scale=0.9]{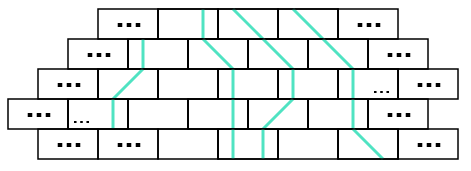}
    \end{center}
     This contradicts the fact that we have  $\frac{T}{ES}$. Thus, we must have:
    \begin{center}
        \includegraphics[scale=0.6]{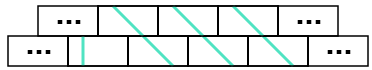}
    \end{center}
    Because we have $\frac{T}{ST}$, we must have:
 \begin{center}
        \includegraphics[scale=0.6]{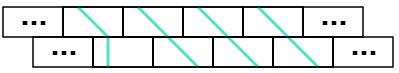}
    \end{center}
    where the leftmost strand turning strand pictured in generation $\delta_i$ is strand $k$. 
    
\hfill\break
    \textbf{Claim 1:} The pattern $p$, pictured below:
     \begin{center}
        \includegraphics[scale=0.6]{scaStuff/scafig4.png}
    \end{center}
    occurs on strands $u + 1- t, u + 2 - t, u + 3 - t$ in generation $\delta_{i + t}$ for all $0\leq t\leq u$.
\hfill\break
    \textbf{Proof:} We proceed by finite induction on $t$.

\hfill\break
    \textbf{Base case:} If $t = 0$, the base case holds by choice of $i$.

\hfill\break
    \textbf{Inductive Step:} Suppose $p$ occurs on strands $u + 1 - t, u + 2 - t, u + 3 - t$ in generation $\delta_{i + t}$ for some $0\leq t\leq u-1$. There are four options for the outlined cell below:
    \begin{center}
        \includegraphics[scale=0.6]{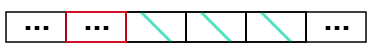}
    \end{center}

    Option 1: The cell is $[n_j^{\left(\varnothing\right)}, n^{\left(\varnothing\right)}_{j+1}]$ for some $j$. Then, we get:
    
    \begin{center}
        \includegraphics[scale=0.6]{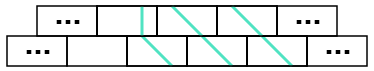}
    \end{center}
    
    because we have the rule $\frac{S}{ET}$. By the rule $\frac{T}{ST}$, the glider contains a crossing, contradicting our assumption that $g$ is pure.

    Option 2: The cell is $[x, s_{j+1}^{(u-t)}]$ for some $x\in\{n_j^{\left(\varnothing\right)}, s_j^{(u-t-1)}\}$. By the rule $\frac{T}{ST}$, there exists a crossing in the next generation, a contradiction to our assumption that $g$ is non-crossing. 

    Option 3: The cell is $[r_j^{(u-t)}, n^{\left(\varnothing\right)}_{j+1}]$. By the fact that we have $\frac{T}{TT}$, we get a crossing in the next generation, a contradiction to our assumption that $g$ contains no crossings.

    Option 4: The cell is $[n_j^{\left(\varnothing\right)}, l_{j+1}^{(u-t)}]$. By the fact that we have $\frac{T}{TT}$, we get:

   \begin{center}
        \includegraphics[scale=0.6]{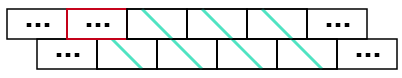}
    \end{center}

   By option 2, the outlined cell cannot have $[x, s_{j+1}^{(u-t-1)}]$ for some $x\in\{n_j^{\left(\varnothing\right)}, s_j^{(u-t-2)}\}$, so it must have $[n_j^{\left(\varnothing\right)}, l_{j+1}^{(u-t-1)}]$ for some $j$.

    Pattern $p$ occurs on strands $u+1 - t - 1 = u + 1 - (t + 1), u - t + 2 - 1 = u + 2 - (t + 1), u + 3 - t - 1 = u + 3 - (t + 1)$ in generation $\delta_{i + t + 1}$, so the inductive hypothesis is confirmed.

    Option 5: The cell is $[s_j^{(u-t)}, n^{\left(\varnothing\right)}_{j+1}]$. We must have, by $\frac{T}{ST}, \frac{T}{TT}$:

  \begin{center}
        \includegraphics[scale=0.6]{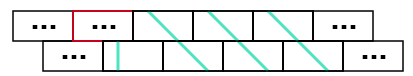}
    \end{center}

    If the outlined cell is $[x, s_{j + 1}^{(u-t-1)}], x\in\{n_j^{\left(\varnothing\right)}, s_j^{(u-t-2)}\}$, we reach a contradiction as in option 2. Thus, we must have:

 \begin{center}
        \includegraphics[scale=0.6]{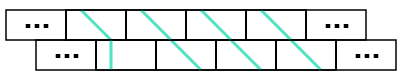}
    \end{center}

    Which confirms the inductive hypothesis.

    We have shown that the inductive hypothesis is confirmed in all possible cases, so the claim is proved.
\begin{flushright}
$\dashv_{\text{Claim}}$
\end{flushright}

By claim 1, let $t = u$. Then, the pattern:

 \begin{center}
        \includegraphics[scale=0.6]{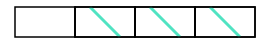}
    \end{center}

occurs in $g$. By case 1, this is a contradiction.

\hfill\break
\textbf{Cases 3 and 4:} The subpattern $[n_j^{\left(\varnothing\right)}, s_{j+1}^{(u)}][n^{\left(\varnothing\right)}_{j+2}, l_{j+3}^{(u+1)}][n^{\left(\varnothing\right)}_{j+4}, l_{j+5}^{(u+2)}][n^{\left(\varnothing\right)}_{j+6}, l_{j + 7}^{(u+3)}]$ occurs in $g$. Notice that considering the pattern as a subpattern instead of a sublist allows us to consider cases 3 and 4 at once.

Let $r$ be the maximum number of consecutive left turns occurring anywhere in the period of $g$. Then, $r\geq 3$, and in any generation $\delta_i$ that contains $r$ consecutive left turns, they must be preceded by $[n_k^{\left(\varnothing\right)}, s_{k+1}^{(u)}]$, or else we would reach a contradiction to case 1 or 2. Choose $\delta_i$ such that $i\geq n + 2$, which is possible by the fact that $g$ is a glider. Let $m = r + 2$.

\hfill\break
\textbf{Claim 2:} Pattern $p$:
 \begin{center}
        \includegraphics[scale=0.9]{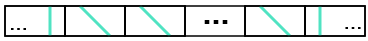}
\end{center}
appears in generation $\delta_{i-k}$ on strands $u - k, \cdots, u + m - k$ for $0\leq k\leq u-1$
\hfill\break
\textbf{Proof:} We proceed by finite induction on $k$.

\hfill\break
\textbf{Base Case:} This follows from our assumption on $i$.

\hfill\break
\textbf{Inductive Step:} 
Suppose $p$ appears in generation $\delta_{i-k}$ on strands $u - k, \cdots, u + m - k$ for some $1\leq k\leq u - 2$.
 To avoid a contradiction to maximality and because we have $\frac{S}{ST}$, we must have: 
 \begin{center}
        \includegraphics[scale=0.6]{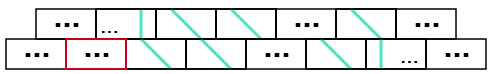}
    \end{center}
so we have the bit $\frac{T}{TS}$. Notice that the outlined cell must contain $[n_j^{\left(\varnothing\right)}, s_{j+1}^{(u-1)}]$ or else we would reach a contradiction to maximality or could apply cases 1 and 2.
Then we have:
 \begin{center}
        \includegraphics[scale=0.6]{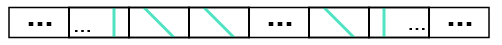}
    \end{center}

in generation $
\delta_{i-k}$.

 $\delta_{i -k + 1}$ has, by the fact that $g$ contains no crossings and the fact that we have $\frac{T}{TT}$ (claim 1):
 \begin{center}
        \includegraphics[scale=0.9]{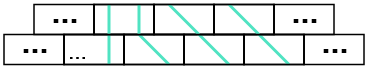}
    \end{center}
Thus, the rule $\frac{S}{ST}$ is used. 

\hfill\break
\textbf{Subclaim 1:} Strand $u-k$ turns in generation $i -k - 1$.
\hfill\break
\textbf{Proof:} For contradiction, suppose otherwise. By $\frac{S}{ST}$, we must have:
 \begin{center}
        \includegraphics[scale=0.9]{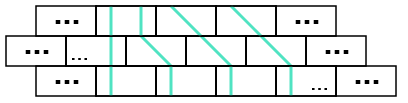}
    \end{center}
Thus the rule $\frac{T}{SS}$ is used. We now have two cases. If strand $u$ turns in $\delta_{i -k- 2}$, we get:
 \begin{center}
        \includegraphics[scale=0.9]{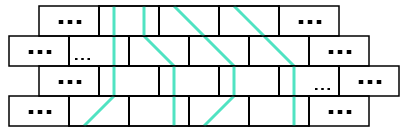}
    \end{center}
as we have $\frac{T}{SS}, \frac{T}{TT}$. To avoid a crossing, in generation $\delta_{i-3}$, we get:

 \begin{center}
        \includegraphics[scale=0.9]{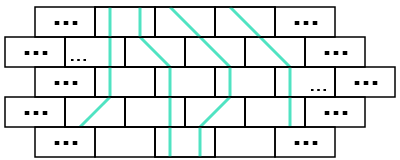}
    \end{center}
a contradiction to $\frac{T}{ES}$. If instead we have that strand $u$ goes straight in $\delta_{i -k- 2}$, we get:
 \begin{center}
        \includegraphics[scale=0.9]{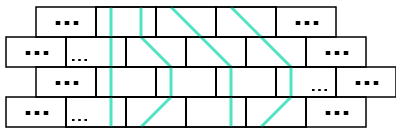}
    \end{center}
because we also have $\frac{T}{SS}, \frac{T}{TT}$. To avoid a crossing in generation $\delta_{i-k-3}$, we have:
 \begin{center}
        \includegraphics[scale=0.9]{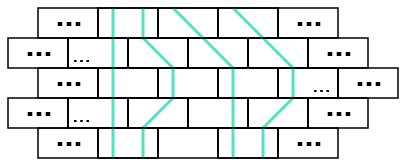}
    \end{center}
which contradicts the fact that we have $\frac{T}{ES}$. 
\begin{flushright}
$\blacktriangle_{\text{Subclaim}}$
\end{flushright}

\hfill\break
Thus, by the bit $\frac{S}{ST}$ and subclaim 1, we must have:
 \begin{center}
        \includegraphics[scale=0.6]{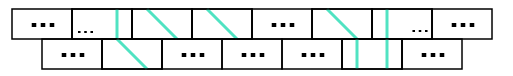}
    \end{center}
in $\delta_{i-k-1}$. If we have:

 \begin{center}
        \includegraphics[scale=0.6]{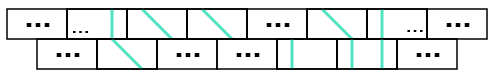}
    \end{center}
We must have $\frac{T}{SS}$ along with $\frac{T}{TS}$, which means we have:

 \begin{center}
        \includegraphics[scale=0.6]{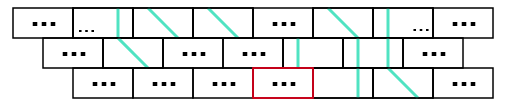}
    \end{center}
Notice that if the strand in the highlighted cell turns, we reach a contradiction to $\frac{T}{TS}$ but if it goes straight, we reach a contradiction to $\frac{T}{SS}$. Thus, we reach a contradiction in all cases where strand $u + m - k - 2$ is straight in generation $\delta_{i-k-1}$. Because we have the rule $\frac{S}{ST}$ and because we already established that strand $u - k - 1$ turns in $\delta_{i-k-1}$, we cannot have strand $u - k + t$ be straight in $\delta_{i-k-1}$ for $1\leq t\leq m - 2$.







Thus, we must have:

 \begin{center}
        \includegraphics[scale=0.6]{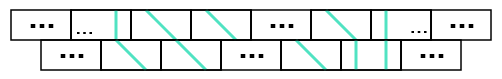}
    \end{center}
This is the same number of left turns as in the base case, so there are more than 2 consecutive left turning strands, so by cases 1 and 2 and because $m - 2$ is the maximum number of consecutive left turns any generation can have, we must have:

\begin{center}
        \includegraphics[scale=0.6]{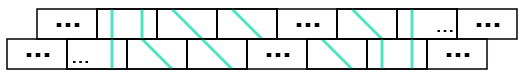}
\end{center}
Thus, the pattern $p$ repeats in $\delta_{i - k}$ on strands $u - (k + 1), u + m - (k + 1)$ and the inductive hypothesis is confirmed.
\begin{flushright}
$\dashv_{\text{Claim}}$
\end{flushright}

By subclaim 1, we have that the pattern $p$ occurs in $\delta_{i - (u - 1)}$ on strands $1$ through $1 + m$. By the bit $\frac{T}{ES}$, we must have:

\begin{center}
        \includegraphics[scale=0.6]{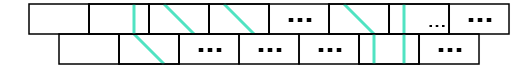}
\end{center}
and by the same reasoning as was given in the proof of the inductive step, we must have:

\begin{center}
        \includegraphics[scale=0.6]{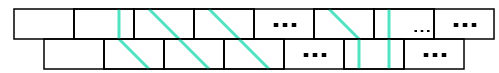}
\end{center}
Thus we have an empty cell followed by at least three left turns, a contradiction to case 1. Recall that case 5 leads to a contradiction because $\frac{T}{TT}$ forces $g$ to contain a crossing.
\end{proof}

\begin{lemma}\thlabel{no3Turns}
    Let $g$ be an $n$-stranded noncrossing glider that does not have speed $1$ or $-1$. If $g$ is positive, then there exists no subpattern of $g$ with:
    \begin{center}
        \includegraphics[scale=0.6]{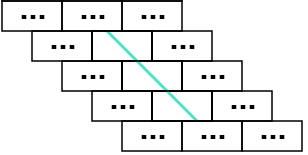}
    \end{center}
    and if $g$ is negative, there exists no subpattern of $g$ with:
  \begin{center}
        \includegraphics[scale=0.6]{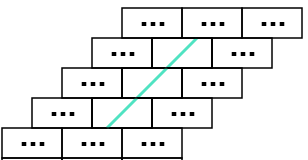}
    \end{center}
\end{lemma}
\begin{proof}
    We prove the result for positive gliders. The negative case follows analogously. First, for contradiction, suppose there exists such a $g$ with the subpattern. We refer to
    \begin{center}
        \includegraphics[scale=0.6]{scaStuff/mchafig1.png}
    \end{center}
    as the 3-turn if it is a subpattern of $g$. There are three options for the highlighted cell $C_{i, j}$ in generation $i$, as $\frac{S}{ET}$ and $\frac{S}{TE}$ are in some turning rule of $g$ by \thref{mustHaveSET}:
      \begin{center}
        \includegraphics[scale=0.6]{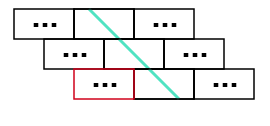}
    \end{center}
    These options are $C_{i, j} = [s_k^{(u)}, n^{\left(\varnothing\right)}_{k+1}], C_{i, j} = [n_k^{\left(\varnothing\right)}, l_{k+1}^{(u)}], C_{i, j} = [n_k^{\left(\varnothing\right)}, n_k^{\left(\varnothing\right)}], C_{i, j} = [n^{\left(\varnothing\right)}_{k}, s_{k+1}^{(u)}], C_{i, j} = [s_k^{(u)}, s_{k+1}^{(u+1)}],$ and $C_{i, j} = [r_k^{(u)}, n^{\left(\varnothing\right)}_{k+1}]$ for some $u, k\geq 1$. Notice that the last three cases are immediately eliminated by the placement of the highlighted cell. The third case is also ruled out by the fact $g$ does not have speed 1 and thus that a turning rule of $g$ has bit $\frac{S}{ET}$ by \thref{mustHaveSET}.  
    \hfill\break
    \textbf{Claim 1:} All turning rules of $g$ contain the bit $\frac{T}{TT}$ and $[n_j^{\left(\varnothing\right)},l_{j+1}^{(u)}], [n^{\left(\varnothing\right)}_{j+2},l_{j+3}^{(u+1)}]$ is such that $[n^{\left(\varnothing\right)}_{j+2},l_{j+3}^{(u+1)}]$ is a subpattern of the 3-turn.
    \hfill\break
    \textbf{Proof:} If $C_{i, j} = [s_k^{(u)}, n^{\left(\varnothing\right)}_{k+1}]$ for some $k, u\geq 1$, we must have the bit $\frac{T}{ST}$. Thus, to avoid a crossing, we must have:
      \begin{center}
        \includegraphics[scale=0.6]{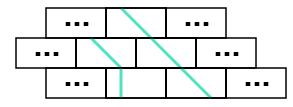}
    \end{center}
    Thus every turning rule of $g$ has the bit $\frac{T}{TT}$ in this case.
    If we have $C_{i, j} = [n_k^{\left(\varnothing\right)}, l_{k+1}^{(u)}]$, we must have the bit $\frac{T}{TT}$ by the next generation.
    \begin{flushright}
$\dashv_{\text{Claim}}$
\end{flushright}

\hfill\break
By claim 1 and \thref{noconsective3Turns}, the subpattern $[n^{\left(\varnothing\right)}_{k}, l_{k+1}^{(u)}][n^{\left(\varnothing\right)}_{k+2}, l_{k+3}^{(u+1)}][n^{\left(\varnothing\right)}_{k+4}, l_{k+5}^{(u+2)}]$ does not occur as a subpattern of the 3-turn in $g$, but the pattern $[n^{\left(\varnothing\right)}_{k}, l_{k+1}^{(u)}][n^{\left(\varnothing\right)}_{k+2}, l_{k+3}^{(u+1)}]$ must occur as a subpattern of the 3-turn as was shown in the proof of claim 1.

\hfill\break
\textbf{Case 1:} The sublist $[n^{\left(\varnothing\right)}_{k}, n^{\left(\varnothing\right)}_{k+1}][n^{\left(\varnothing\right)}_{k+2}, l_{k+3}^{(u)}][n^{\left(\varnothing\right)}_{k+4}, l_{k+5}^{(u+1)}]$ is such that $[n^{\left(\varnothing\right)}_{k+4}, l_{k+5}^{(u+1)}]$ is a subpattern of the 3-turn. Choose $i\geq 3$. We have:
\begin{center}
        \includegraphics[scale=0.6]{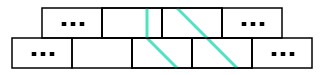}
\end{center}
    as we have the rule $\frac{S}{ET}$. In the first case, if we have $\frac{S}{ST}$, we get:
\begin{center}
        \includegraphics[scale=0.6]{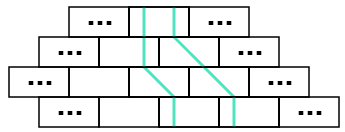}
\end{center}
    which contradicts our assumption that that $[n^{\left(\varnothing\right)}_{k+4}, l_{k+5}^{(u+1)}]$ is a subpattern of a 3-turn. If we instead have $\frac{T}{ST}$, we get a crossing in $\delta_{i + 2}$, which contradicts the fact that $g$ contains no crossings. 

\hfill\break
\textbf{Case 2:} $[s^{(u)}_k, n^{\left(\varnothing\right)}_{k+1}][n^{\left(\varnothing\right)}_{k+2}, l_{k+3}^{(u + 1)}][n^{\left(\varnothing\right)}_{k+4}, l_{k+5}^{(u+2)}]$ is such that $[n^{\left(\varnothing\right)}_{k+4}, l_{k+5}^{(u+2)}]$ is a subpattern of the 3-turn.

Choose $i\geq 2$ such that $\delta_i = \dots[s^{(u)}_k, n^{\left(\varnothing\right)}_{k+1}][n^{\left(\varnothing\right)}_{k+2}, l_{k+3}^{(u + 1)}][n^{\left(\varnothing\right)}_{k+4}, l_{k+5}^{(u+2)}]\dots$ and $[n^{\left(\varnothing\right)}_{k+4}, l_{k+5}^{(u+2)}]$ is a subpattern of the 3-turn. Because we have $\frac{T}{TT}$, we have:
\begin{center}
        \includegraphics[scale=0.6]{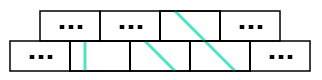}
\end{center}
   If strands $u$ and $u +1$ turn, we get:

\begin{center}
        \includegraphics[scale=0.6]{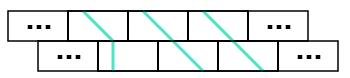}
\end{center}
   This is a contradiction because $[n^{\left(\varnothing\right)}_{k}, l_{k+1}^{(u)}][n^{\left(\varnothing\right)}_{k+2}, l_{k+3}^{(u+1)}][n^{\left(\varnothing\right)}_{k+4}, l_{k+5}^{(u+2)}]$ occurs nowhere in $g$. If strand $u$ turns but strand $u + 1$ does not, we get:
\begin{center}
        \includegraphics[scale=0.6]{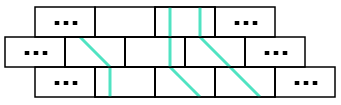}
\end{center}
   Thus, we have $\frac{T}{TT}, \frac{S}{ST}, \frac{S}{ET}$, so we must have:
\begin{center}
        \includegraphics[scale=0.6]{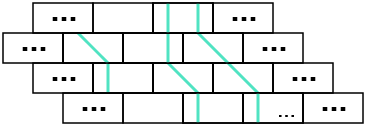}
\end{center}
   This is a contradiction to the assumption that $[n^{\left(\varnothing\right)}_{k+4}, l_{k+5}^{(u+2)}]$ occurs as a subpattern of the 3-turn. 
   
   If instead strand $u$ does not turn, we must consider the case a turning rule has $\frac{T}{ST}$ and the case where a turning rule has $\frac{S}{ST}$. In the first case, we get
\begin{center}
        \includegraphics[scale=0.9]{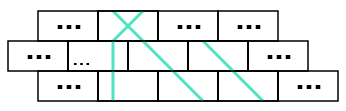}
\end{center}
   a contradiction to the fact that $g$ does not contain any crossings. In the second case, we get:

\begin{center}
        \includegraphics[scale=0.6]{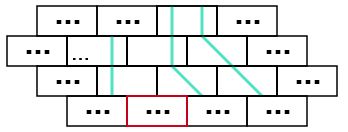}
\end{center}
   Notice that the outlined cell must be empty. Thus, because a turning rule has $\frac{S}{ST}$, we get:

\begin{center}
        \includegraphics[scale=0.6]{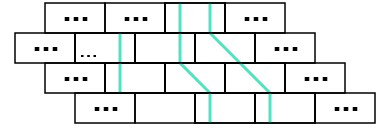}
\end{center}
   This is a contradiction to the fact $[n^{\left(\varnothing\right)}_{k+4}, l_{k+5}^{(u+2)}]$ is a subpattern of the 3-turn.

   Thus, we reach contradictions in all subcases of case 2.

\hfill\break
\textbf{Case 3:} We now consider the case where the subpattern in $\delta_i$, $[n_j^{\left(\varnothing\right)}, s_{j+1}^{(u)}][n^{\left(\varnothing\right)}_{j+2}, l_{j+3}^{(u+1)}][n^{\left(\varnothing\right)}_{j+4}, l_{j+5}^{(u+2)}]$ is such that $[n^{\left(\varnothing\right)}_{j+4}, l_{j+5}^{(u+2)}]$ is a subpattern of the 3-turn. Choose $i\geq 4$. By $\frac{S}{ST}, \frac{T}{TT}$, we get:
\begin{center}
        \includegraphics[scale=0.6]{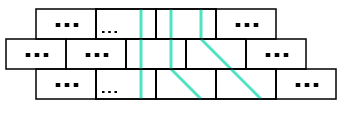}
\end{center}
So consider $\delta_{i-1}$. We must have:

\begin{center}
        \includegraphics[scale=0.6]{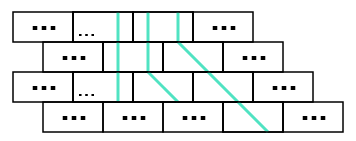}
\end{center}
because $u + 2$ turns in three consecutive generations, one of which is $\delta_i$. So because we have $\frac{S}{ST}$, we get:

\begin{center}
        \includegraphics[scale=0.6]{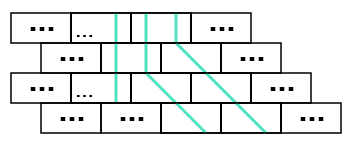}
\end{center}
and then

\begin{center}
        \includegraphics[scale=0.6]{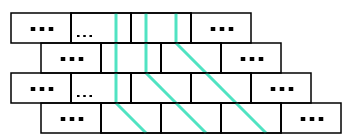}
\end{center}
a contradiction to \thref{noconsective3Turns}.

\end{proof}
\begin{lemma}\thlabel{2TurnPossibilities}
    Let $g$ be a noncrossing positive glider with speed less than $1$. Then, if the subpattern:
    \begin{center}
        \includegraphics[scale=0.6]{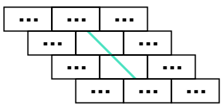}
    \end{center}
    occurs in $g$ on some strand $k$ and $g$ is pure, it must be a subpattern of one of:
    \begin{itemize}
        \item \includegraphics[scale=0.6]{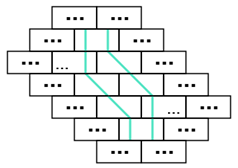}
        \item \includegraphics[scale=0.6]{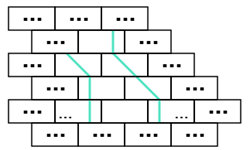}
    \end{itemize}
    If $g$ is not pure, it must be a subpattern of:
\begin{itemize}
    \item 
        \includegraphics[scale=0.5]{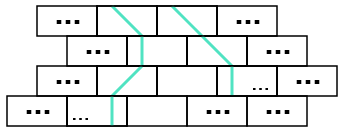}
\end{itemize}
  Let $g$ be a noncrossing negative glider with speed greater than $-1$. Then, if the subpattern:
   \begin{center}
        \includegraphics[scale=0.6]{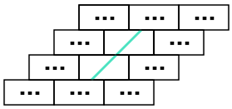}
    \end{center}
  occurs in $g$ and $g$ is pure, it must be a subpattern of one of:
    \begin{itemize}
        \item \includegraphics[scale=0.6]{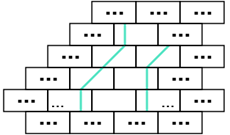}
        \item \includegraphics[scale=0.6]{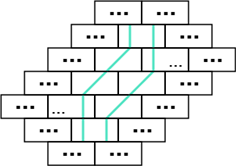}
    \end{itemize}
    if $g$ is not pure, it must be a subpattern of:
\begin{itemize}
    \item \includegraphics[scale=0.7]{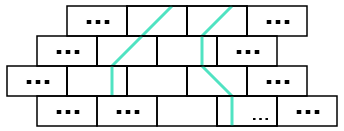}
\end{itemize}
\end{lemma}
\begin{proof}
    As the negative case is analogous, we will only prove the lemma for positive pure gliders. Consider:
    \begin{center}
        \includegraphics[scale=0.6]{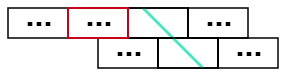}
\end{center}
    We consider some cases:
    \begin{itemize}
        \item $\bm{a_1}:$ The red outlined cell in the above diagram is empty.
        \item $\bm{a_2}:$ The red outlined cell in the above diagram is $[r_j^{(k-1)}, n_{j+1}^{(\varnothing)}]$ for some $j, k\in\mathbb{N}$.
        \item $\bm{a_3}:$ The red outlined cell in the above diagram is $[n_j, l_{j+1}^{(k)}]$ for some $j, k\in\mathbb{N}$.
        \item $\bm{a_4}:$ The red outlined cell in the above diagram is $[s_j^{(k-1)}, n_{j+1}^{(\varnothing)}]$ for some $j, k\in\mathbb{N}$.
        \item $\bm{a_5}:$ The red outlined cell in the above diagram is $[n_j^{(\varnothing)},s_{j+1}^{(k)}]$ for some $j, k\in\mathbb{N}$.
    \end{itemize}
We treat these cases as branch points similar to those used in the proof of the main theorem of section \ref{sec:TwoStrands}. Our decision tree is as follows:
\[\begin{tikzcd}
	&& \bullet \\
	{\bm{a_1}} & {\bm{a_2}} & {\bm{a_3}} & {\bm{a_4}} & {\bm{a_5}} \\
	&&& {\bm{b_1}} && {\bm{b_1}} \\
	&&&& {\bm{c_1}} && {\bm{c_2}} \\
	&&&&& {\bm{d_1}} && {\bm{d_2}}
	\arrow[no head, from=1-3, to=2-1]
	\arrow[no head, from=1-3, to=2-2]
	\arrow[no head, from=1-3, to=2-3]
	\arrow[no head, from=1-3, to=2-4]
	\arrow[no head, from=1-3, to=2-5]
	\arrow[no head, from=2-5, to=3-4]
	\arrow[no head, from=2-5, to=3-6]
	\arrow[no head, from=3-6, to=4-5]
	\arrow[no head, from=3-6, to=4-7]
	\arrow[no head, from=4-7, to=5-6]
	\arrow[no head, from=4-7, to=5-8]
\end{tikzcd}\]
\hfill\break
$\bm{a_1}:$ The outlined cell is empty.
\hfill\break
    This implies:
\begin{center}
        \includegraphics[scale=0.6]{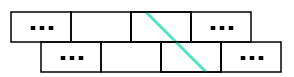}
\end{center}
    a contradiction to $\frac{S}{ET}$ by \thref{mustHaveSET}.

\hfill\break
$\bm{a_2}:$ The outlined cell is $[r_j^{(k-1)}, n_{j+1}^{(\varnothing)}]$ for some $j, k\in\mathbb{N}$.
\hfill\break
We get:
\begin{center}
        \includegraphics[scale=0.6]{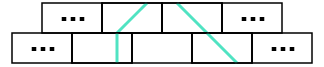}
\end{center}
    by \thref{mustHaveSET}, a contradiction to $\frac{S}{SE}$. 

\hfill\break $\bm{a_3}:$ The outlined cell is $[n_j, l_{j+1}^{(k)}]$ for some $j, k\in\mathbb{N}$.
\hfill\break
Consider the cell below, outlined in green: 
\begin{center}
        \includegraphics[scale=0.925]{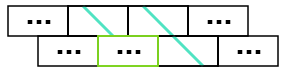}
\end{center}
The outlined cell must contain $[s_{j}^{(k-1)}, n^{\left(\varnothing\right)}_{j+1}]$, or else we would get $\frac{T}{TT}$, which would mean we get:
\begin{center}
        \includegraphics[scale=0.6]{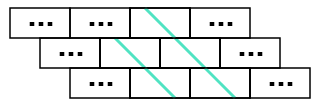}
\end{center}
This contradicts \thref{no3Turns}.

    Now we have
\begin{center}
        \includegraphics[scale=0.6]{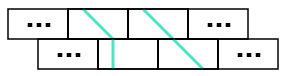}
\end{center}
so by \thref{mustHaveSET} we have:
\begin{center}
        \includegraphics[scale=0.925]{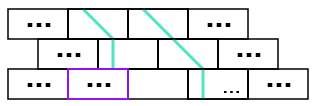}
\end{center}
If we had 
\begin{center}
        \includegraphics[scale=0.6]{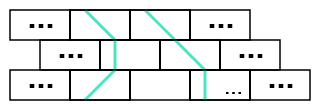}
\end{center}
    by \thref{mustHaveSET} we would have:
\begin{center}
        \includegraphics[scale=0.6]{scaStuff/sca726_9.png}
\end{center}
a contradiction to $\frac{S}{SE}$-- in the case that the glider is pure, we must have:
\begin{center}
        \includegraphics[scale=0.6]{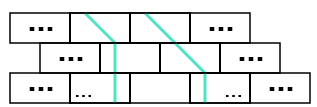}
\end{center}
To avoid a contradiction to \thref{no3Turns}, we must have:

\begin{center}
        \includegraphics[scale=0.6]{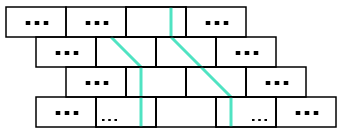}
\end{center}
Thus, the claim is satisfied in this case.

Now consider the case where the glider is not pure. The claim is satisfied in this case as well.



\hfill\break
$\bm{a_4}:$ The outlined cell is $[s_j^{(k-1)}, n_{j+1}^{(\varnothing)}]$ for some $j, k\in\mathbb{N}$.

\hfill\break
We have:
\begin{center}
        \includegraphics[scale=0.6]{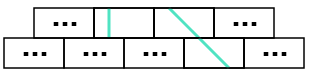}
\end{center}
This means we must have:
\begin{center}
        \includegraphics[scale=0.6]{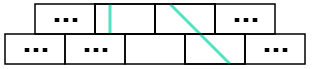}
\end{center}
so all turning rules of the glider have $\frac{T}{ET}$, a contradiction to \thref{mustHaveSET}. 

\hfill\break
$\bm{a_5}:$ The outlined cell is $[n_j^{(\varnothing)},s_{j+1}^{(k)}]$ for some $j, k\in\mathbb{N}$.
\hfill\break
We may finally consider:
\begin{center}
        \includegraphics[scale=0.6]{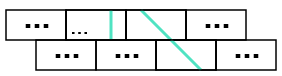}
\end{center}
Immediately, we must have that all turning rules of this glider have $\frac{S}{ST}$ and:
\begin{center}
        \includegraphics[scale=0.6]{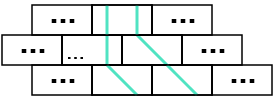}
\end{center}
to avoid a contradiction to the fact that $g$ is a non-crossing glider. To avoid a contradiction to \thref{no3Turns}, we must have:
\begin{center}
        \includegraphics[scale=0.925]{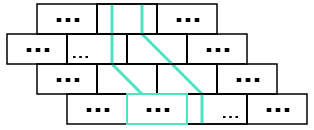}
\end{center}
Notice that there are two possibilities for the cell outlined in cyan:
\begin{itemize}
    \item $\bm{b_1}:$ The outlined cell is $[n_j^{\left(\varnothing\right)}, l_{j+1}^{(k-1)}]$.
    \item $\bm{b_2}:$ The outlined cell is $[s_j^{(k-1)}, n_{j+1}^{\left(\varnothing\right)}]$.
\end{itemize}

\hfill\break
$\bm{b_1}:$ The outlined cell is $[n_j^{\left(\varnothing\right)}, l_{j+1}^{(k-1)}]$.
\hfill\break
To avoid a contradiction to the fact that $g$ is non-crossing, we get:

\begin{center}
        \includegraphics[scale=0.6]{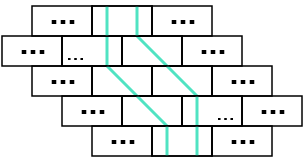}
\end{center}
    which satisfies the claim. 

\hfill\break
$\bm{b_2}:$ The outlined cell is $[s_j^{(k-1)}, n_{j+1}^{\left(\varnothing\right)}]$.
\hfill\break
The outlined cell is $[s_j^{(k-1)}, n^{\left(\varnothing\right)}_{j+1}]$, so we get:

\begin{center}
        \includegraphics[scale=0.6]{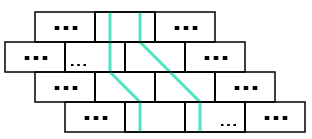}
\end{center}
Thus, all turning rules of $g$ contain $\frac{T}{TT}, \frac{S}{ST}, \frac{T}{SS}$. Because we have $\frac{T}{TT}, \frac{T}{SS}$, we must have one of the following:

\begin{center}
        \includegraphics[scale=0.6]{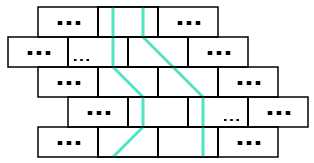}
\end{center}
or
\begin{center}
        \includegraphics[scale=0.6]{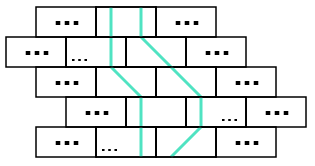}
\end{center}
In the first case, if $g$ is pure, by \thref{mustHaveSET}, we get:

\begin{center}
        \includegraphics[scale=0.6]{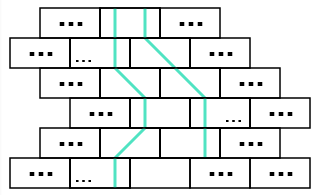}
\end{center}
a contradiction to $\frac{S}{SE}$. If $g$ is not pure, then $g$ has 
\begin{align*}
    \frac{T}{TT}, \frac{S}{ST}, \frac{S}{TS}, \frac{T}{SS}, \frac{S}{ET}, \frac{T}{SE}, \frac{T}{ES}
\end{align*}
in every turning rule. Suppose the strands the pattern occurs on are $k, k + 1$. Because $g$ is a glider, we  may assume the first generation of the pattern is in $\delta_m$ for some $m > 3(n - k-1)$. We now prove the following claim:

\hfill\break
\textbf{Claim 1:} The below pattern occurs
\begin{center}
    \includegraphics[scale=0.9]{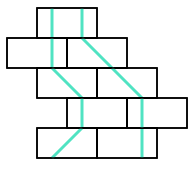}
\end{center}
as a subpattern of $g$ on strands $k + i, k + i + 1$ beginning in generation $\delta_{m - 3i}$ for all $0\leq i\leq n - (k+ 1)$.
\hfill\break
\textbf{Proof:} We proceed by finite induction on $i$.

\hfill\break
\textbf{Base Case:} When $i = 0,$ this follows from the definition of $m, k$.

\hfill\break
\textbf{Inductive Step:} Suppose the pattern occurs on strands $k + i, k + i + 1$ beginning in generation $\delta_{m - 3i}$ for some $0\leq i< n - (k + 1)$. Then, we have:
\begin{center}
    \includegraphics[scale=0.9]{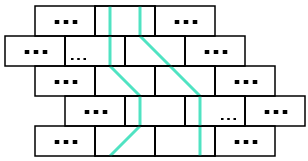}
\end{center}
Because all turning rules of $g$ have $\frac{T}{SS}$, we must have:
\begin{center}
    \includegraphics[scale=0.9]{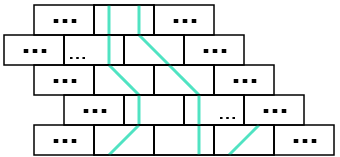}
\end{center}
or:
\begin{center}
    \includegraphics[scale=0.9]{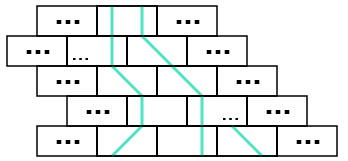}
\end{center}
In the first case, to avoid a crossing, we get:
\begin{center}
    \includegraphics[scale=0.9]{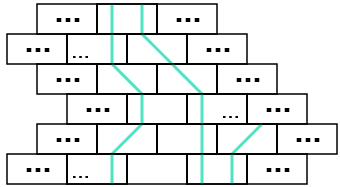}
\end{center}
This is a contradiction to the fact that all turning rules of $g$ have $\frac{T}{ES}$. In the second case, to avoid a crossing:
\begin{center}
    \includegraphics[scale=0.9]{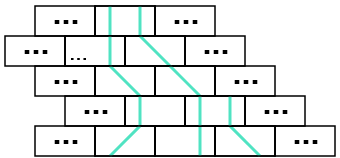}
\end{center}
By $\frac{S}{ET}, \frac{S}{ST}, \frac{T}{TT},$ and $\frac{T}{ES}$ we get:
\begin{center}
    \includegraphics[scale=0.9]{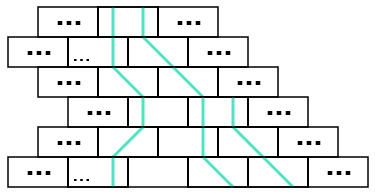}
\end{center}
and then:
\begin{center}
    \includegraphics[scale=0.9]{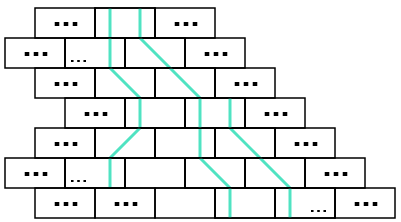}
\end{center}
Now, we must have:
\begin{center}
    \includegraphics[scale=0.9]{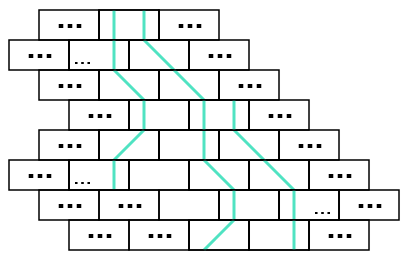}
\end{center}
or
\begin{center}
    \includegraphics[scale=0.9]{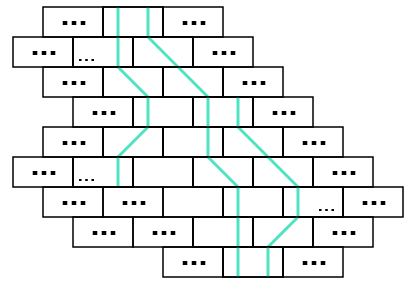}
\end{center}
In the first case, notice that the pattern in the claim occurs on strands $k + i + 1$ and $k + i + 2$, beginning in generation $\delta_{m-3(i+1)}$. Thus, in this case, the inductive hypothesis is confirmed. In the second case, because we have $\frac{T}{ES}, \frac{T}{SS}$, we must have:
\begin{center}
    \includegraphics[scale=0.9]{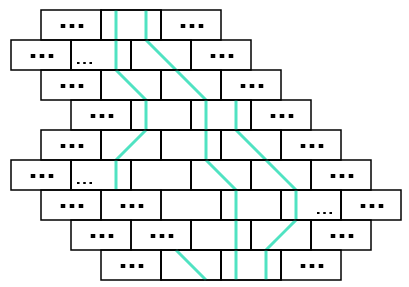}
\end{center}
By continuity, this implies:
\begin{center}
    \includegraphics[scale=0.9]{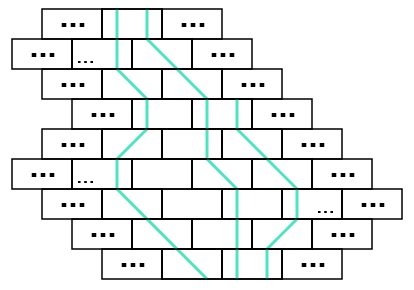}
\end{center}
This contradicts \thref{no3Turns}. 
\begin{flushright}
$\dashv_{\text{Claim}}$
\end{flushright}
By claim 1, 
\begin{center}
    \includegraphics[scale=0.9]{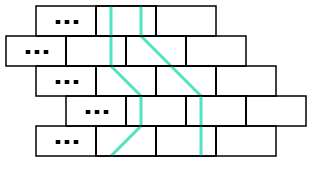}
\end{center}
occurs on strands $n - 1, n$, contradicting the assumption that all turning rules of $g$ have $\frac{T}{SE}$.

In the second case, because $g$ contains no crossings, we have:
\begin{center}
        \includegraphics[scale=0.925]{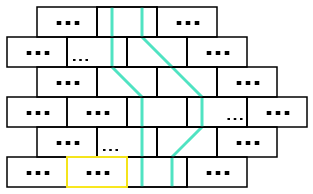}
\end{center}
The above cell, outlined in yellow, must contain a strand, or else we would get a contradiction to $\frac{T}{ES}$. Because we have $\frac{T}{SS}$, we have two cases:
\begin{itemize}
    \item $\bm{c_1}:$ The outlined cell is $[r_j^{(k-2)}, n_{j+1}^{(\varnothing)}]$ for some $j\in\mathbb{N}$.
    \begin{center}
        \includegraphics[scale=0.6]{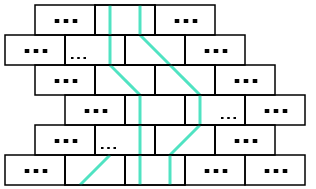}
\end{center}
    \item $\bm{c_2}:$ The outlined cell is $[n_j^{(\varnothing)}, l_{j+1}^{(k-2)}]$ for some $j\in\mathbb{N}$.
    \begin{center}
        \includegraphics[scale=0.6]{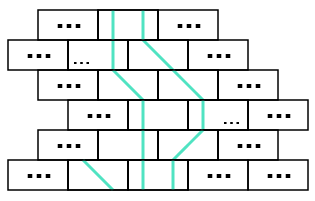}
\end{center}
\end{itemize}

\hfill\break
$\bm{c_1}:$ The outlined cell is $[r_j^{(k-2)}, n_{j+1}^{(\varnothing)}]$ for some $j\in\mathbb{N}$.
\hfill\break
 $g$ does not contain crossings, so:

\begin{center}
        \includegraphics[scale=0.6]{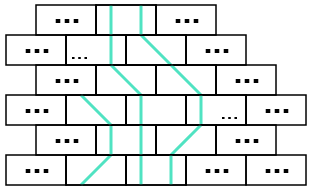}
\end{center}
which forces $g$ to have both $\frac{S}{TS}, \frac{T}{TS}$ in all of its turning rules, a contradiction. 

\hfill\break
$\bm{c_2}:$ The outlined cell is $[n_j^{(\varnothing)}, l_{j+1}^{(k-2)}]$ for some $j\in\mathbb{N}$.
\hfill\break
We have:
\begin{center}
        \includegraphics[scale=0.925]{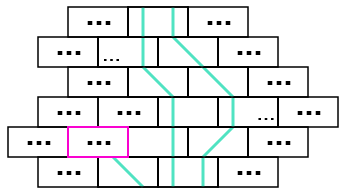}
\end{center}
For the cell outlined in pink above, one of the following must be true:
\begin{itemize}
    \item $\bm{d_1}:$ The outlined cell is $[x, s_{j+1}^{(k-2)}]$ for some $x\in\mathcal{U}$.
    \item $\bm{d_2}:$  The outlined cell is $[n_j^{(\varnothing)}, r_{j+1}^{(k-2)}]$ for some $j\in\mathbb{N}$.
\end{itemize}

\hfill\break
 $\bm{d_1}:$ The outlined cell is $[x, s_{j+1}^{(k-2)}]$ for some $x\in\mathcal{U}$.
 \hfill\break
 We have:
\begin{center}
        \includegraphics[scale=0.6]{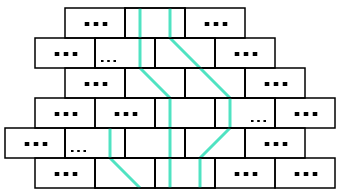}
\end{center}
By $\frac{T}{SS}$, we get:

\begin{center}
        \includegraphics[scale=0.6]{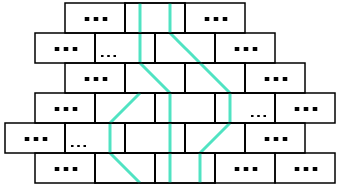}
\end{center}
contradicting the continuity of $g$. 

\hfill\break
$\bm{d_2}:$  The outlined cell is $[n_j^{(\varnothing)}, r_{j+1}^{(k-2)}]$ for some $j\in\mathbb{N}$.
\hfill\break
We have:
\begin{center}
        \includegraphics[scale=0.6]{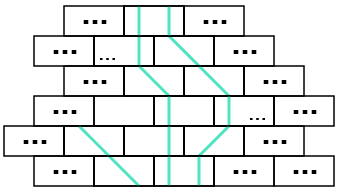}
\end{center}
Because $g$ is a glider, we can choose $z > 6k$ such that the first generation pictured occurs in generation $\delta_z$. Let $h$ be the subpattern:

\begin{center}
        \includegraphics[scale=0.6]{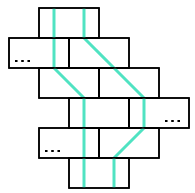}
\end{center}
We will show that the presence of $h$ as a subpattern of $g$ leads to a contradiction.

\hfill\break
    \textbf{Claim 2:} If $z -2(k-2) > 0$, for all $0\leq i\leq k - 2$, pattern $h$ occurs on strands $k - (i +1), k - i$ beginning in generation $\delta_{z-2i}$.
\hfill\break
\textbf{Proof:} We proceed by induction

\hfill\break
    \textbf{Base case:} Let $i = 0$. By the work above, we have:

\begin{center}
        \includegraphics[scale=0.6]{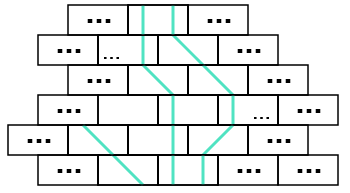}
\end{center}
beginning in generation $\delta_{z - 0}$. We can see $h$ occurs on strands $k - 1, t$ in generation $\delta_{z - 3}$, so the base case is confirmed. We now move to the inductive step:

\hfill\break
    \textbf{Inductive Step:} Let $0\leq i\leq k - 3$, and suppose the claim holds for $i$. Then we have, starting in $\delta_{z-2i}$:
\begin{center}
        \includegraphics[scale=0.6]{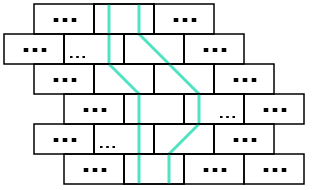}
\end{center}
where the leftmost strand pictured is strand $k - (i + 1)$ and the rightmost strand pictured is $k - i$. By the cases considered earlier, the cell outlined in blue:
\begin{center}
        \includegraphics[scale=0.925]{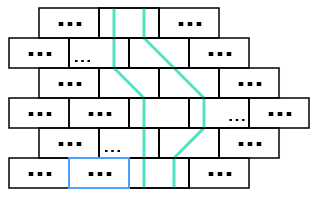}
\end{center}
cannot be empty or $[r_j^{(k-(i+2))}, n^{\left(\varnothing\right)}_{j+1}]$. Because $g$ has $\frac{T}{SS}$ in all of its turning rules, we must instead have:

\begin{center}
        \includegraphics[scale=0.6]{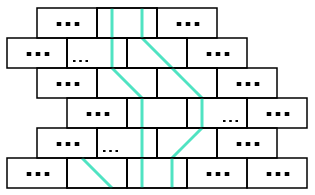}
\end{center}
We have previously proved that, for any strands:

\begin{center}
        \includegraphics[scale=0.6]{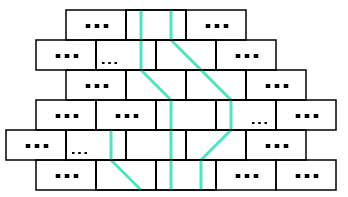}
\end{center}
is impossible, so we must have:

\begin{center}
        \includegraphics[scale=0.6]{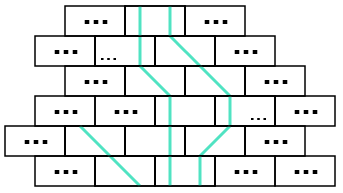}
\end{center}
To avoid a contradiction to \thref{no3Turns}, we have:

\begin{center}
        \includegraphics[scale=0.6]{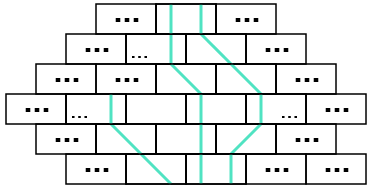}
\end{center}
By the fact that we have $\frac{S}{ET}, \frac{S}{ET},$ and $\frac{T}{TT}$, we must have:
\begin{center}
        \includegraphics[scale=0.6]{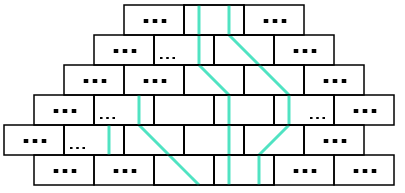}
\end{center}
By the fact that $g$ is continuous, we must have:

\begin{center}
        \includegraphics[scale=0.6]{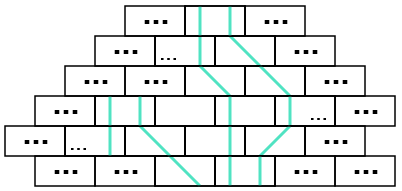}
\end{center}
Because we have $\frac{S}{ST}$, we must have:
\begin{center}
        \includegraphics[scale=0.6]{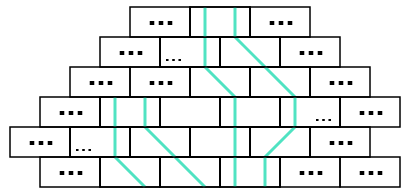}
\end{center}
$g$ contains no crossings, so we have:

\begin{center}
        \includegraphics[scale=0.6]{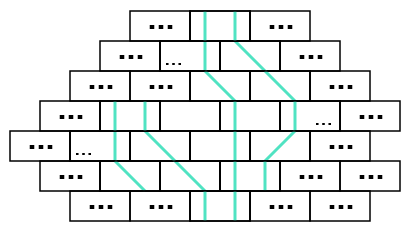}
\end{center}
Because all turning rules of $g$ have $\frac{S}{TS}$, we must have:

\begin{center}
        \includegraphics[scale=0.6]{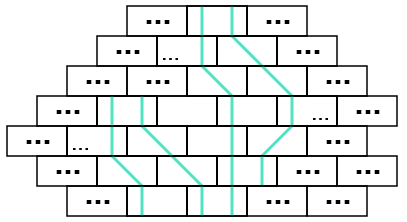}
\end{center}
Notice that $g$ has $\frac{T}{TT}, \frac{T}{SS}$ in all of its turning rules, so we have two options:

\begin{center}
        \includegraphics[scale=0.6]{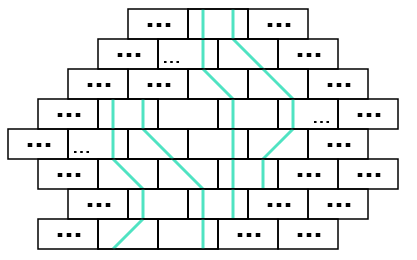}
\end{center}
or
\begin{center}
        \includegraphics[scale=0.6]{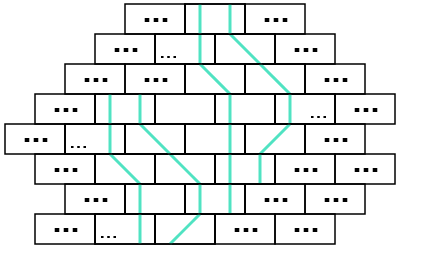}
\end{center}
In the first case, because there exists a turning rule of $g$ with $\frac{S}{TE}$ by \thref{mustHaveSET}, we have:

\begin{center}
        \includegraphics[scale=0.6]{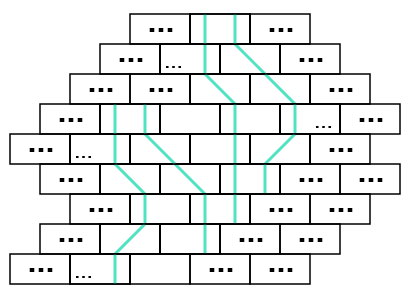}
\end{center}
If $g$ is pure, this is a contradiction to the fact that all turning rules of $g$ have $\frac{S}{SE}$. If $g$ is not pure, notice that, by $\frac{T}{SE}$ and the continuity of $g$, we get:
\begin{center}
    \includegraphics[scale=0.9]{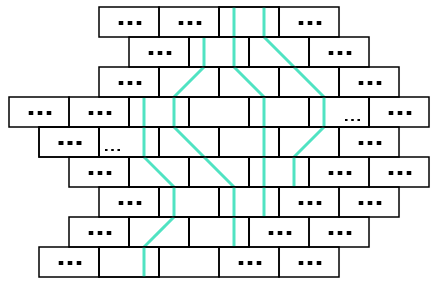}
\end{center}
a contradiction to the fact that all turning rules of $g$ must have $\frac{T}{TT}$. Thus, we must have:

\begin{center}
        \includegraphics[scale=0.6]{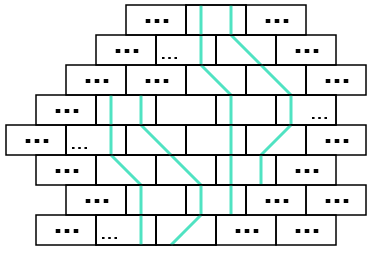}
\end{center}
So by the fact that $g$ cannot contain crossings, we get:

\begin{center}
        \includegraphics[scale=0.6]{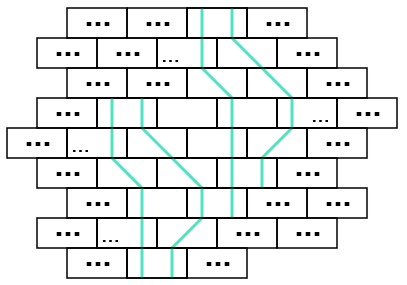}
\end{center}
where $h$ occurs starting in generation $z - 2i - 2 = z - 2(i + 1)$ on strands $k - (i +1) - 1 = k - (i + 2)$ and $k - i - 1 = k - (i + 1)$. Thus, the induction hypothesis is confirmed.
    \begin{flushright}
$\dashv_{\text{Claim}}$
\end{flushright}
Choosing $i= k - 2$ we have $z - 2(k-2) = z-2k-4\geq z-6k >  0$, so by claim 2, we have:
\begin{center}
        \includegraphics[scale=0.6]{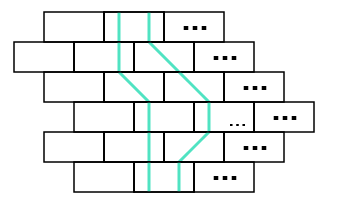}
\end{center}
    Thus, $g$ must have bits $\frac{S}{ES}$ and $\frac{T}{ES}$, a contradiction. Thus, we have proven that in all cases, either the lemma statement is satisfied or we get a contradiction, so we have proven the lemma.
\end{proof}
For the next lemma, recall the notation defined in \thref{def}:
\begin{lemma}\thlabel{pureSubglidersprec}
    Let $g$ be a pure positive $n$-stranded glider $g$. All turning configurations involving strands $n - 1$ and $n$ that occur in $g$ produce the bit $\frac{S}{SE}$. If $g$ is a pure negative glider, all turning configurations involving strands $n - 1$ and $n$ that occur in $g$ produce the bit $\frac{S}{SE}$.
\end{lemma}
\begin{proof}
    We prove the result for pure positive gliders as the negative case follows analogously. If $g$ has speed 1, the result follows from \thref{speedcgliders}, so consider the case where $g$ does not have speed 1. We list all possible 2-stranded turning configurations involving strand $n$:
\begin{center}
        \includegraphics[scale=1.2]{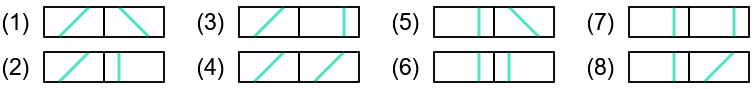}
\end{center}
We proceed by cases:

\hfill\break
    \textbf{Case 1:} We aim to prove the configuration $[r_j^{(n-1)}, n^{\left(\varnothing\right)}_{j+1}][n^{\left(\varnothing\right)}_{j+2}, l_{j+3}^{(n)}]$ 
    \begin{center}
        \includegraphics[scale=0.9]{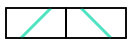}
\end{center}
    cannot occur in $g$. Because $g$ is a glider, there exists $i > 1$ such that $\delta_i = \cdots[r_j^{(n-1)}, n^{\left(\varnothing\right)}_{j+1}][n^{\left(\varnothing\right)}_{j+2}, l_{j+3}^{(n)}]$. Then, because we have rule $\frac{S}{TE}$ and $\frac{S}{ET}$, we must have $\delta_{i-1} = \cdots[\cdots, s_j^{(n-1)}][n^{\left(\varnothing\right)}_{j+1}, n^{\left(\varnothing\right)}_{j+2}][s_{j+3}^{(n)}, n^{\left(\varnothing\right)}_{j+4}]$
    \begin{center}
        \includegraphics[scale=0.9]{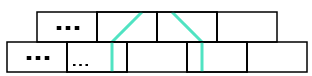}
\end{center}
    This implies that the rule $\frac{T}{SE}$ occurs, a contradiction to the purity of $g$.

\hfill\break
    \textbf{Case 2:} We aim to prove the configuration $[r_j^{(n-1)}, n^{\left(\varnothing\right)}_{j+1}][s_{j+2}^{(n)}, n^{\left(\varnothing\right)}_{j+3}]$
    \begin{center}
        \includegraphics[scale=0.9]{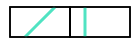}
\end{center}
    cannot occur in $g$. Because $g$ is a glider, we may choose $i\geq 3$ such that $\delta_i = \cdots[r_j^{(n-1)}, n^{\left(\varnothing\right)}_{j+1}][s_{j+2}^{(n)}, n^{\left(\varnothing\right)}_{j+3}]$. Note, because $g$ is a non-crossing glider, we must have $\frac{S}{TS}$ in all turning rules of $g$. We must eliminate all of the following possibilities:
    \begin{itemize}
        \item[a] \includegraphics[scale=0.6]{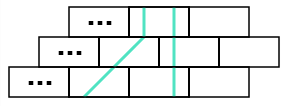}
        \item[b] \includegraphics[scale=0.6]{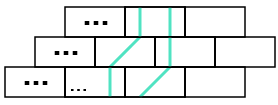}
        \item[c] \includegraphics[scale=0.6]{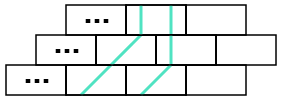}
        \item[d] \includegraphics[scale=0.6]{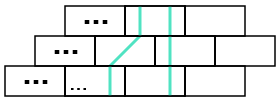}
    \end{itemize}
In case (a), notice that the rules $\frac{T}{TS}$ and $\frac{S}{TS}$ are used, a contradiction. Then, $\delta_{i - 2} = \cdots[s_j^{(n-1)}, s_{j+1}^{(n)}]$ by \thref{mustHaveSET}. $[\delta_{i-2}, \delta_{i-1}, \delta_{i}, \delta_{i+1}]$ is pictured below:
    \begin{center}
        \includegraphics[scale=0.6]{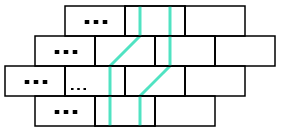}
    \end{center}
    From this we can see the rule $\frac{T}{SE}$ is used, a contradiction to the fact that $g$ is a pure positive glider. Consider case (c). $\delta_{i - 2} = \cdots[n_j^{\left(\varnothing\right)}, s_{j+1}^{(n)}]$ by \thref{mustHaveSET}. $[\delta_{i-2}, \delta_{i-1}, \delta_{i}, \delta_{i+1}]$ is pictured below:
    \begin{center}
        \includegraphics[scale=0.6]{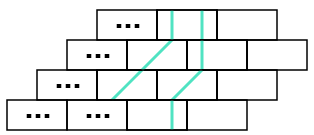}
    \end{center}
    From which we can see that $\frac{T}{SE}$ is used, a contradiction to the fact that $g$ is a pure glider.
    
For case (d), first notice that some turning rule of $g$ has bits $\frac{S}{TS}, \frac{T}{SS}, \frac{T}{ES}, \frac{S}{SE}, \frac{S}{TE}, \frac{S}{ET}$ from the case (d) description, \thref{mustHaveSET}, and the fact that $g$ is a positive pure glider. We now must consider two subcases for any generations following a generation $\delta_x$ such that $\delta_x = \dots[s_{j}^{(n-1)}, s_{j+1}^{(n)}]$. The first subcase is pictured below:
\begin{center}
    \includegraphics[scale=0.6]{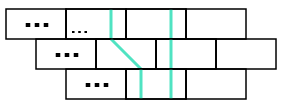}
\end{center}
and the second subcase is pictured below:
\begin{center}
   \includegraphics[scale=0.6]{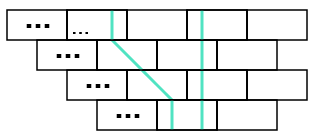}
\end{center}
These subcases are exhaustive by \thref{no3Turns}. We first consider the case where pattern:
\begin{center}
   \includegraphics[scale=0.9]{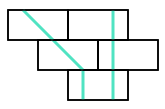}
\end{center}
never occurs in $g$ on strands $n-1$ and $n$. Let $\delta_r$ be a generation of $g$ such that $\delta_r= \cdots[s_{k}^{(n-1)}, s_{k+1}^{(n)}]$ for some $k\in\mathbb{N}$. We consider two cases: The case where the $n-1$st strand goes straight and the case where the $n-1$st strand turns left:
\begin{center}
   \includegraphics[scale=0.9]{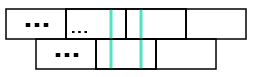}
\end{center}
and
\begin{center}
   \includegraphics[scale=0.9]{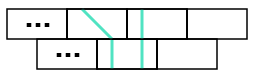}
\end{center}
By \thref{no3Turns} and the fact that there exists a turning rule of $g$ with the bit $\frac{S}{SE}$, these cases are exhaustive. In the first case, by the fact that $g$ contains no crossings, we have:
\begin{center}
   \includegraphics[scale=0.9]{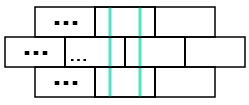}
\end{center}
In the second case, by the fact that the pattern:
\begin{center}
   \includegraphics[scale=0.9]{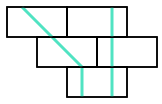}
\end{center}
does not occur in $g$ on strands $n$ and $n - 1$, we get:
\begin{center}
   \includegraphics[scale=0.9]{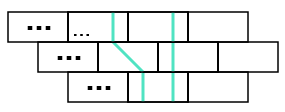}
\end{center}
The next two generations are determined by the rules $\frac{S}{SE}, \frac{T}{SS}, \frac{S}{TS}$:
\begin{center}
    \includegraphics[scale=0.6]{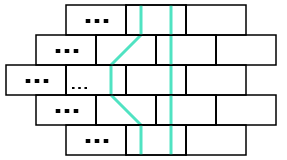}
\end{center}
We have shown that if subcase 2 does not occur in $g$, the $n$th strand of $g$ has speed 0 (as we know $\delta_k = \dots[s_{k}^{(n-1)}, s_{k+1}^{(n)}]$ is a generation of $g$ for some $k\in\mathbb{N}$). This contradicts \thref{gliderSpeed}, as $g$ is a glider and thus does not have speed 0.


Notice that, by the above result and \thref{no3Turns}, the following pattern must occur:
\begin{center}
   \includegraphics[scale=0.6]{scaStuff/sca727_18.png}
\end{center}
on strands $n-1$ and $n$ in $g$. Thus, it suffices to consider the second subcase. By \thref{2TurnPossibilities} and the fact that a turning rule of $g$ has the bit $\frac{S}{TS}$, we get:
\begin{center}
    \includegraphics[scale=0.9]{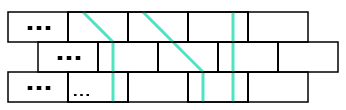}
\end{center}
We may suppose the generation at which the above generations begin is $\delta_m = \dots[s_j^{(n)}, s_{j+1}^{(n+1)}]$ for some $j$.

For contradiction, suppose there exists a positive-speed glider $g$ with strands $n - 1, n$ satisfying the second subcase. Notice that $g$ must have at least 3 strands, otherwise we would reach an immediate contradiction to \thref{mustHaveSET} as this would imply that all turning rules of $g$ have bit $\frac{T}{ET}$. We call the following pattern pattern $p$:
\begin{center}
    \includegraphics[scale=1]{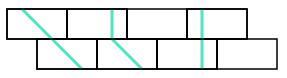}
\end{center}
We now prove the following claim:

\hfill\break
\textbf{Claim:} Pattern $p$ occurs as a sublist in generations $m + 2 + k$ and $m  + k + 3$ on strands $n - k, n-(k + 1),$ and $n - (k + 2)$ for all $0\leq k \leq n - 3$ in $g$.
\hfill\break
\textbf{Proof:} We proceed by induction on $k$.

\hfill\break
\textbf{Base Case:} We aim to prove pattern $p$ occurs on generations $m + 2, m + 3$ on strands $n, n - 1$, and $n - 2$. If strand $n - 2$ does not turn in generation $m + 3$, we get
\begin{center}
    \includegraphics[scale=0.6]{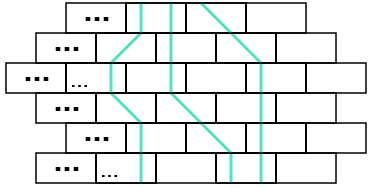}
\end{center}
using the rules $\frac{T}{ES}, \frac{S}{SE}, \frac{S}{TS}, \frac{T}{ST}, \frac{T}{SS}$. Thus, we have $\frac{S}{TT}$ By the rule $\frac{T}{ST}$, the next generation contains a crossing. This is a contradiction to the fact that $g$ is a pure glider. This means strand $n - 2$ must turn in generation $m + 3$, so we get, by \thref{no3Turns} and the bits $\frac{S}{TT}, \frac{S}{SE}$:
\begin{center}
    \includegraphics[scale=0.6]{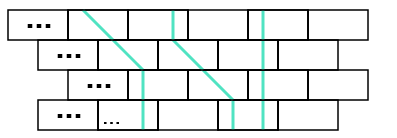}
\end{center}
Pattern $p$ occurs on strands $n, n - 1,$ and $n - 2$ in generations $m + 2, m + 3$.
    

\hfill\break
    \textbf{Inductive Step:} Suppose, for $0\leq k\leq n - 4$, that pattern $p$ occurred on strands $n - k, n - (k + 1), n - (k+2)$ in generations $m + 2 + k$ and $m + 3 + k$ (see the below figure for a visualization)
    \begin{center}
    \includegraphics[scale=0.6]{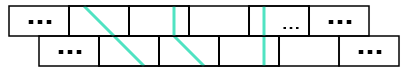}
\end{center}
    We must have:
    \begin{center}
    \includegraphics[scale=0.6]{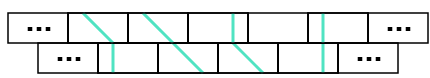}
\end{center}
    by \thref{2TurnPossibilities}, the fact that a turning rule of $g$ has $\frac{S}{TS},$ and \thref{no3Turns}. If $n - (k + 3)$ does not turn in generation $i + k + 4$, we get:
    \begin{center}
    \includegraphics[scale=0.6]{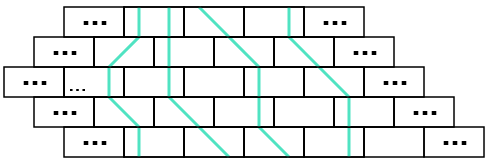}
\end{center}
    by the bits $\frac{T}{SS}, \frac{S}{TT}, \frac{S}{SE}, \frac{T}{ES},$ and $\frac{T}{ST}$. By the bit $\frac{T}{ST}$, strands $n - (k + 1)$ and $n - (k + 2)$ cross in generation $m + k + 7$,   
    a contradiction. Then, we must have:
   \begin{center}
    \includegraphics[scale=0.6]{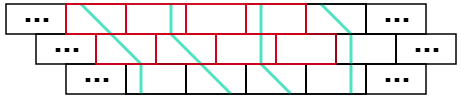}
\end{center}
    and we can see that pattern $p$ (outlined in red) occurs in generations $m + (k + 1) + 2$ and $m + (k + 1) + 3$ on strands $n - (k + 1), n - (k + 2), n - (k + 3)$. Thus, the inductive hypothesis is confirmed.
\begin{center}
        \includegraphics[scale=0.6]{scaStuff/sca727_9.png}
\end{center}
    \begin{flushright}
$\dashv_{\text{Claim}}$
\end{flushright}
Now, by the claim, we have:
\begin{center}
    \includegraphics[scale=0.6]{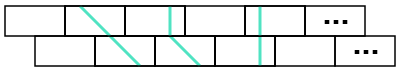}
\end{center}
starting in generation $m + n - 2$ on strands $1, 2, $ and $3$. This contradicts the fact that there exist turning rule of $g$ with $\frac{S}{ET}$ (\thref{mustHaveSET}).



\hfill\break
\textbf{Case 3:} We have the following configuration:
\begin{center}
        \includegraphics[scale=0.9]{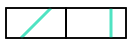}
\end{center}
Choose $i\geq 2$. Then, $\delta_i = \cdots[r_j^{(n-1)}, n^{\left(\varnothing\right)}_{j+1}][n^{\left(\varnothing\right)}_{j+2}, s_{j+3}^{(n)}]$. By \thref{mustHaveSET}, 
\begin{align*}
    \delta_{i - 1} = \cdots[\cdots, s_j^{(n-1)}][n^{\left(\varnothing\right)}_{j+1}, n^{\left(\varnothing\right)}_{j+2}][\cdots]
\end{align*}
\begin{center}
        \includegraphics[scale=0.9]{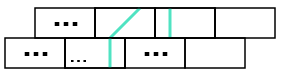}
\end{center}
This implies that the rule $\frac{T}{SE}$ is used, a contradiction.

\hfill\break
\textbf{Case 4:} Consider the configuration:
\begin{center}
        \includegraphics[scale=0.9]{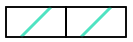}
\end{center}
Choose $i\geq 2$. Then, $\delta_i = \cdots[r_j^{(n-1)}, n^{\left(\varnothing\right)}_{j+1}][r_j^{(n-1)}, n^{\left(\varnothing\right)}_{j+1}]$. By \thref{mustHaveSET}, $\delta_{i - 1} = \cdots[\cdots, s_j^{(n-1)}][n^{\left(\varnothing\right)}_{j+1}, n^{\left(\varnothing\right)}_{j+2}]$.
\begin{center}
        \includegraphics[scale=0.9]{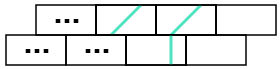}
\end{center}
This implies $\frac{T}{SE}$ is used, a contradiction.

\hfill\break
\textbf{Case 5:} Consider the configuration:
\begin{center}
        \includegraphics[scale=0.9]{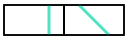}
\end{center}
We must have $\frac{S}{ST}$ in the turning rule of $g$ by the fact that $g$ is a non-crossing glider. The corresponding configuration of $g_{n-1}^l$ is $\frac{S}{SE}$, so the claim is satisfied.

\hfill\break
    \textbf{Case 6:} We may consider the configuration:
    \begin{center}
        \includegraphics[scale=0.9]{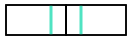}
\end{center}
    Because $g$ is a non-crossing glider, we must have $\frac{S}{SS}$. The corresponding configuration of $g_{n-1}^l$ is $\frac{S}{SE}$, so the bit $\frac{S}{SE}$ is produced. 

\hfill\break
    \textbf{Case 7:} Consider the configuration:
    \begin{center}
        \includegraphics[scale=0.9]{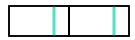}
\end{center}
    If this configuration occurs in $\delta_i$ we must have either $\delta_{i+1} = \cdots, [r_j^{(n-1)}, n^{\left(\varnothing\right)}_{j+1}], [s_{j+2}^{(n)}, n^{\left(\varnothing\right)}_{j+3}]$
    \begin{center}
        \includegraphics[scale=0.9]{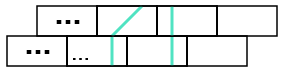}
\end{center}
    or $\delta_{i+1} = \cdots, [s_j^{(n-1)}, n^{\left(\varnothing\right)}_{j+1}], [s_{j+2}^{(n)}, n^{\left(\varnothing\right)}_{j+3}]$
    \begin{center}
        \includegraphics[scale=0.9]{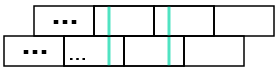}
\end{center}
    It was proven that the first case cannot occur in case 2. In the second case, the corresponding configuration in $g_{n-1}$ is $\delta_i' = \cdots, [n_j^{\left(\varnothing\right)}, s_{j+1}^{(n-1)}], \delta_{i+1}' = \cdots, [s_{j}^{(n)}, n^{\left(\varnothing\right)}_{j+1}]$
    \begin{center}
        \includegraphics[scale=0.9]{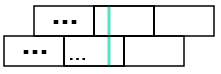}
\end{center}
    which means $g_{n-1}^l$ has $\frac{S}{SE}$.

\hfill\break
    \textbf{Case 8:} We aim to prove the configuration $[x_j, s_{j+1}^{(n-1)}][r_{j+2}^{(n)}, n^{\left(\varnothing\right)}_{j+3}]$
    \begin{center}
        \includegraphics[scale=0.9]{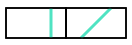}
\end{center}
    cannot occur in a positive pure glider. We can choose $\delta_i = \cdots, [x_j, s_{j+1}^{(n-1)}][r_{j+2}^{(n)}, n^{\left(\varnothing\right)}_{j+3}]$ such that $i > 1$ by the fact that $g$ is a glider. Note that because pure gliders do not contain crossings, if $\delta_i = \cdots, [x_j, s_{j+1}^{(n-1)}][r_{j+2}^{(n)}, n^{\left(\varnothing\right)}_{j+3}]$, then $\delta_{i-1} =\cdots,[s_{k}^{(n-1)}, s_{k+1}^{(n)}]$
    \begin{center}
        \includegraphics[scale=0.9]{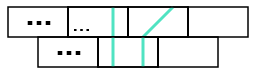}
\end{center}
    so $g$ uses the rule $\frac{T}{SE}$. This is a contradiction because $g$ has turning rules with $\frac{S}{SE}$, so the configuration $[x_j, s_{j+1}^{(n-1)}][r_{j+2}^{(n)}, n^{\left(\varnothing\right)}_{j+3}]$ cannot occur in $g$.
\end{proof}
\begin{lemma}\thlabel{pureSubgliders}
    For any pure positive $n$-stranded glider $g$, $g_{n-1}^l$ is also a pure positive glider under all the turning rules of $g$ with $\frac{S}{SE}$. For any pure negative $n$-stranded glider $g$, $g_{n-1}^r$ is also a pure negative glider under all the turning rules of $g$ with $\frac{S}{ES}$. 
\end{lemma}
\begin{proof}
    We will prove this result only for the pure positive case, as the other case follows analogously. We consider two cases: The case where $g$ has speed $1$ and the case where $g$ does not have speed $1$. In the first case, by \thref{speedcgliders}, $g_{n-1}^l$ is a nested speed $1$ glider on $n - 1$ strands. Thus, also by \thref{speedcgliders}, $g_{n-1}^l$ has the same set of turning rules as $g$. In the second case, all turning configurations involving strands $n - 1$ and $n$ that occur in $g$ produce the same bit in $g_{n-1}^l$ as they did in $g$ or produce the bit $\frac{S}{SE}$ by \thref{pureSubglidersprec}.
    
    Now, let $t$ be a turning rule of $g$ with the bit $\frac{S}{SE}$. For contradiction, suppose $g_{n-1}^l$ is not a positive pure SCA glider. By the fact that $g_{n-1}^l$ is a repeating grid pattern, \thref{gliderSpeed}, and the definition of a pure glider, $g_{n-1}^l$ must not have $\frac{S}{SE}$ in any of its turning rules or $g_{n-1}^l$ must not have a turning rule. If $g_{n-1}^l$ does not have $\frac{S}{SE}$ in any of its turning rules but does have a turning rule, $\frac{T}{SE}$ is in all of $g_{n-1}^l$'s turning rules, and the subpattern:
\begin{center}
        \includegraphics[scale=0.9]{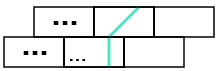}
\end{center}
occurs in $g_{n-1}^l$. The pattern cannot occur on any strand $k$ with $k < n - 1$, or else the pattern would be a subpattern of $g$, a contradiction to the fact that $g$ is pure. This contradicts \thref{pureSubglidersprec}.

Now consider the case where $g_{n-1}^l$ does not have a turning rule. There must be some subpattern of $g_{n-1}^l$ which generates a bit that conflicts with $t$. Consider the corresponding turning configuration:
\begin{center}
        \includegraphics[scale=0.9]{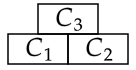}
\end{center}
If strand $n-1$ appears in $C_2$, we reach a contradiction as the same configuration must occur in $g$. If the configuration is a 1-stranded configuration for a strand other than $n-1$, we reach a contradiction for the same reason. Thus, by \thref{pureSubglidersprec}, the configuration must be:
\begin{center}
        \includegraphics[scale=0.9]{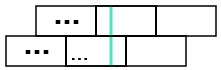}
\end{center}
which produces the bit $\frac{S}{SE}$. This does not conflict with $t$, a contradiction.
\end{proof}
We now introduce some new terminology:
\begin{definition}
    A \textbf{positive nested} $n$\textbf{-stranded glider under} $t$ is a glider $g$ with positive speed such that $g_k^l$ is a glider for all $1\leq k\leq n$ and there exists some turning rule $t$ such that $t$ is a turning rule for $g_k^l$ for all $1\leq k\leq n$.
\end{definition}
Similarly,
\begin{definition}
    A \textbf{negative nested} $n$\textbf{-stranded glider under} $t$ is a glider $g$ with negative speed such that $g_k^r$ is a glider for all $1\leq k\leq n$ and there exists some turning rule $t$ such that $t$ is a turning rule for $g_k^r$ for all $1\leq k\leq n$.
\end{definition}
If $t$ is not specified, we will say $g$ is a \textbf{positive nested} or \textbf{negative nested} $n$\textbf{-stranded glider under a turning rule}. A \textbf{nested} $n$\textbf{-stranded glider under a turning rule} is a positive nested $n$-stranded glider under a turning rule or a negative nested $n$ stranded glider under a turning rule.
\begin{theorem}\thlabel{pureisNested}
    Let $g$ be a pure glider. Then $g$ is a nested $n$-stranded glider under a turning rule.
\end{theorem}
\begin{proof}
    As usual, we will prove this result only for the pure positive case, as the other case follows analogously. Let $g$ be a pure positive glider. 

\hfill\break
\textbf{Claim:} $g_{n-k}^l$ is a pure positive glider under all turning rules of $g$ with bit $\frac{S}{SE}$ for all $0\leq k\leq n-1$. 
\hfill\break
\textbf{Proof:} The proof is by finite induction on $k$.

    \hfill\break
    \textbf{Base Case}: Let $k = 0$. By assumption, $g_{n-k}^l = g$ is a pure positive glider under all the turning rules of $g$ with the bit $\frac{S}{SE}$.
    
\hfill\break
    \textbf{Inductive Step}: Suppose, for some $k < n - 1$, that $g_{n-k}^l$ is a pure positive glider with all turning rules of $g$ such that $\frac{S}{SE}$. Now consider $g_{n - (k+1)}^l = g_{(n - k) - 1}^l$. Because $n - k > 1$, by \thref{pureSubgliders}, $g_{(n - k) - 1}^l$  is a pure positive glider under all turning rules of $g_{n-k}^l$ with $\frac{S}{SE}$. $g_{n-k}^l$ is a pure positive glider under all turning rules of $g$ with $\frac{S}{SE}$, so $g_{(n - k) - 1}^l = g_{n - (k + 1)}^l$ is a pure positive glider under all the turning rules of $g$ with $\frac{S}{SE}$. Thus, the induction hypothesis is confirmed.
\begin{flushright}
$\dashv_{\text{Claim}}$
\end{flushright}
$g$ is a pure positive glider, so there exists a turning rule $t$ with bit $\frac{S}{SE}$ under which $g$ is a glider. By the claim, $g_{k}^l$ is a glider under $t$ for all $1\leq k\leq n$.
\end{proof}
\begin{theorem}\thlabel{pureAndNested}
    Let $g$ be an $n-$stranded noncrossing glider. The following are equivalent:
    \begin{enumerate}
        \item[1.] $g$ is pure.
        \item[2.] $g$ is a nested $n-$stranded glider under a turning rule.
        \item[3.] $g$ is a nested $n-$stranded glider.
    \end{enumerate}
\end{theorem}
\begin{proof}
 $(1)\implies(2):$ Follows from \thref{pureisNested}.
 \hfill\break
 $(2)\implies(3):$ This follows from the definition of a $n-$stranded nested glider.
 \hfill\break
 $(3)\implies (1):$ We prove the result in the case where $g$ is positive, as the negative case is analogous. The case where $g$ is speed 1 follows trivially, so we assume $g$ does not have speed 1. For contradiction, suppose there exists a $n-$stranded nested noncrossing glider $g$ that has $\frac{T}{SE}$ in all of its turning rules. Then, notice that a sublist of the form:
\begin{center}
    \includegraphics[scale=0.9]{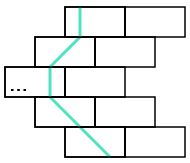}
\end{center}
 must always occur on strand $n$ when the $n$th strand moves exactly two positions to the left over two consecutive generations, and a sublist of the form:
\begin{center}
    \includegraphics[scale=0.9]{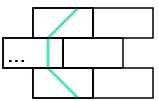}
\end{center}
 must occur on strand $n$ when the $n$th strand moves exactly one position to the left. Thus, the $n$th strand of $g$ has valuated speed at most $\nu(\frac{1}{5}) = \frac{1}{5}$. By \thref{gliderSpeed}, $g$ has valuated speed at most $\nu(\frac{1}{5}) = \frac{1}{5}$. $g$ is nested, so $g_1^l$ has valuated speed at most $\frac{1}{5}$. By \thref{1StrandClassification}, this is a contradiction.
\end{proof}

\begin{corollary}\thlabel{pureGlidercor}
    Let $g$ be a $n$-stranded pure glider. Then there exists a turning rule of $g$ such that $g_m^l$ is an $n$-stranded pure glider for all $1\leq m\leq n$.
\end{corollary}
\begin{proof}
    This follows from the proof of \thref{pureisNested} and the fact that $g$ contains no crossings.

\end{proof}
Note that it is still open as to whether nonpure noncrossing gliders exist.
\begin{theorem}
    Let $g$ be a noncrossing glider. The following are equivalent:
    \begin{itemize}
        \item[1.] $g$ is nonpure.
        \item[2.] $g$ contains the subpattern:
        \begin{center}
        \includegraphics[scale=0.5]{scaStuff/sca726_9.png}
        \end{center}
        if positive and:
        \begin{center}
            \includegraphics[scale=0.7]{scaStuff/sca1-11.png}
        \end{center}
        if negative.
        \end{itemize}
        \item[3.] The set of the turning rules under which $g$ is an SCA pattern is a subset of $\{011100y00|y\in\{0,1\}\}$.
\end{theorem}
\begin{proof}
    $(1)\implies (2):$ We prove this in the case that $g$ is positive, as the negative case is analogous. Consider strand $n$. Notice that, because $g$ is nonpure, it does not have speed 1. Thus, if strand $n$ does not turn in any two consecutive generations, we have:
\begin{center}
    \includegraphics[scale=0.9]{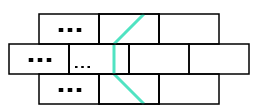}
\end{center}
    so $g$ has speed 0 by \thref{gliderSpeed}. Thus, strand $n$ must turn in at least two consective generations. The claim follows from from \thref{2TurnPossibilities}
\hfill\break
    $(2)\implies (3):$ We prove the result in the positive case, as the negative case is analogous. From this glider and the fact that $g$ is positive, we have bits $\frac{T}{ES}, \frac{T}{SE}, \frac{S}{TE}, \frac{T}{ST}$. By \thref{no3Turns}, we have the bit $\frac{S}{TT}$. Because all rules of $g$ have $\frac{T}{SE}$, it does not have speed $1$ by \thref{speedcgliders}, so it has $\frac{S}{ET}$ in all of its turning rules. $g$ has positive speed, so the following subpattern must occur on strand 1:
\begin{center}
    \includegraphics[scale=0.9]{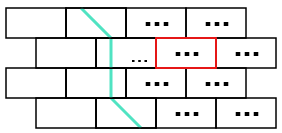}
\end{center}
    Notice that the cell outlined in red must contain a straight strand, as $\frac{T}{ST}$ and $\frac{T}{SE}$. Thus, we have $\frac{S}{SS}$ in all turning rules, so the generic turning rule for $g$ is $011100X00$.
\hfill\break
    $(3)\implies (1):$ Notice that the glider has $\frac{T}{SE}$ in all turning rules because all turning rules under which $g$ is a glider have bit 1 as 1. Thus, $g$ is nonpure.
\end{proof}

\section{Decidability and Pure Gliders}\label{sec:decidabilityAndPure}
In this section, we design an algorithm to enumerate all $n-$stranded pure gliders given a fixed $n$. Recall that the approach we used to classify gliders in \thref{2StrandClassification} cannot be immediately generalized to classify $n$-stranded gliders where $n > 2$ because there could be two different pairs of strands which interact in different generations, and the process we used to obtain \thref{2StrandClassification} relied on all interactions being confined to one generation which had a maximum width. Intuitively, nested gliders allow us to avoid this problem, as we can divide an $n-$stranded nested glider into two parts: The $n-1-$stranded left (right) subpattern and the $n$th (first) strand, from which we can then use a similar process to what we used in \ref{sec:TwoStrands} to classify all $n-$stranded pure gliders. \thref{pureAndNested} allows us to transfer the result for nested gliders to pure gliders. The results we prove before presenting the algorithm are necessary to prove that the algorithm produces all $n-$stranded pure gliders in finite time.

We now introduce some new notation. We denote the set of pure gliders on $n$ strands by $U_n$. Throughout this section, we use the notation ``$n < \infty$" to mean ``$n$ is finite". We now present a proposition:
\begin{proposition}\thlabel{finitePure}
    For all $n\in\mathbb{N}^+$, $|U_n| < \infty$. 
\end{proposition}
\begin{proof}
    For contradiction, suppose otherwise. Note that $|U_1| < \infty$ by \thref{1StrandClassification}. Thus, there exists a least $n$ such that $|U_n| < \infty$ but $|U_{n+1}|$ is infinite.  By \thref{pureAndNested}, \thref{pureGlidercor}, and the fact that $|U_n| < \infty$, there must exist some $h\in U_n$ such that the set $A$ of $a\in U_{n+1}$ with $a_n^l =h$ is infinite. Because there are finitely many turning rules, there exists a turning rule $t$ such that $B = \{g|g\in A\text{ and }g\text{ is a nested glider under } t\}\subseteq A$ such that $|B|$ is infinite. Notice that the period $r$ of $h$ is finite, say it is of length $m$. Then, for each $b\in B$, there are $m$ possible generations for which the strand $n+1$ could be in an adjacent cell to or be in the same cell as strand $n$. Thus, there exists a generation $\delta_i'$ of $r$ such that
    \begin{align*}
        C = \{b|b\in B\text{ and }b\text{ contains a generation in which strand } n\\
        \text{ is in the same cell with or an adjacent cell to strand } \\
        n + 1 \text{ and has the first } n \text{ strands in that generation as they are in }\delta_i'\}
    \end{align*}
    is infinite. For $c\in C$, fix a $\delta_j^c$ such that $\delta_j^c$ is the generation of $c$ with the least $j$ such that the first $n$ strands of $\delta_j^c$ are equal to $\delta_i'$ and strands $n$ and $n +1$ either share a cell or are contained in adjacent cells. Thus, $[\delta_j^c]^l_{n} = \delta_i'$ for all $c$, and $\width(\delta_j^c) \leq \width(\delta_i') + 3$. Because the first $n$ strands of $\delta_j^c$ are fixed for all $c$, there are at most 17 possible equivalence classes of $\delta_j^c$s. Thus, there are at most 17 possible $\delta_k$ such that for every $c$, there exists a $k$ such that $\delta_k = \delta_j^c$ and for all $k$ there exists a $c$ such that $\delta_k = \delta_j^c$. This means that there exists a $\delta_k$ such that:
    \begin{align*}
        D = \{c|c\in C\text{ and there exists a }u\in\mathbb{N}^+\text{ such that } \delta_u \text{ of } c \text{ is such that } \delta_u = \delta_k\}
    \end{align*}
    is infinite. Notice $D\subseteq C\subseteq B$, so all elements of $D$ are nested gliders under $t$, and because every $d\in D$ is a non-crossing glider, every $d\in D$ has crossing rule $000000000$. Thus, for all $d\in D$, $d$ must be some shift of the glider generated by $\delta_j, t, $ and $ 000000000$, so $|D| <\infty$ but $|D|$ is also infinite, a contradiction.
\end{proof}
We now give some intuition for the above proof in the form of an example $\delta_i'$ and the corresponding example $\delta_k$s. If $\delta_i'$ is:
\begin{center}
    \includegraphics[]{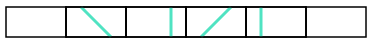}
\end{center}
then there are 5 $\delta_k$, which are exactly the following:
\begin{center}
    \includegraphics[]{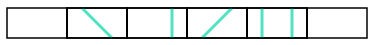}
\end{center}
\begin{center}
    \includegraphics[]{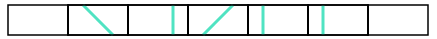}
\end{center}
\begin{center}
    \includegraphics[]{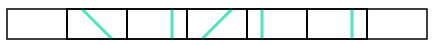}
\end{center}
\begin{center}
    \includegraphics[]{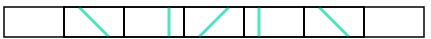}
\end{center}
\begin{center}
    \includegraphics[]{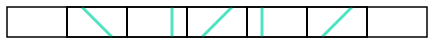}
\end{center}
\begin{lemma}\thlabel{speedcEarly}
    Let $g$ be a noncrossing glider. If $|\per(g)| < 3$, $g$ is the speed $1$ or $-1$ glider.
\end{lemma}
\begin{proof}
    We prove the result when $g$ has positive speed, as the case where $g$ has negative speed is analogous. Notice $|\per(g)| = 1$ or $|\per(g)| = 2$. If $|\per(g)| = 1$, then the leftmost strand must turn left in every generation, so by the definition of speed, $g$ has speed $1$ by \thref{speedcgliders}. For contradiction, assume $|\per(g)| = 2$. The leftmost strand must turn left in exactly one generation of the period, or we would reach a contradiction to the fact that $g$ is a glider or the fact that $|\per(g)|\neq 1$. Then, $\per(g) = [\delta_1,\delta_2]$ where one of the following is true:
    \begin{itemize}
        \item[1.] $\delta_1 = [s_1^{(1)}, \cdots]\cdots, \delta_2= [n_0, l_1^{(1)}]\cdots$
        \begin{center}
        \includegraphics[scale=1]{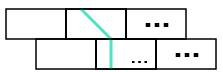}
\end{center}
        \item[2.] $\delta_1 = [n_0, l_1^{(1)}]\cdots, \delta_2 = [n_0, s_1^{(1)}]\cdots$
        \begin{center}
        \includegraphics[scale=1]{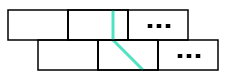}
\end{center}
    \end{itemize}
    In the first case, we have $g = [\delta_1, \delta_2]^\infty$. Notice that the subpattern $[\delta_2, \delta_1]$ cannot be continuous as there is no way for a continuous pattern to have $[s_1^{(1)}, \cdots]$ generated by $[n^{\left(\varnothing\right)}_{-2}, n^{\left(\varnothing\right)}_{-1}]$ and $[n_0, l_1^{(1)}]$. The graphic below illustrates this:
    \begin{center}
        \includegraphics[scale=1]{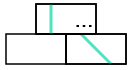}
    \end{center}
    In the second case, we have $g = [\delta_1, \delta_2]^\infty$. Notice there is no way for $[n_0, s_1^{(1)}], [\dots]$ to generate $[n_0, l_1^{(1)}]$ in a continuous pattern. The graphic below illustrates this:

    \begin{center}
        \includegraphics[scale=1]{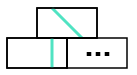}
    \end{center}
Thus, in both cases we have that $g$ is not continuous, a contradiction to the fact that $g$ is a glider.
\end{proof}

\begin{lemma}\thlabel{pigeonhole}
    Let $g$ be an $n+1$-stranded pure glider. Then, the $n+1$st strand must be within two cells of the $n$th strand at least once in each period of $g_{n}^l$.
\end{lemma}
\begin{proof}
    As usual, we will prove the result for the case where $g_n^l$ has positive speed, as the result for the case where it has negative speed follows analogously. For contradiction, suppose there exists some $n+1$-stranded pure glider $g$ for which the claim fails. By \thref{pureAndNested}, $g$ is a non-crossing $n+1$-stranded nested glider under some turning rule $t$. Call $h = g_n^l$. Then, there exists some set of $m\geq |\per(h)|$ consecutive generations for which the $n+1$st strand, $s$, is not within two cells of $h$ and such that, for all $y > m$, any set of $y$ consecutive generations must have at least one generation in which $s$ is within two cells of $h$. 
    
   We first consider the case where $m\geq |\per(h)| > 2$. Notice that $s$ is a glider in those $m$ generations, along with the ones immediately preceding and succeeding the $m$ generations (it is still bordered by an empty cell on the left).
     Let $\nu$ be the valuation. We consider two cases: The case where $\nu(\operatorname{Speed}(s)) \leq \nu(\operatorname{Speed}(h))$ and the case where $\nu(\operatorname{Speed}(s)) > \nu(\operatorname{Speed}(h))$. In the first case, choose the least $i\in\mathbb{N}^+$ such that for all $i< j \leq  i + |\per(h)|$, $s$ is not within two cells of $h$. Now choose the least $r$ following $i$ such that $s$ is within two cells of $h$ in $\delta_r$. Then, there exists a greatest $u\in\mathbb{N}^+$ such that $r-u|\per(h)| > i$. Notice that in generation $\delta_{r-u|\per(h)|}$, $s$ was not within two cells of $h$. Thus, $s$ must have moved at least $u|\per(h)|\nu(\operatorname{Speed}(h)) + 1$ indices to the left in $u|\per(h)|$ generations. Because $\nu(\operatorname{Speed}(s))\leq \nu(\operatorname{Speed}(h)), u|\per(h)|\nu(\operatorname{Speed}(s))\leq u|\per(h)|\nu(\operatorname{Speed}(h))$, so $\lceil u|\per(h)|\nu(\operatorname{Speed}(s))\rceil\leq \lceil u|\per(h)|\nu(\operatorname{Speed}(h))\rceil = u|\per(h)|\nu(\operatorname{Speed}(h))$. Thus, $s$ can move at most $u|\per(h)|\nu(\operatorname{Speed}(h))$ indices to the left in $u|\per(h)|$ generations, a contradiction to the fact that $s$ must have moved at least $u|\per(h)|\operatorname{Speed}(h) + 1$ indices to the left in $u|\per(h)|$ generations.

    We now consider the case where $\nu(\operatorname{Speed}(s)) > \nu(\operatorname{Speed}(h))$. There exists an $i > 2m$ such that strand $s$ is within two cells of $h$ in $\delta_i$ and for the previous $m$ generations, strand $s$ is not within two cells of $h$. Consider $\delta_{i-m - 1}$ --- notice that $s$ is within two cells of $h$ in this generation. Then, there exists a greatest $u\in\mathbb{N}^+$ such that $(i-m-1) + u|\per(h)| < i$, and notice that $s$ is not within two cells of $h$ in $\delta_{(i-m-1) + u|\per(h)|}$. Then, $h$ must move at least $\lfloor u|\per(h)|\nu(\operatorname{Speed}(s))\rfloor + 1$ indices to the left in $u|\per(h)|$ generations to be within two cells of $s$ in $\delta_{i-m - 1}$. It can move at most $\lceil u|\per(h)|\nu(\operatorname{Speed}(h))\rceil = u|\per(h)|\nu(\operatorname{Speed}(h))$ indices to the left in that time, and notice that because $u|\per(h)|\nu(\operatorname{Speed}(h)) < u|\per(h)|\nu(\operatorname{Speed}(s))$, $u|\per(h)|\nu(\operatorname{Speed}(h))\leq \lfloor u|\per(h)|\nu(\operatorname{Speed}(s))\rfloor$, so $u|\per(h)|\nu(\operatorname{Speed}(h)) < \lfloor u|\per(h)|\nu(\operatorname{Speed}(s))\rfloor + 1$, a contradiction.

    Now consider the case where $|\per(h)| \leq 2$. By \thref{speedcEarly}, $h$ must be the speed $1$ or $-1$ glider on $n$ strands. Because $g$ is a positive speed glider, by \thref{mustHaveSET}, all turning rules of $g$ contain $\frac{T}{ET}$. Thus, $g$ must have speed $1$, or else we would reach a contradiction to \thref{mustHaveSET}. By \thref{speedcgliders}:
    \begin{align*}
         g = [[n_0, l_1^{(1)}]\cdots[n^{\left(\varnothing\right)}_{2(n-1)}, l_{2n-1}^{(n)}][n^{\left(\varnothing\right)}_{2n}, l_{2n+1}^{(n+1)}]]^\infty,
    \end{align*}
    a contradiction to the claim that $m \geq|\per(h)| = 1 > 0$. Thus, we have reached contradictions in all cases, so the claim follows.    
\end{proof}
We now give a couple of examples to illustrate the idea of the above proof. In the case where $\nu(\operatorname{Speed}(s))\leq\nu(\operatorname{Speed}(h))$, the intuitive idea is that we derive a contradiction from the fact that the $n+1$st strand is ``too slow" to catch up with the rest of the glider before the period repeats. An example of such a pattern is below:
\begin{center}
    \includegraphics[scale=0.9]{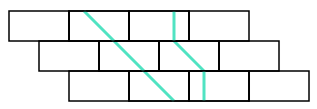}
\end{center}
In the case where $\nu(\operatorname{Speed}(s))>\nu(\operatorname{Speed}(h))$, the  idea is that the subpattern $h$ is ``too slow" to reach the first generation of the period again. An example of what we mean by this is as follows:
\begin{center}
    \includegraphics[scale=0.9]{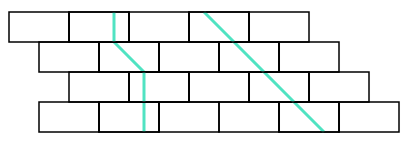}
\end{center}
As we can see in the above figure, the speed of the left subpattern of this grid pattern makes it impossible for the first generation pictured to repeat, unless the second strand were to move to the right in future generations.

We now introduce a few lemmas:
\begin{lemma}\thlabel{shiftSubpattern}
    Let $g$ be a glider. Any subpattern that is a set of consecutive generations of $g$ with the same length as $\per(g)$ is a shift of $\per(g)$.
\end{lemma}
\begin{proof}
    As usual, we shall only consider the case where $g$ is a positive glider. Let $h = [\delta_i, \delta_{i+1},\cdots, \delta_m]\subseteq g$ be such that $|[\delta_i,\delta_{i+1},\cdots,\delta_m]| = |\per(g)|$. Choose the greatest $r$ such that $r < i$ and $\delta_r = \delta_1$. If $r + |\per_g| - 1 < i$, then $\delta_i = \delta_r = \delta_1$ by choice of $r$, so because $g$ is a glider, $h = \per(g)$. Otherwise, there exists some least $t$ such that $\delta_{r + t} = \delta_i, 1\leq t\leq |\per_g| - 1$. Then, let $\per(g) = [\delta_1',\cdots, \delta_u']$. We have $h = [\cdots, \delta_1',\cdots,\delta_{|\per(g) - t|}']$ by the fact that $g$ is a glider. Also by the fact that $g$ is a glider, $h = [\cdots, \delta_{|\per(g)|}', \delta_1',\cdots,\delta_{|\per(g) - t|}']$. By the fact that $|h| = |\per(g)|$, we must have $h = [\delta_{|\per(g) - t|}',\cdots, \delta_{|\per(g)|}', \delta_1',\cdots,\delta_{|\per(g) - t|}']$, which is an cyclic permutation of $\per(g) = [\delta_1',\cdots, \delta_{u}']$ where $u = |\per(g)|$. Thus, $h$ is a shift of $\per(g)$ by definition.
\end{proof}
\begin{lemma}\thlabel{quotientSCA}
    Let $g$ be a SCA pattern under turning rule $t$. Then, if $h = [\delta_i,\cdots, \delta_m]\subseteq g$, $h$ is an SCA pattern under $t$.
\end{lemma}
\begin{proof}
    For contradiction, suppose $h$ is not an SCA pattern under $t$. Then there exist cells $C_{i, j}, C_{i, j + 1}$ in $\delta_i$ of $h$ for some $i$ such that the configuration involving $C_{i, j+1}, C_{i,j},$ $\gen(C_{i, j}, C_{i, j + 1})$ in $\delta_{i+1}$ generates a bit that conflicts with some bit of the turning rule $t$. Because $h = [\delta_i,\cdots, \delta_m]\subseteq g$, $[\delta_i, \delta_{i+1}]\subseteq g$, so the configuration involving $C_{i, j}, C_{i, j + 1}$ and $\gen(C_{i, j}, C_{i, j + 1})$ generates a bit that conflicts with the turning rule of $g$. Thus, $t$ is not a turning rule of $g$, a contradiction.
\end{proof}
We pause to introduce some new notation. For a non-crossing nested glider $g$, let $T_g$ be the set of turning rules under which $g$ is a non-crossing glider. Define $V_n = \{(\per(g), T_g, \width(g))|g\in U_n\}$. Notice that $T_g, \width(g)$ can be extracted from $\per(g)$, but we include them in $V_n$ to simplify the algorithm. The elements of $V_n$ are called \textbf{descriptions} of $n$-stranded gliders. Intuitively, each element of $V_n$ corresponds to an element of $U_n$ and preserves all the necessary information to generate that element, but instead of being an infinite grid pattern, each element of $V_n$ is a finite description of an infinite grid pattern. 

\label{example1}
As an example, we explicitly calculate $V_1$. Notice that the non-crossing nested gliders on one strand are exactly:
\begin{center}
    \includegraphics[]{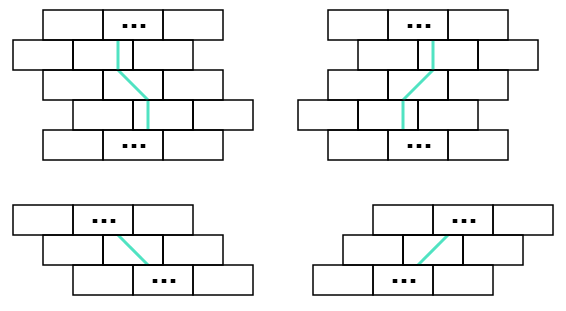}
\end{center}
up to shifts. Thus, the possible periods for a 1-stranded glider are:
\begin{center}
    \includegraphics[]{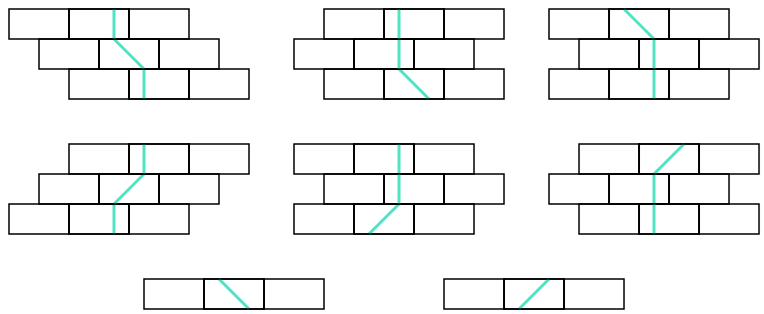}
\end{center}
Define:
\begin{align*}
    T_1 = \{x_00x_2100x_6x_7x_8|x_0,x_2,x_6,x_7,x_8\in\{0,1\}\}\\
    T_2 = \{x_01x_200x_5x_60x_8|x_0,x_2,x_5,x_6,x_8\in\{0,1\}\}\\
    T_3 = \{x_0x_1x_2x_301x_6x_7x_8|x_0,x_1,x_2,x_3,x_6,x_7,x_8\in\{0,1\}\}\\
    T_4 = \{x_0x_1x_2x_30x_5x_61x_8|x_0,x_1,x_2,x_3,x_5,x_6,x_8\in\{0,1\}\}
\end{align*}
Then, by \thref{1StrandClassification}, $V_1$ is the set containing:
\begin{align*}
    ([[s_1^{(1)}, n_2],[n_0,l_1^{(1)}],[n_0, s_1^{(1)}]], T_1, 1)\\
    ([[n_0,l_1^{(1)}],[n_0, s_1^{(1)}],[s_1^{(1)}, n_2]], T_1, 1)\\
    ([[n_0, s_1^{(1)}],[s_1^{(1)}, n_2],[n_0,l_1^{(1)}]], T_1, 1)\\
    ([[n_0, s_1^{(1)}],[r_1^{(1)}, n_2],[s_1^{(1)}, n_2]], T_2, 1)\\
    ([[r_1^{(1)}, n_2],[s_1^{(1)}, n_2],[n_0, s_1^{(1)}]], T_2, 1)\\
    ([[s_1^{(1)}, n_2],[n_0, s_1^{(1)}],[r_1^{(1)}, n_2]], T_2, 1)\\
    ([[n_0, l_1^{(1)}]], T_3,1)\\
    ([[r_1^{(1)}, n_2]], T_4,1).
\end{align*}
\begin{lemma}\thlabel{shiftTurning}
    Let $g$ be a nested non-crossing glider under a turning rule. Let $h$ be a shift of $\per(g)$. Then $T_{h^2} = T_{\per(g)^2}$.
\end{lemma}
\begin{proof}
    We only prove the result for positive gliders as the negative case follows analogously. Without loss of generality, assume $g$ is a positive glider. Let $\per(g) = [\delta_1',\cdots, \delta_m']$. Then, $h = [\delta_i',\cdots, \delta_m',\delta_1',\cdots, \delta_{i-1}']$. If $h = g$ the claim follows, so assume $h\neq g$. Then $i\neq 1$. For contradiction, suppose the claim does not hold. Then, there exists a $t\in T_{h^2}\setminus T_{\per(g)^2}$ or a $t\in T_{\per(g)^2}\setminus T_{h^2}$. In the first case, there exists a $1\leq k\leq n$ such that there exist cells $C_{u, j}, C_{u, j+ 1}$ in $\delta_j$ of $(\per(g)^2)_k^l$ such that $t$ conflicts with the bit generated by the configuration corresponding to $C_{u,j}, C_{u,j+1}$, and $\gen(C_{u, j}, C_{u, j + 1})$ in $\delta_{j + 1}$. We claim $[\delta_j, \delta_{j+1}]\subseteq (h^2)_k^n$. Notice $(h^2)_k^l = [(\delta_i')_k^l,\cdots, (\delta_m')_k^l,(\delta_1')_k^l,\cdots, (\delta_{i-1}')_k^l, (\delta_i')_k^l,\cdots, (\delta_m')_k^l,(\delta_1')_k^l,\cdots, (\delta_{i-1}')_k^l]$. Because $i > 1, [(\delta_1')_k^l,\cdots, (\delta_m')_k^l, (\delta_1')_k^l]\subseteq (h^2)_k^l$. There are three cases for $j: j < m, j = m, j > m$. In the first and third case, $[\delta_j, \delta_{j+1}]\subseteq[(\delta_1')_k^l,\cdots, (\delta_m')_k^l] = (\per(g))_k^l$, so $[\delta_j, \delta_{j+1}]\subseteq [(\delta_1')_k^l,\cdots, (\delta_m')_k^l, (\delta_1')_k^l]\subseteq (h^2)_k^l$, so we reach a contradiction to \thref{quotientSCA}. In the second case, $[\delta_j, \delta_{j + 1}] = [(\delta_m')_k^l, (\delta_{1}')_k^l]\subseteq[(\delta_1')_k^l,\cdots, (\delta_m')_k^l, (\delta_1')_k^l]\subseteq (h^2)_k^l$, which contradicts \thref{quotientSCA}.
    
    Now consider the case where $t\in T_{\per(g)^2}\setminus T_{h^2}$. Then, there exists a $1\leq k\leq n$ such that there exist cells $C_{u, j}, C_{u, j+ 1}$ in $\delta_j$ of $(h^2)_k^l$ such that $t$ conflicts with the bit generated by the configuration corresponding to $C_{u,j}, C_{u, j + 1}$ and $\gen(C_{u, j}, C_{u, j + 1})$ in $\delta_{j + 1}$. There are three cases: $[\delta_j]\subseteq[(\delta_i')_k^l,\cdots, (\delta_{m-1}')_k^l]$, $\delta_j = \delta_m'$, or $[\delta_j]\subseteq[(\delta_1')_k^l,\cdots, (\delta_{i-1}')_k^l, (\delta_i')_k^l,\cdots, (\delta_m')_k^l,(\delta_1')_k^l,\cdots, (\delta_{i-1}')_k^l]$. In the first and third cases, because $[\delta_j, \delta_{j+1}]\subseteq [(\delta_i')_k^l,\cdots, (\delta_{m}')_k^l]\subseteq(\per(g))_k^l\subseteq (\per(g))_k^l$ and $[\delta_j, \delta_{j+1}]\subseteq[(\delta_1')_k^l,\cdots, (\delta_{i-1}')_k^l, (\delta_i')_k^l,\cdots, (\delta_m')_k^l,(\delta_1')_k^l,\cdots, (\delta_{i-1}')_k^l]\subseteq(\per(g))_k^l$, we reach contradictions by \thref{quotientSCA}. In the second case, $[\delta_j, \delta_{j+1}] = [(\delta_m')_k^l, (\delta_1')_k^l]\subseteq(\per(g)^2)_k^l$, so we reach a contradiction by \thref{quotientSCA}. Thus, we have reached contradictions in all cases, so the lemma follows. 
\end{proof}
Notice that we need to consider $h^2$ instead of $h$ to account for the bits of the turning rule needed to make the glider repeat after the last generation of its period.
\begin{lemma}\thlabel{lemmaTwo}
    Let $h$ be an $n$-stranded SCA pattern such that $h^\infty$ is a non-crossing nested glider. Then, $T_{h^2} = T_{h^\infty}$.
\end{lemma}
\begin{proof}
    We prove the result for positive gliders as the negative case follows analogously. Assume $h^\infty$ is a positive glider. We first prove $T_{h^2}\subseteq T_{h^\infty}$. For contradiction, suppose otherwise. Let $h = [\delta_1',\cdots, \delta_m']$. Then, there exists a turning rule $t$ in $T_{h^2}$ that is not in $T_{h^\infty}$. There exists some $1\leq k\leq n$ such that $(h^\infty)_k^l$ is not an SCA pattern under $t$. Then there exist cells $C_{i, j}, C_{i, j + 1}$ in $\delta_i$ for some $i$ such that the bit generated by the configuration corresponding to $C_{i,j}, C_{i,j+1}$, and $\gen(C_{i, j}, C_{i, j + 1})$ in $\delta_{i+1}$ conflicts with the turning rule $t$. Thus, any pattern containing $[\delta_i, \delta_{i+1}]$ is not an SCA pattern under $t$. Choose $j$ as the greatest $j$ such that $\delta_j = \delta_1'$ and $j\leq i$. Then, $[\delta_i, \delta_{i+1}]\subseteq[\delta_j,\cdots,\delta_{j + |h| - 1},\cdots, \delta_{j + 2|h| - 1}] = (h^2)_k^l$ by choice of $j$-- but this implies $h^2$ is not an SCA pattern because $[\delta_i, \delta_{i+1}]\subseteq h^2$, a contradiction. Thus, $T_{h^2}\subseteq T_{h^\infty}$. Now, suppose $t\in T_{h^\infty}$. Let $1\leq k\leq n$, and consider $(h^2)_k^l$. Notice $(h^\infty)_k^l$ is an SCA pattern under $t$, so by \thref{quotientSCA}, $(h^2)_k^l$ is an SCA pattern under $t$. Then, $(h^2)_k^l$ is a non-crossing nested glider for all $1\leq k\leq n$, so $h^2$ is a non-crossing nested glider under $t$, and thus $t\in T_{h^2} $, so $T_{h^\infty}\subseteq T_{h^2}$. 
\end{proof}

To make the notion of the correspondence between $V_n$ and $U_n$ more rigorous, first notice that we can define a bijection $f:U_n \to V_n$ by $f(g) = (\per(g), T_g, \width(g))$. $f$ is surjective by definition, we now show that it is injective. Suppose $f(g_1) = f(g_2)$. Then, $\per(g_1) = \per(g_2)$, so $g_1 = \per(g_1)^\infty = \per(g_2)^\infty = g_2$. Thus, $f$ is a bijection, so $|V_n| = |U_n|$. We claim $f^{-1}$ is $f^{-1}((\per(g), T_g, \width(g))) = \per(g)^\infty$. Let $g\in U_n$. Then, $f^{-1}(f(g)) = f^{-1}((\per(g), T_g, \width(g))) = \per(g)^\infty = g$. As such, to find all pure gliders on $n + 1$ strands given $U_n$, it suffices to find all pure gliders on $n + 1$ strands using $V_n$ instead. We first present a helper algorithm to determine turning rules of a SCA pattern of finite length:

\hfill\break
\textbf{Turning Rule Algorithm:}
Let $s$ be a SCA pattern of finite length. If $s$ is not the empty list, let $\delta_j$ be the $j$th generation of $s$, $1\leq j\leq |s| - 1$. For each $j$, let $i$ range between 1 and $\width(\delta_j)$. Define $\ts(x_i)$ to be $E$ if $x_i = [n_k^{\left(\varnothing\right)}, n^{\left(\varnothing\right)}_{k+1}]$ for some $k$, $S$ if $x_i$ contains $s_k^{(u)}$ for some $k$ and $u$, or $T$ if $x_i$ contains $r_k^{(u)}$ or $l_k^{(u)}$ for some $k$ and $u$. Intuitively, $\operatorname{ts}$ indicates the type of strands in a cell. Each set $A_i$ corresponds to a bit in the turning rule. Notice that we need a union for $A_1$ to deal with the case where the last strand in a generation is straight the case when some strand $k<n$ is straight and is in a cell bordered by a cell containing no strands on the right (recall that, in our description of a generation, the first and last cells in a generation must contain strands so we need to account for cases involving a strand and an empty cell differently based on the number of the strand). The reasoning for why $A_3, A_5$ and $A_7$ need unions in their descriptions is similar.

We may write:
\begin{align*}
    A_0 &= \{\ts(\gen(x_i, x_{i+1}))|\ts(x_i)=S, \ts(x_{i+1}) = S\}\\
    A_1 &= \{\ts(\gen(x_i, x_{i + 1})|\ts(x_i) = S, \ts(x_{i+1}) = E\}\\
    &\cup\{\ts(\gen(x_{\width(\delta_j)}, [n^{\left(\varnothing\right)}_{2\width(\delta_j) + 1}, n^{\left(\varnothing\right)}_{2\width(\delta_j) + 2}])| \ts(x_{\width(\delta_j)}) = S\}\\
    A_2 &= \{\ts(\gen(x_i, x_{i+1}))|\ts(x_i) = S, \ts(x_{i+1}) = T\}\\
    A_3 &= \{\ts(\gen(x_i, x_{i + 1})|\ts(x_i) = E, \ts(x_{i+1}) = S\}\cup\\
    &\{\ts(\gen([n^{\left(\varnothing\right)}_{-2}, n^{\left(\varnothing\right)}_{-1}], x_1))| \ts(x_1) = S\}\\
    A_4 &= \{\ts(\gen(x_i, x_{i+1}))|\ts(x_i)=E, \ts(x_{i+1}) = E\}\\
    A_5 &= \{\ts(\gen(x_i, x_{i + 1})|\ts(x_i) = E, \ts(x_{i+1}) = T\}\cup\\
    &\{\ts(\gen([n^{\left(\varnothing\right)}_{-2}, n^{\left(\varnothing\right)}_{-1}], x_1))| \ts(x_1) = T\}\\
    A_6 &= \{\ts(\gen(x_i, x_{i+1}))|\ts(x_i) = T, \ts(x_{i+1}) = S\}\\
    A_7 &= \{\ts(\gen(x_i, x_{i + 1})|\ts(x_i) = T, \ts(x_{i+1}) = E\}\cup\\
    &\{\ts(\gen(x_{\width(\delta_j)}, [n^{\left(\varnothing\right)}_{2\width(\delta_j) + 1}, n^{\left(\varnothing\right)}_{2\width(\delta_j) + 2}])| \ts(x_{\width(\delta_j)}) = T\}\\
    A_8 &= \{\ts(\gen(x_i, x_{i+1}))|\ts(x_i)=T, \ts(x_{i+1}) = T\}
\end{align*}
We now determine if the pattern has a turning rule, and if so, what the possible turning rules are. If for any $A_i$ with $0\leq i\leq 8$ we have $|A_i| > 1$, then there is no turning rule because the $i$th bit would have to be both 0 and 1. If $E\in A_i$ for any $i\neq 4$, we also do not have a turning rule as this would imply the pattern is not continuous. If $|A_i|\leq 1$ for all $0\leq i\leq 8$, then $s$ has a turning rule. For $1\leq i\leq 8,i\neq 4$, define $s_i$ as:
\begin{center}
    \begin{itemize}
        \item $s_i = 0$ $A_i = \{S\}$
        \item $s_i = 1$ if $A_i = \{T\}$
        \item $s_i = X$ if $A_i =\varnothing$
    \end{itemize}
\end{center}
Note that, if $g$ is continuous, $A_4 = \{E\}$ for every $s$. Thus, if $A_4 = \{E\},$ define $s_4 = 0$. Otherwise, $g$ does not have a turning rule as it is not continuous. Then, a generic turning rule for $s$ is $s_8s_7s_6s_5s_4s_3s_2s_1s_0$.

\hfill\break
We give a small example of this algorithm, applied to the pattern $g = [[n_0, s_1^{(1)}][n_2,l_3^{(2)}],$\\
$[r_1^{(1)},l_2^{(2)}],[s_1^{(1)},s_2^{(2)}]]$, pictured below:
\begin{center}
    \includegraphics[scale=0.9]{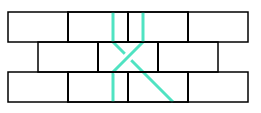}
\end{center}
\begin{center}
\begin{align*}
    A_0 = \varnothing\\
    A_1 = \varnothing\\
    A_2 = \{T\}\\
    A_3 = \varnothing\\
    A_4 = \{E\}\\
    A_5 = \{S\}\\
    A_6 = \varnothing\\
    A_7 = \{S\}\\
    A_8 = \varnothing
\end{align*}
\end{center}
Thus, the algorithm produces the turning rule $XX1X00X0X$ for $g$. Below, we highlight the cells used for each set (except for $A_4$, as that corresponds to the bit that controls whether two empty cells produce a strand in the next generation). 

\hfill\break
$A_2:$
\begin{center}
    \includegraphics[]{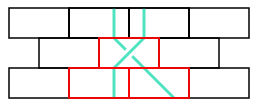}
\end{center}
\hfill\break
$A_5:$
\begin{center}
    \includegraphics[]{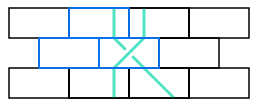}
\end{center}
\hfill\break
$A_7:$
\begin{center}
    \includegraphics[]{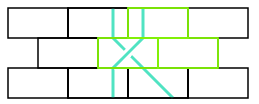}
\end{center}

\hfill\break
We now give a more complicated example. Consider the grid pattern:
\begin{center}
    \includegraphics[]{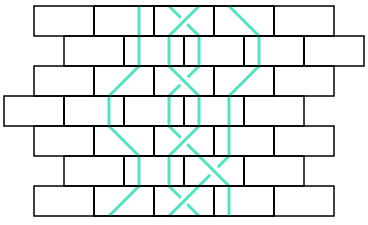}
\end{center}
Notice that:
\begin{align*}
    A_0 = \{T\}\\
    A_1=\{T\}\\
    A_2 = \{T\}\\
    A_3 = \{T, S\}\\
    A_4 = \{E\}\\
    A_5 = \{S\}\\
    A_6 = \{T, S\}\\
    A_7 = \{S\}\\
    A_8 = \{S\}.
\end{align*}
Thus, the grid pattern is not an SCA pattern because it does not have a turning rule, as the sets $A_3$ and $A_6$ have cardinality greater than 1.

\hfill\break
We now define the \textbf{shift set} of a finite grid pattern $g$ to be the set of all grid patterns $h$ such that $h$ is a shift of $g$. We denote this set by $S(g)$. Define the shift set of a glider $g$ to be the set of gliders with period in $S(\per(g))$. Given a generation $\delta$, a turning rule $t$, and a crossing rule $c$, we call the pattern with $k$ generations with turning rule $t$, crossing rule $c$, and first generation $\delta$ the result of \textbf{generating} $\delta$ for $k$ generations.
Finally, we present an algorithm to find $V_{n+1}$ given $V_n$:

\hfill\break
\textbf{Pure Glider Algorithm}
\begin{itemize}
        \item[1.] Enumerate $V_n$ as $\{s_i\}_{i=1}^w$ where $s_i = (\per(g_i), T_{g_i}, \width(g_i))$.
        \item[2.] For each $i$, enumerate $T_{g_i}$ as $\{t_{i, 1},\cdots, t_{i, m_i}\}$ and let $\alpha_i$ be the first generation of $\per(g_i)$. Now define $\{s_{i, j}\}_{i=1, j = 1}^{k, m_i}$ as $s_{i, j} = (\alpha_i,t_{i, j}, \width(g_i), |\per(g_i)|)$.
        \item[3.] For each $i, j$:
        \begin{itemize}
            \item[a.] If $s_{i,j}$ corresponds to a positive glider:
            \begin{itemize}
            \item[i.] There are $v_i<\infty$ many generations with $n + 1$ strands where the first $n$ strands match $\alpha_i$ and the $n + 1$st strand is within two cells to the right of the $n$th strand, so enumerate these generations as $\beta_1^{(i)},\cdots, \beta_{v_i}^{(i)}$. Now, define $\{s_{i, j, k}\}$ by $s_{i, j,k} = (\alpha_i, t_{i, j}, \width(g_i), |\per(g_i)|, \beta_{k}^{(i)})$. 
            \item[ii.] For each $i, j, k$ $s_{i,j,k} = (\alpha_i, t_{i, j}, \width(g_i), |\per(g_i)|, \beta_{k}^{(i)})$, so separately generate $(8^{2 + \width(g_i)} + 1)|\per(g_i)|$ generations of $\alpha_i$ and $\beta_k^{(i)}$ under $t_{i,j}, 00000000$.
            \begin{itemize}
                \item[*] Check if the $n$-stranded left subpattern of the pattern generated by $\beta_k^{(i)}$ matches the pattern generated by $\alpha_i$. If so, move to (**), if not, define $B_{i, j, k} = \varnothing$ (indicating no possible gliders) and move to the next tuple if we have not run out of tuples, or stop if we have.
                \item[**] Check to see if the generated pattern contains any crossings. If so, define $B_{i, j, k} = \varnothing$ and move to  the next tuple if we have not run out of tuples, or stop if we have. If not, move to (***).
                \item[***] Check whether or not there exist two generations $\delta_r, \delta_d, d < r$ of the pattern generated by $\beta_k^{(i)}$ that are identical.  If so, choose the minimal $r$ and $d$, define $h = [\delta_d,\cdots, \delta_r] $.  Calculate $T_{h^2}$ using the following process, for each $1\leq m\leq n + 1$:
                
                Run the turning rule algorithm on $(h^2)_m^l$, and let $T_m$ be the set of turning rules given by substituting $0$ or $1$ in for each $X$ that appears in the output (note that these values need not be identical for $X$s representing different bits).
                
            Define $T_{h^2} = \bigcap_{1\leq m\leq n+ 1} T_m$, and define $B_{i,j,k} = \{(h, T_{h^2}, \width(h))|h\in S([\delta_r,\cdots, \delta_d])\}$. Notice that $T_{h^2}$ contains at least $t_{i,j}$. If no such $\delta_r,\delta_d$ exist, define $B_{i,j,k} = \varnothing$ and move to the next tuple if there are more tuples or stop if there aren't.
    \end{itemize}
        \end{itemize}
        \item[b.] If $s_{i,j}$ corresponds to a negative glider, proceed similarly to part (a).
    \end{itemize}
        \item[5.] Return $\bigcup B_{i,j,k}$.
\end{itemize}

\hfill\break
We must prove that step 5 actually returns $V_{n+1}$. Let $(\per(g), T_g, \width(g))\in V_{n+1}$. Assume $g$ is positive -- the case where $g$ has negative speed is analogous. Then, there exists a least generation $\delta_r$ of $\per(g)$ where the $n + 1$st strand is within two cells of the $n$th strand, and there exists $t\in T_g$ by the definition of a glider. Thus, $(([\delta_r])_n^{l}, t, \width(g_n^l), |\per(g_n^l)|,\delta_r)$ is generated for $(8^{2+\width(g_n^l)} + 1)|\per(g_n^l)|$ generations in step 4. We claim that $\delta_r = \delta_1'$  must repeat (where $\delta_1'$ is the first generation in the generated pattern). By \thref{pigeonhole}, the $n+1$st strand must be within two cells of the $n$th strand at least once in every period of $g_n^l$. By the pigeonhole principle, there must exist a generation in which $\delta_r$ repeats. Thus, take the first $\delta_s'$ in the generated pattern such that $\delta_1' = \delta_s', s > 1$. Notice that, because $g$ is a non-crossing glider and $t$ is a turning rule of the nested glider $g$, $[\delta_1',\cdots, \delta_s'] = [\delta_r,\cdots,\delta_m]$ for some $m > r, \delta_r = \delta_m$. Then, by \thref{shiftSubpattern}, $[\delta_r, \cdots, \delta_{m-1}]$ is a shift of $\per(g)$. By \thref{lemmaTwo}, $T_g = T_{\per(g)^2}$. Thus,  $(\per(g), T_{\per(g)^2}, \width(\per(g)))\in B_{i,j,k}\subseteq\bigcup B_{i,j,k}$ for some $i,j,k$. Thus, $V_{n+1}\subseteq\bigcup B_{i,j,k}$. 

Now let $(h, T_{h^2}, \width(h))\in \bigcup B_{i,j,k}$. Without loss of generality, assume $h^\infty$ is positive. By construction, $h^\infty$ is a nested non-crossing glider with the turning rules $T_{h^\infty}$, which by \thref{lemmaTwo} is $T_{h^2}$. By \thref{pureAndNested}, $h^\infty$ is pure. Thus, $h^\infty\in U_{n+1}$, so $(\per(h^\infty), T_{h^\infty}, \width(h^\infty))\in V_{n+1}$. By construction, $\per(h^\infty) = h$, $T_{h^\infty} = T_{h^2}$, and by the definition of width, $\width(h^\infty) = \width(h)$. Thus, $(h, T_{h^2}, \width(h))\in V_{n+1}$, so $\bigcup B_{i,j,k}\subseteq V_{n+1}$. Thus  $\bigcup B_{i,j,k} =  V_{n+1}$, so step 5 is correct.

Given a description of an $n$-stranded grid pattern $d$, one may determine if $d$ is a pure glider by finding $V_n$ by starting with $V_1$ (constructed on page \pageref{example1}) and iterating the process above and comparing each element to $d$. This process always terminates because $|V_n| = |U_n| < \infty$ for all $n$ (\thref{finitePure}).

\section{Outlook}
This project established some early results on SCA gliders, but there is much more to be done. One direction for future work would be to continue the classification of non-pure gliders. Another would be to examine the interactions of different gliders on the same turning rules and crossing rules. One could also try to enforce some kind of a group structure on certain subsets of periods of gliders on $n$-strands. Yet another direction for future work would be to develop the theory of gliders which contain crossings.
\section*{Acknowledgments}
Thanks to Dr.~Holden for advising this project and the Rose Research Fellows 2023-24 program for (partially) funding it.

\end{document}